\documentclass[reqno,10pt]{amsart}
\usepackage[all,line,poly,rotate,ps,dvips]{xy}
\SelectTips{cm}{}

\usepackage{comment}
\usepackage{amssymb,eucal,graphicx}
\usepackage{enumerate}

\numberwithin{equation}{section}

\newtheorem{theorem}{Theorem}[section]

\newtheorem{corollary}[theorem]{Corollary}
\newtheorem{claim}[theorem]{Claim}
\newtheorem{lemma}[theorem]{Lemma}
\newtheorem{proposition}[theorem]{Proposition}

\theoremstyle{definition}

\newtheorem{definition}[theorem]{Definition}
\newtheorem{remark}[theorem]{Remark}
\newtheorem{conjecture}[theorem]{Conjecture}

\newtheorem{example}[theorem]{Example}

\theoremstyle{remark}

\newcommand{\C}{\mathbb{C}}
\newcommand{\ZZ}{\mathbb{Z}}

\newcommand{\Pp}{\mathbb{P}}

\newcommand{\Ff}{\mathcal{F}}
\newcommand{\Ef}{\mathcal{E}}

\newcommand{\Efra}{\mathfrak{E}}

\newcommand{\zd}{\mathcal{Z}}
\newcommand{\ud}{\mathcal{U}}
\newcommand{\vdj}{\mathcal{V}^{\delta,j}_{d}}

\newcommand{\vdjt}{\mathcal{V}^{\delta,j,t}_{d}}

\newcommand{\Ii}{\mathcal{I}}
\newcommand{\Jj}{\mathcal{J}}

\newcommand{\N}{\mathcal{N}}

\newcommand{\Oc}{\mathcal{O}}

\DeclareMathOperator{\Pic}{{\rm Pic}}
\DeclareMathOperator{\Ext}{{\rm Ext}}
\DeclareMathOperator{\cork}{{\rm cork}}

\begin{document}

\title[Extensions of line bundles and Brill--Noether loci]
{Extensions of line bundles and Brill--Noether loci of rank-two vector bundles on a general curve}

\author{Ciro Ciliberto}
\curraddr{Dipartimento di Matematica, Universit\`a degli Studi di
Roma Tor Vergata\\ Via della Ricerca Scientifica - 00133 Roma
\\Italy} \email{cilibert@mat.uniroma2.it}

\author{Flaminio Flamini}
\curraddr{Dipartimento di Matematica, Universit\`a degli Studi di
Roma Tor Vergata\\ Via della Ricerca Scientifica - 00133 Roma
\\Italy} \email{flamini@mat.uniroma2.it}


\let\thefootnote\relax

\thanks{{\em Thanks} The authors wish to thank E. Ballico, for interesting discussions on \cite{Ballico1,Lau}, and P.\,E. Newstead, A. Castorena for having pointed out to us additional references we missed in the huge amount of papers on this topic. At last, warm thanks to the referee for careful reading and positive comments.}

\begin{abstract}  In this paper we study Brill-Noether loci for rank-two, (semi)stable vector bundles on a general curve $C$. Our aim is to describe  the general member $\Ff$ of some of its components just in terms of extensions of line bundles with suitable {\em minimality properties}, providing information both on families of irreducible, unisecant curves of the ruled surface $\Pp(\Ff)$ and on the birational geometry of the component of the Brill-Noether locus to which $\Ff$ belongs. 
\end{abstract}

\maketitle

\tableofcontents


\section*{Introduction}\label{sec:intro} Let $C$ be a smooth, irreducible projective curve of genus $g$ and $U_C(d)$ be the moduli space of (semi)stable, degree $d$, rank-two vector bundles on $C$. In this paper we will be mainly concerned with $C$ of general moduli. 

Our aim is to study the Brill-Noether loci $B_C^{k}(d) \subset U_C(d)$ parametrizing (classes of) vector bundles \linebreak $[\Ff] \in U_C(d)$ having  $h^0(C,\Ff) \geqslant k$, with $k$ a  non-negative integer.

The classical Brill-Noether theory for line bundles on a general curve is very important and 
well established (cf., e.g., \cite{ACGH}). Brill-Noether theory for higher-rank vector bundles is a very active research area (see References, for some results in the subject),  but several basic questions concerning Brill--Noether loci,  like non-emptiness, dimension, irreducibility, local structure, etc.,  are still open in general.  Contrary to the rank-one case,  the Brill-Noether loci for $C$ general do not always behave as expected (cf. e.g. \cite{BeFe} and \S\,\ref{SS:low}).

Apart from its intrinsic interest, Brill-Noether theory is important in view of applications to other areas, like birational geometry to mention just one (cf. e.g. \cite{Beau,Be,BeFe,HR,Muk}).

The most general existence result in the rank-two case is the following:  

\begin{theorem}\label{thm:TB} (see \cite{TB}) Let $C$ be a curve with general moduli of genus $g \geqslant1$. Let $k \geqslant 2$ and $i := k + 2g-2-d \geqslant 2$ be integers. Let $\rho_d^k := 4g-3 - i k$ and assume 
$$\rho_d^k \geqslant 1 \;\;  {\rm when} \; d \; {\rm odd}, \;\;\;  \rho_d^k \geqslant 5 \;\;  {\rm when} \; d \; {\rm even}.$$Then 
$B^{k}_C(d)$ is not empty and  it contains a  component $\mathcal B$ of the expected dimension $\rho_d^k$. 
\end{theorem}   

\noindent
This previous result is proved with a quite delicate degeneration argument (cf. also \cite{CMT}); in some particular cases, one has improvements of it (cf., e.g.,  \cite{Sun,Lau,TB0,TB00,Tan,FO,TB1,LNP}).

\bigskip

The degeneration technique used in \cite{TB}, though powerful, does not provide a \emph{geometric description} of the (isomorphism classes of) bundles $\Ff$ in $B^{k}_C(d)$, in particular of the general one in a component. By ``geometric description'' we mean a description of  families of curves on the ruled surface $\Pp(\Ff)$, in particular of unisecant curves to its fibres. This translates in turn to exhibiting $\Ff$ as an extension of line bundles $$(*)\;\;\; 0 \to N \to \Ff \to L \to 0$$(cf.\,\eqref{eq:Fund}), which we call a {\em presentation} of $\Ff$. Of particular interest is a presentation $(*)$ with suitable {\em minimality properties} on the quotient line bundle $L$, which translate into minimality properties for families of irreducible unisecant curves on the surface $\Pp(\Ff)$ (cf.\,\S\,\ref{S:VB} below).

This approach provides basic information about the vector bundle $\Ff$, which can be useful in a field in which so little is known and which has not been given so far: indeed, the description of such a minimal presentation is not known in general and, in particular, has not been provided in Theorem \ref{thm:TB}.  

One of the main objective of this paper is to shed some light on this subject. As a consequence of our analysis, we provide explicit parametric representations (and so information about the birational geometry) of some components of $B_C^{k}(d)$ (cf.\,\S's\,\ref{S:PSBN},\,\ref{S:BND}).

\bigskip

Viewing rank two vector bundles as extensions of line bundles is very classical: by suitably interpreting the classical language, this   goes back to C. Segre \cite {Seg}.   In  recent times, it  has been exploited, e.g., in \cite[\S\,2,3]{BeFe}, \cite[\S\,8]{Muk}, where the case of canonical determinant and  $g \leqslant 12$ has been treated. 
As noted in \cite{BeFe}, this approach ``{\em works well enough in low genera.... but seems difficult to implement in general }''. However, we tried to follow this route, with no upper-bounds on the genus but, as we will see, by bounding the speciality 
$i:=h^1(C,\Ff)$.  

Our approach is as follows.  We construct (semi)stable vector bundles $\Ff$ in Brill-Noether loci, as extensions of line bundles $L$ and $N$: the Brill--Noether loci we hit in this way depend on the cohomology of $L$ and $N$ and on the behaviour of the {\em coboundary map} $H^0(C, L) \stackrel{\partial}{\longrightarrow} H^1(C,N)$ associated to  $(*)$, cf.\,\S's\,\ref{S:BNLN},\,\ref{sec:constr}; we exhibit explicit constructions of such vector bundles in Theorems \ref{LN}, \ref{C.F.VdG}, \ref{uepi}, \ref{unepi}. 
These theorems provide existence results for $B^{k}_C(d)$ which are comparable to, though slightly worse but easier to prove than, Theorem \ref {thm:TB} (cf.\,Remark.\,\ref{rem:C.F.VdG}--(3)).  At the same time, they imply non--emptiness for fibres of the determinant map $B^{k}_C(d)\to {\rm Pic}^ d(C)$, i.e. for Brill--Noether loci with fixed determinant $\det(\Ff) := L\otimes N$, for any possible $L$ and $N$ as in the assumptions therein; this is in the same spirit of \cite {LaNw}, where however only the case with fixed determinant of odd degree has been considered. 

In any event, as we said, the main purpose of this paper is not the one of constructing new components of Brill-Noether loci, but of  providing a minimal presentation for the \emph{general element} of components of $B^{k}_C(d)$, for $C$ with general moduli. To do this, we take line bundles $L$ and $N$ with assumptions as in Theorems \ref{LN}--\ref{unepi} and let them vary in their own Brill-Noether loci of their Picard schemes. Accordingly, we let the constructed bundles $\Ff$ vary in suitable degeneracy loci $\Lambda \subseteq {\rm Ext}^1(L,N)$, defined in such a way that
$\Ff \in \Lambda$ general has the desired speciality $i$ (cf. \S\,\ref{S:PSBN}). In this way we obtain irreducible varieties  parametrizing triples $(L,N,\Ff)$, i.e. any such variety is endowed with a morphism $\pi$ to $U_C(d)$, whose image is contained in a component $\mathcal B$ of a Brill-Noether locus.  To find a  minimal presentation $(*)$ for a general member $\Ff$ of $\mathcal B$, we are reduced to find conditions on $L$, $N$ and on the coboundary map $\partial$ ensuring the morphism $\pi$ to be dominant onto $\mathcal B$. We achieve this goal by using results in \S\,\ref{S:VB}, which deal with the study of some families of irreducible unisecants of given speciality on the ruled surface $\Pp(\Ff)$ (cf.\,Lemmas \ref{lem:claim1},\,\ref{lem:claim2},\,Corollaries \ref{C.F.1b},\,\ref{C.F.1c},\,\ref{C.F.2},\,\ref{C.F.3b}\,and Remarks\,\ref{rem:17lug},\,\ref{rem:bohb}).

As is clear from the foregoing description, to find a presentation of $\Ff$ general in a component of a Brill-Noether locus is a difficult problem in general. We are able to solve it here for  $i\leqslant 3$.

Our main results for Brill--Noether loci $B_C^k(d)$ are Theorems \ref{i=1}, \ref{i=2}, \ref{i=3}, which respectively deal with cases $i=1,2,3$ and $k= d-2g+2+i$. A first fact we prove therein is that (a component of) $B_C^k(d)$ is filled-up by vector bundles $\Ff$ having a minimal {\em special} presentation as 
$$0 \to N \to \Ff \to \omega_C(- D_{i-1}) \to 0,$$where $D_{i-1} \in {\rm Sym}^{i-1}(C)$, $N \in {\rm Pic}^{d-2g+1-i}$ and  $\Ff \in \Lambda_{i-1}$ are general, where $\Lambda_{i-1}$ is a {\em good component} of the degeneracy locus 
$$\left\{\Ff \in {\rm Ext}^1(\omega_C(- D_{i-1}), N) \,|\, \dim {\rm Coker} \left(H^0(C,L) \stackrel{\partial}{\longrightarrow} H^1(C,N) \right) \geqslant i-1 \right\} \subseteq {\rm Ext}^1(\omega_C(- D_{i-1}), N)$$(cf.\,Def.\,\ref{def:ass1}, for precise definitions of special presentation and minimality, Rem.\,\ref{rem:wt} and Defs.\,\ref{def:goodc}, \ref{def:goodtot}, for goodness, and Thms.\,\ref{thm:mainext1},\,\ref{thm:mainextt}, for existence of good components). The case $i=1$ was already treated in \cite {Ballico1} with different methods (cf. Remark \ref{rem:i=1}); in cases $i=2,3$ our results are new. 

Statements of our main results, to which the reader is referred, contain even more. Indeed they also describe families of special, irreducible unisecants of $\Pp(\Ff)$ which are of minimal degree with respect to its tautological line bundle. Apart from its intrinsic interest, this description plays a fundamental role when one tries to construct components of the Hilbert scheme parametrizing linearly normal, genus $g$ and degree $d$ special scrolls in projective spaces and  whose general point parametrizes a {\em stable scroll} (cf. e.g.\,\cite{CF}). 

Finally, the proofs of Theorems \ref{i=1}, \ref{i=2}, \ref{i=3} show in particular that the map $\pi$ from the parameter space of triples $(L, N, \Ff)$ to (the dominated component of) $B_C^k(d)$ is generically finite, sometimes even birational (cf. Rmk.\,\ref{rem:i=1}--(3)), giving therefore information about the birational geometry of the Brill--Noether locus.

\bigskip 

Other main results of the paper are given by Theorems \ref{prop:M12K}, \ref{prop:M22K}, \ref{prop:M32K} which deal with the canonical determinant case.

\bigskip

In principle, there is no obstruction in pushing further the ideas in this paper, to treat higher speciality cases. However, this is increasingly complicated and therefore we limited ourselves to expose at the end of the paper a few suggestions on how to proceed in general and  propose a conjecture (see \S\,\ref{sec:high}). 

\bigskip

The paper is organized as follows. Section \ref{S:VB} is  devoted to preliminaries about families of special, irreducible unisecants on ruled surfaces $\Pp(\Ff)$ and the corresponding special presentation of the bundle $\Ff$. Sections \ref{S:MS} and \ref{S:BNLN} are devoted to recalling basic facts on (semi)stable, rank--two vector bundles of degree $d$, extensions of line bundles, and useful results of Lange--Narashiman and Maruyama (cf. Proposition \ref{prop:LN} and Lemma \ref{lem:technical}). Section \ref{sec:constr} is the technical one, which contains our  constructions of vector bundles in Brill-Noether loci as extensions of line bundles $L$ and $N$.  
Section \ref{S:PSBN} is where we deal with parameter spaces of triples and maps from them to  $U_C(d)$, landing in Brill-Noether loci. The general  machinery developed in the previous sections is then used in \S\,\ref{S:BND}, in order to prove our main results mentioned above.


\section{Notation and terminology}\label{sec:P}

In this paper we work over $\mathbb C$. All schemes will be endowed with the Zariski topology. We will  interchangeably use the terms rank-$r$ vector bundle on a scheme $X$ and rank-$r$ locally free sheaf. 

We denote by $\sim$ the linear equivalence of divisors, by $\sim_{alg}$ their algebraic equivalence and by $\equiv$ their numerical equivalence. We may abuse notation and identify divisor classes with the corresponding line bundles, interchangeably using additive and multiplicative notation.  
 
If $\mathcal P$ is the parameter space of a flat family of subschemes of $X$ and if $Y$ is an element of the family, we denote by 
$Y \in \mathcal P$ the point corresponding to $Y$. If $\mathcal M$ is a moduli space, parametrizing geometric objects modulo a given equivalence relation, we denote by $[Z] \in \mathcal M$ the moduli point corresponding to the equivalence class of $Z$.

Let 

\noindent
$\bullet$ $C$ be a smooth, irreducible, projective curve of genus $g$, and

\noindent
$\bullet$ $\Ff$ be a rank-two vector bundle on $C$.

Then,  $F:= \Pp(\Ff) \stackrel{\rho}{\to} C$ will denote the {\em (geometrically) ruled
surface} (or {\em the scroll}) associated to  $(\Ff, C)$; $f$ will denote the general $\rho$-fibre and 
$\Oc_F(1)$ the {\em tautological line bundle}. A divisor in $|\Oc_F(1)|$ will be usually denoted by $H$. 
If $\widetilde{\Gamma}$ is a divisor on $F$, we will set ${\rm deg}(\widetilde{\Gamma}) := \widetilde{\Gamma} H$.


We will use the notation $$d:=  \deg(\Ff) = \deg(\det(\Ff)) = H^2 = \deg(H);$$$i(\Ff):= h^1(\Ff)$ is called the {\em speciality} of $\Ff$ and will be denoted by $i$, if there is no danger of confusion. $\Ff$ (and $F$) is {\em non-special}
if $i= 0$,  {\em special} otherwise.

As customary,  $W^r_a(C)$ will denote the  {\em Brill-Noether locus},  parametrizing  line bundles $A \in Pic^a(C)$ such that 
$h^0(A) \geqslant r+1$,  
$$\rho(g, r, a):= g - (r+1) ( r+g-a)$$the {\em Brill-Noether number} and 
\begin{equation}\label{eq:Petrilb}
\mu_0(A) : H^0(C,A) \otimes H^0(\omega_C \otimes A^{\vee}) \to H^0(C, \omega_C)
\end{equation} the {\em Petri map}. As for the rest, we will use standard terminology and notation as in e.g. \cite{ACGH}, \cite{Ha}, etc. 


\section{Scrolls unisecants}\label{S:VB}

We recall some basic facts on unisecant curves 
of the scroll $F$ (cf.  \cite{GP2, Ghio} and \cite[V-2]{Ha}). 

One has ${\rm Pic}(F) \cong \ZZ[\Oc_F(1)] \oplus \rho^*({\rm Pic}(C))$ 
(cf.\ \cite[\S\,5, Prop.\,2.3]{Ha}). Let ${\rm Div}_F$ be the scheme (not of finite type) of effective divisors on $F$, which is a sub-monoid of ${\rm Div}(F)$. For any 
$k \in {\mathbb N}$, let ${\rm Div}^k_F$ be the subscheme (not of finite type) of ${\rm Div}_F$ formed by all divisors 
$\widetilde{\Gamma}$ such that $\Oc_F(\widetilde{\Gamma}) \cong \Oc_F(k) \otimes \rho^*(N^{\vee})$, for some $N \in {\rm Pic} (C)$ (this $N$ is uniquely determined); then one has a natural morphism 
$$\Psi_k: {\rm Div}^k_F \to {\rm Pic}(C), \;\;\;\;\;\;\;\; \widetilde{\Gamma} \stackrel{\Psi_k}{\longrightarrow} N.$$

If $D \in {\rm Div} (C)$, then $\rho^*(D)$ will be denoted by $f_D$. Then $\widetilde{\Gamma} \in {\rm Div}^k_F$ 
if and only if $\widetilde{\Gamma} \sim kH - f_D$, for some $D \in {\rm Div}(C)$, and $\deg(\widetilde{\Gamma}) 
= k \deg(\Ff) - \deg(D)$.

The curves in ${\rm Div}^1_F$ are called {\em unisecants}  of $F$. Irreducible unisecants are isomorphic to $C$ and called  {\em sections} of $F$. For any positive integer $\delta$, we consider (cf. \cite[\S\;5]{Ghio}) 
$${\rm Div}_F^{1,\delta} : = \{\widetilde{\Gamma} \in {\rm Div}_F^1 \; | \; {\rm deg}(\widetilde{\Gamma}) = \delta \},$$which is the {\em Hilbert scheme of unisecants of degree $\delta$ of $F$} (w.r.t. $H$).

\begin{remark}\label{rem:Hilbert} Let 
$\widetilde{\Gamma} = \Gamma + f_A$ be a unisecant, with  $\Gamma$ a section and $A $ effective. Equivalently, 
we have an exact sequence 
\begin{equation}\label{eq:Fund2}
0 \to N ( - A) \to \Ff \to L \oplus \Oc_{A} \to 0
\end{equation} (cf.\;\cite{CCFMnonsp, CCFMBN}); in particular if $A =0$, i.e. 
$\widetilde{\Gamma} = \Gamma$ is a section, $\Ff$ fits in the exact sequence 
\begin{equation}\label{eq:Fund}
0 \to N \to \Ff \to L \to 0
\end{equation} and
\begin{equation}\label{eq:Ciro410}
\N_{\Gamma/F} \cong L \otimes N^{\vee}, \;\; {\rm so} \;\; \Gamma^2 = \deg(L) - \deg(N) = 2\delta - d,  
\end{equation} (cf. \;\cite[\S\;V,\;Prop. 2.6, 2.9]{Ha}). Accordingly  $\Psi_{1,\delta}: {\rm Div}^{1,\delta}_F \to {\rm Pic}^{d-\delta}(C)$, the restriction of $\Psi_1$, 
endows ${\rm Div}_F^{1,\delta}$ with a  structure of Quot scheme: with notation as in \cite[\S\,4.4]{Ser}, one has
\begin{equation}\label{eq:isom1}
\begin{array}{rrcc}
\Phi_{1,\delta}: & {\rm Div}_F^{1,\delta} & \stackrel{\cong}{\longrightarrow} & {\rm Quot}^C_{\Ff,\delta+t-g+1} \\
 & \widetilde{\Gamma} & \longrightarrow & \left\{\Ff \to\!\!\to L \oplus \Oc_A \right\}.  
\end{array}
\end{equation} From standard results (cf. e.g. \cite[\S\,4.4]{Ser}), \eqref{eq:isom1} gives 
identifications between tangent and obstruction spaces:
\begin{equation}\label{eq:tang}
H^0(\N_{\widetilde{\Gamma}/F}) \cong T_{[\widetilde{\Gamma}]} ({\rm Div}_F^{1,\delta}) \cong {\rm Hom} (N (-A), L \oplus \Oc_A) \;\;\; {\rm and} \; \;\; H^1(\N_{\widetilde{\Gamma}/F})  \cong {\rm Ext}^1 (N (-A), L \oplus \Oc_A) 
\end{equation} Finally, if $\widetilde{\Gamma} \sim H - f_D$, then one easily checks that 
\begin{equation}\label{eq:isom2}
|\Oc_F(\widetilde{\Gamma})| \cong \Pp(H^0(\Ff (-D))).  
\end{equation}

\end{remark}

\begin{definition}\label{def:ass0} $\widetilde{\Gamma} \in {\rm Div}^{1,\delta}_F$ is said to be:  

\noindent
(a) {\em linearly isolated (li)}  if $\dim(|\Oc_F(\widetilde{\Gamma} )|) =0$, 

\noindent
(b) {\em algebraically isolated (ai)}  if  $\dim({\rm Div}^{1,\delta}_F) =0$.

\end{definition}

\begin{remark}\label{rem:linisol} (1) If $\widetilde{\Gamma}$ is ai, then it is also li but the converse is false (c.f. e.g. Example \ref{ex:contro}, Corollary \ref{C.F.3b}).

\noindent 
(2) When ${\rm Div}^{1,\delta}_F$ is of pure dimension,  
a sufficient condition for $\widetilde{\Gamma}$ to be  ai is $h^0(\N_{\widetilde{\Gamma}/F}) = 0$ (cf. e.g. Theorem \ref{C.F.VdG}, Corollary \ref{C.F.1c} and \S\,\ref{i=1} below). 
\end{remark}



\subsection{The Segre-invariant}\label{ss:S} 

\begin{definition}\label{def:seginv}
The {\em Segre invariant} of $\Ff$ is defined as $$s(\Ff) := \deg(\Ff) - 2 ({\rm max} \; \{\deg (N) \}),$$where the maximum is taken among all sub-line bundles $N$ of $\Ff$ (cf.\ e.g.\ \cite{LN}). The bundle $\Ff$ is {\em stable} [resp. {\em semi--stable}], if $s(\Ff)>0$ [resp. if $s(\Ff)\geqslant 0$]. 

Equivalently $\Ff$ is stable [resp. semistable] if for every sub-line bundle $N\subset \Ff$ one has $\mu(N)<\mu(\Ff)$ [resp. $\mu(N)\le\mu(\Ff)$], where
 $\mu(\mathcal E)=\deg (\mathcal E)/{\rm rk}(\mathcal E)$ is the \emph{slope}  of a vector bundle $\mathcal E$.

\end{definition}Note that, for any $A \in {\rm Pic}(C)$, one has
\begin{equation}\label{eq:seginv}
s(\Ff) = s(\Ff \otimes A). 
\end{equation}

\begin{remark}\label{rem:seginv}
From \eqref{eq:Ciro410}, $s(\Ff)$ coincides with the minimum 
self-intersection  of sections of  $F$. In particular, 
if $\Gamma \in {\rm Div}^{1,\delta}_F$ is a section s.t. $\Gamma^2 = s(\Ff)$, then 
$s(\Ff) = 2 \delta - d$ and $\Gamma$ is a section of {\em minimal degree} of $F$, i.e. for any 
section $\Gamma' \subset F$ one has $\deg(\Gamma') \geqslant \deg(\Gamma)$. 
\end{remark}

We recall the following fundamental result. 

\begin{proposition}\label{prop:Nagata} Let $C$ be of genus $g \geqslant 1$ and let $\Ff$ be indecomposable. Then, $2-2g \leqslant s(\Ff) \leqslant g$. 
\end{proposition}

\begin{proof} The lower-bound  follows from $\Ff$ being indecomposable (see e.g. \cite[V,\;Thm.\;2.12(b)]{Ha}). The upper-bound is Nagata's Theorem (see \cite{Na}). 
\end{proof}



\subsection{Special scrolls unisecants}\label{ss:SU}

In the paper we will be mainly concerned about the speciality of unisecants of a (necessarily special) scroll $F$.

\begin{definition}\label{def:lndirec} For $\widetilde{\Gamma} \in {\rm Div}_F$, we set $\Oc_{\widetilde{\Gamma}} (1) := \Oc_F(1) \otimes \Oc_{\widetilde{\Gamma}}$. The {\em speciality} 
of $\widetilde{\Gamma}$ is $i(\widetilde{\Gamma}) := h^1(\Oc_{\widetilde{\Gamma}} (1))$. $\widetilde{\Gamma}$ is {\em special} if $i(\widetilde{\Gamma}) >0$. 
\end{definition}

If $\widetilde{\Gamma}$ is given by \eqref{eq:Fund2}, then by \eqref{eq:isom2} one has 
$\widetilde{\Gamma} \in |\Oc_F(1) \otimes \rho^*(N^{\vee}(A))|$. Applying $\rho_*$ to the exact sequence  
$$0 \to \Oc_F(1) \otimes \Oc_F(-\widetilde{\Gamma})\to \Oc_F(1) \to \Oc_{\widetilde{\Gamma}}(1) \to 0$$and using 
$\rho_*( \Oc_F(1) \otimes \Oc_F(-\widetilde{\Gamma})) \cong N ( - A)$, 
$ R^1\rho_*(\Oc_F(\rho^*(N ( -A))) = 0$, we get 
\begin{equation}\label{eq:iLa}
i(\widetilde{\Gamma}) = h^1(L \oplus \Oc_{A}) = h^1(L) = i(\Gamma), 
\end{equation}where $\Gamma$ the unique section in  $\widetilde{\Gamma}$.

The following examples show that, in general, speciality is not constant either in linear systems or in  algebraic families. 

\begin{example}\label{ex:1} Take $g=3$, $i=1$ and $d=9= 4g-3$. There are smooth, linearly normal, special scrolls $S \subset \Pp^5$ of degree $9$, speciality $1$, sectional genus $3$ with general moduli containing a unique 
special section $\Gamma$ which is a genus $3$  canonical curve (cf. \cite[Thm.\;6.1]{CCFMsp}). Moreover,  $\Gamma$ is the unique section of minimal degree $4$ (cf. also \cite{Seg}). There are lines $f_1, \ldots, f_5$ of the ruling, such that $\widetilde{\Gamma} := \Gamma + f_1 + \ldots + f_5 \in |H|$, where $H$ the hyperplane section of $S$. These curves $\widetilde{\Gamma}$ vary in a sub-linear system of dimension $2$ contained in $|H|$, whose movable part is the complete linear system $|f_1 + \cdots + f_5|$. The curves as $\widetilde{\Gamma}$ are the only special unisecants in $|H|$. 
\end{example}

\begin{example}\label{ex:contro} Let $C$ be a non-hyperelliptic curve of genus $g \geqslant 3$,  
$d = 3g-4$ and $N \in {\rm Pic}^{g-2}(C)$ general. $N$ is non-effective with $h^1(N) = 1$. Consider ${\rm Ext}^1(\omega_C, N)$. It has dimension $2g-1$ and its general point  
gives rise to a rank-two vector bundle $\Ff$ of degree $d$, fitting in an exact sequence  like \eqref{eq:Fund}, with $L = \omega_C$. By generality, 
the coboundary map $\partial : H^0(\omega_C) \to H^1(N) \cong \C$ is surjective (cf. Corollary \ref{cor:mainext1} below); therefore 
$i(\Ff_u)=1$. Since $\Ff$ is of rank-two with $\det(\Ff) = \omega_C \otimes N$, by Riemann-Roch one has $h^0(\Ff \otimes N^{\vee})  =1$. From \eqref{eq:isom2}, the canonical section 
$\Gamma \subset F$, corresponding to $\Ff \to\!\! \to \omega_C$, is li. From \eqref{eq:Ciro410}, $\N_{\Gamma/F} \cong \omega_C \otimes N^{\vee}$ hence $h^i(\N_{\Gamma/F}) = 1-i$, for $i = 0, 1$. Let $\mathcal D$ be the irreducible, one-dimensional component of the Hilbert scheme containing the point corresponding to $\Gamma$ (which is smooth for the Hilbert scheme). Therefore 
$\mathcal D$ is an algebraic (non-linear) family whose general member is a li section. As a consequence of Proposition \ref{prop:lem4} below,  $\Gamma$ is the only special section in $\mathcal D$. In particular, if all curves in $\mathcal D$ are irreducible, then $\Gamma$ is the only special curve in $\mathcal D$ (see Lemma \ref{lem:ovviolin}). 

Note that $\Ff$ is indecomposable. Indeed, assume $\Ff = A \oplus B$, with $A$, $B$ line bundles. Since $h^0(\Ff \otimes N^{\vee}) =1$,  we may assume $h^0(A-N) = 1$ and $h^0(B-N) = 0$. By the genericity of $N$, $A-N$ and $B-N$ are both general of their degrees. Therefore $\deg(A-N) = g$, hence $\deg(A) = 2g-2$ and $\deg(B) = g-2$. The image of $A$ in the surjection $\Ff \to \!\!\to \omega_C$ is zero, otherwise $A=\omega_C$ hence $B =N$ which is impossible, because $h^0(B-N) = 0$. Then we would have an injection $A \hookrightarrow N$ which is impossible by degree reasons. 
\end{example}

Since ${\rm Div}^{1,\delta}_F$ is a Quot-scheme, there is the universal quotient $\mathcal Q_{1,\delta} \to {\rm Div}^{1,\delta}_F$. Taking 
${\rm Proj} (\mathcal Q_{1,\delta}) \stackrel{p}{\to} {\rm Div}^{1,\delta}_F$, we can consider 

\begin{equation}\label{eq:aga}
\mathcal S_F^{1,\delta} := \{\widetilde{\Gamma} \in {\rm Div}^{1,\delta}_F \;\; | \;\; R^1p_*(\Oc_{\Pp(\mathcal Q_{1,\delta})}(1))_{\widetilde{\Gamma}} \neq 0\} \; \;\;  {\rm and} \; \;\; a_F(\delta) := \dim (\mathcal S_F^{1,\delta}), 
\end{equation} i.e.  $\mathcal S_F^{1,\delta}$ is the support of $R^1p_*(\Oc_{\Pp(\mathcal Q_{1,\delta})}(1))$. It parametrizes degree $\delta$, special unisecants of $F$.

\begin{definition}\label{def:ass1} Let $\widetilde{\Gamma}$ be a special unisecant of $F$. Assume $\widetilde{\Gamma} \in \mathfrak{F}$, where $\mathfrak{F} \subseteq {\rm Div}^{1,\delta}_F$ is a subscheme. 


\noindent
$\bullet$ We will say that $\widetilde{\Gamma}$ is: 

\noindent
(i) {\em specially unique (su)} in $\mathfrak{F}$, if $\widetilde{\Gamma}$ is  the only special unisecant in $\mathfrak{F}$, or   

\noindent
(ii) {\em specially isolated (si)} in $\mathfrak{F}$, if $\dim_{\widetilde{\Gamma}} \left(\mathcal S_F^{1,\delta} \cap \mathfrak{F} \right) = 0$.


\noindent
$\bullet$ In particular:

\noindent
(a) when $\mathfrak{F} = |\Oc_F(\widetilde{\Gamma})|$,  $\widetilde{\Gamma}$ is said to be {\em linearly specially unique (lsu)} in case (i) 
and  {\em linearly specially isolated (lsi)} in case (ii);

\noindent 
(b) when $ \mathfrak{F} =  {\rm Div}^{1,\delta}_F$,  $\widetilde{\Gamma}$ is said to be {\em algebraically specially unique (asu)} in case (i) and 
{\em algebraically specially isolated (asi)} in case (ii).  


\noindent
$\bullet$ When a section $\Gamma \subset F$ is asi, we will say that $\Ff$ is {\em rigidly specially presented (rsp)} as $\Ff \to \!\! \to L$ or by the sequence \eqref{eq:Fund} 
corresponding to $\Gamma$. When $\Gamma$ is ai (cf. Def. \ref{def:ass0}), we will say that 
$\Ff$ is {\em rigidly presented (rp)} via  $ \Ff \to \!\! \to L$ or \eqref{eq:Fund}.  
\end{definition}

For examples, c.f. e.g. \S\,\ref{S:BND} below. 

\begin{lemma}\label{lem:ovviolin} Let $\Gamma \subset F$ be a section corresponding to a sequence as in \eqref{eq:Fund}. A section $\Gamma'$,  corresponding to $\Ff \to\!\!\!\!\! \to L'$, is s.t.  $\Gamma \sim \Gamma'$ if and only if 
$L \cong L'$. In particular  

\noindent
(a) $i(\Gamma) = i(\Gamma')$;  

\noindent
(b) $\Gamma$ is lsu if and only if it is lsi if and only if it is li.  
\end{lemma}

\begin{proof} The first assertion follows from  \eqref{eq:isom2}.  Then, 
(a) and (b) are both clear. 
\end{proof}

\begin{proposition}\label{prop:lem4} Let $F$ be indecomposable and let $\Gamma \in \mathfrak{F} \subseteq \mathcal{S}^{1,\delta}_F$ be a section, where  $\mathfrak{F}$ is an irreducible, projective scheme of dimension $k$. Assume:  

\noindent
(a) $k \geqslant 1$, if $\mathfrak{F}$ is a linear system; 

\noindent
(b) either $k \geqslant 2$, or  $k=1$ and $\mathfrak{F}$ with base points, if $\mathfrak{F}$ is not linear.

Then, $\mathfrak{F}$ contains reducible unisecants $ \widetilde{\Gamma} $ with
\begin{equation}\label{eq:speciality}
i(\widetilde{\Gamma}) \geqslant i(\Gamma).
\end{equation}
\end{proposition}

\begin{proof} If $k \geqslant 2$, let $t$ be the unique integer such that $0 \leqslant k':= k-2t \leqslant 1$. Let $f_1, \ldots, f_t$ be $t$ general $\rho$-fibres of $F$. Since $k' \geqslant 0$, by imposing to the curves in $\mathfrak{F}$ to contain fixed general pairs of points on $f_1, \ldots, f_t$,  we see that 
$$\mathfrak{F}' := \mathfrak{F} \left(- \sum_{i=1}^t f_i\right) \subset \mathfrak{F}$$is non-empty, all components of it have dimension $k'$, and they all parametrize  unisecants 
$\Gamma'  \sim_{alg} \Gamma -  \sum_{i=1}^t f_i$. Then $\mathfrak{F}$ contains reducible elements $\widetilde{\Gamma}$, and they verify \eqref{eq:speciality} by upper-semicontinuity. This proves the assertion when $k \geqslant 2$.  

So we are left with the case $k=1$. Assume first that $\mathfrak{F}$ is a linear pencil. Since $\mathfrak{F} \subseteq |\Oc_F(\Gamma)|$, from the exact sequence $0 \to \Oc_F \to \Oc_F(\Gamma) \to \Oc_{\Gamma}(\Gamma) \to 0$, the line bundle  $\Oc_{\Gamma}(\Gamma)$ is effective so $\Gamma^2 \geqslant 0$. 
Let ${\rm Bs}(\mathfrak{F})$ be the base locus of $\mathfrak{F}$. If  $\Gamma^2 >0$, take $p \in {\rm Bs}(\mathfrak{F})$. We can clearly split off the fibre through $p$ with one condition, thus proving the result. 

If $\Gamma^2 = 0$, $\mathfrak{F}$ is a base-point-free pencil. So $F$ contains two disjoint sections and this implies that $\Ff$ is decomposable, a contradiction. 
 
Finally, if $\mathfrak{F}$  is non-linear, then ${\rm Bs} (\mathfrak{F}) \neq \emptyset$ and we can argue as in the linear case with $\Gamma^2 >0$. 
\end{proof}


\section{Brill-Noether loci} \label{S:MS} 

As usual,  $U_C(d)$ denotes the moduli space of (semi)stable, degree $d$, rank-two vector bundles on  $C$. The subset $U_C^s(d)\subseteq U_C(d)$ parametrizing (isomorphism classes of) stable bundles, is an open subset. The points in  ${U_C}^{ss}(d):=U_C(d)\setminus U_C^s(d)$ correspond to (S-equivalence classes of) \emph{strictly semistable} bundles (cf. e.g. \cite{Ram,Ses}). 

\begin{proposition}\label{prop:sstabh1}
Let $C$ be a smooth curve of  genus $g \geqslant 1$ and let $d$ be an integer.

\noindent
(i) If $d \geqslant 4g-3$, then for any $[\Ff] \in U_C(d)$, one has $i(\Ff) = 0$.

\noindent 
(ii) If $g \geqslant 2$ and $d \geqslant 2g-2$, for $[\Ff] \in U_C(d)$ general,  one has $i(\Ff) = 0$.  

\end{proposition}
\begin{proof} For (i), see  \cite[Lemma 5.2]{New}; for (ii) see \cite[p. 100]{Lau} or  \cite[Rem. 3]{Ballico}. 
\end{proof}

Thus, from Proposition \ref{prop:sstabh1}, Serre duality and invariance of stability under operations like tensoring with a line bundle or passing to the dual bundle, for $g \geqslant 2$ it makes sense to consider the proper sub-loci of 
$U_C(d)$ parametrizing classes  $[\Ff]$ such that  $i(\Ff) > 0$ for   
\begin{equation}\label{eq:congd}
2g-2 \leqslant d \leqslant 4g-4. 
\end{equation}

\begin{definition}\label{def:BNloci} Given non-negative integers $d$, $g$ and $ i $, we set
\begin{equation}\label{eq:ki}
k_i = d - 2 g + 2 + i.
\end{equation}Given a curve $C$ of genus $g$, we define 
$$B^{k_i}_C(d) := \left\{ [\Ff] \in U_C(d) \, | \,  h^0(\Ff) \geqslant k_i \right\} = 
\left\{ [\Ff] \in U_C(d) \, | \, h^1(\Ff) \geqslant i \right\}$$which we call the $k_i^{th}$--{\em Brill-Noether locus}.
\end{definition}

\begin{remark}\label{rem:BNloci}  The Brill-Noether loci $B^{k_i}_C(d)$ have a natural structure of closed subschemes of 
$U_C(d)$: 


\noindent
$(a)$ When $d$ is odd, $U_C(d)=U^s_C(d)$,   then $U_C(d)$  is a fine moduli space and the existence of a universal bundle on $C \times U_C(d)$  allows one to construct $B^{k_i}_C(d)$ as the degeneracy locus of a morphism between suitable  vector bundles on $U_C(d)$ (see, e.g. \cite{GT, M}). Accordingly, the  {\em expected dimension} of $B^{k_i}_C(d)$ is 
${\rm max} \{ -1, \ \rho_d^{k_i}\}$, where 
\begin{equation}\label{eq:bn}
\rho_d^{k_i}:= 4g - 3 - i k_i
\end{equation} is  the \emph{Brill-Noether number}.  If  $\emptyset \neq B^{k_i}_C(d) \neq U_C(d)$, 
then $B^{k_i+1}_C(d) \subseteq {\rm Sing}(B^{k_i}_C(d))$. Since any $[\Ff] \in B^{k_i}_C(d)$ is stable, it is a smooth point of $U_C(d)$ and 
$T_{[\Ff]}(U_C(d))$ can be identified with $H^0( \omega_C \otimes \Ff \otimes \Ff^{\vee})^{\vee}$. If $[\Ff] \in B^{k_i}_C(d) \setminus B^{k_i+1}_C(d)$, 
 the tangent space to $B^{k_i}_C(d)$ at $[\Ff]$ is the  annihilator of the image of the cup--product, {\em Petri map} of $\Ff$ (see, e.g. \cite{TB00}) 
\begin{equation}\label{eq:petrimap}
P_{\Ff} : H^0(C, \Ff) \otimes H^0(C, \omega_C \otimes \Ff^{\vee})
\longrightarrow H^0(C, \omega_C \otimes \Ff \otimes \Ff^{\vee}).
\end{equation} 
If $[\Ff] \in B^{k_i}_C(d) \setminus B^{k_i+1}_C(d)$, then 
$$\rho_d^{k_i} = h^1(C, \Ff \otimes \Ff^{\vee}) - h^0(C, \Ff) h^1(C, \Ff)$$and 
$B^{k_i}_C(d)$ is non--singular, of the expected dimension at $[\Ff]$ if and only if $P_{\Ff}$ is injective.


\noindent
$(b)$ When $d$ is even, $U_C(d)$ is not a fine moduli space (because $U^{ss}_C(d)\neq \emptyset$). There is a suitable open, non-empty subscheme 
$ \mathcal Q^{ss} \subset \mathcal Q$ of a certain Quot scheme $\mathcal Q$ defining $U_C(d)$ via the GIT-quotient sequence
$$0 \to PGL(q) \to \mathcal Q^{ss} \stackrel{\pi}{\longrightarrow} U_C(d) \to 0$$(cf. e.g. \cite{Tha} for details); 
one can define $B^{k_i}_C(d)$ as the image via $\pi$ of the degeneracy locus of a morphism between suitable vector bundles on $\mathcal Q^{ss}$. 
The fibres of $\pi$ over strictly semistable bundle classes are not isomorphic to $PGL(q)$. It may happen for a component   $\mathcal B $ of a Brill--Noether locus to be  totally contained in $U_C^{ss}(d)$; 
in this case the lower bound  $\rho_d^{k_i}$ for the expected dimension of $\mathcal B$ is no longer valid  (cf. Corollary \ref{C.F.2} below and \cite[Remark 7.4]{BGN}).  
The lower bound $ \rho_d^{k_i}$ is still valid if $\mathcal B\cap U^s_C(d)\neq \emptyset$.  
\end{remark}

\begin{definition}\label{def:regsup} Assume $B_C^{k_i}(d) \neq \emptyset$. A component $ \mathcal B \subseteq B_C^{k_i}(d)$ 
such that $\mathcal B \cap U_C^s(d) \neq \emptyset$ will be called {\em regular},  if $\dim(\mathcal B) = \rho_d^{k_i}$, 
{\em superabundant}, if  $\dim(\mathcal B) > \rho_d^{k_i}$. 
\end{definition}


\section{(Semi)stable vector bundles and extensions}\label{S:BNLN}

In this section we discuss how to easily produce special, (semi)stable vector bundles $\Ff$ as extensions of  line bundles $L$ and $N$ as in \eqref{eq:Fund}. This is the same as considering vector bundles $\Ff$, with a sub-line bundle $N$ s.t. 
$\Ff \otimes N^{\vee}$ has a nowhere vanishing section.

If $g =2$, in the range  \eqref{eq:congd} one has  bundles  $\Ff$ with slope $1 \leqslant \mu(\Ff) \leqslant 2$ on a hyperelliptic curve, which  have been studied in \cite{BGN, BMNO, M1, M4}. 
Thus, we will assume $C$ non-hyperelliptic of genus $g \geqslant 3$, with $d$ as in \eqref{eq:congd}.



\subsection{Extensions and a result of Lange-Narashiman}\label{ss:LN} Let $\delta \leqslant d$ be a positive integer. Consider $L \in {\rm Pic}^{\delta}(C)$ and $N \in {\rm Pic}^{d-\delta}(C)$; 
${\rm Ext}^1(L,N)$ parametrizes (strong) isomorphism classes of extensions (cf. \cite[p. 31]{Frie}). Any $u \in {\rm Ext}^1(L,N)$ gives rise to a degree $d$, rank-two vector bundle $\Ff=\Ff_u$  as in \eqref{eq:Fund}. 
In order to get $\Ff_u$ (semi)stable, a necessary condition is (cf. Remark \ref{rem:seginv}) 
\begin{equation}\label{eq:assumptions}
2\delta-d \geqslant 0. 
\end{equation}Therefore, by Riemann-Roch theorem, we have
\begin{equation}\label{eq:lem3note}
m:= \dim({\rm Ext}^1(L,N)) = \left\{\begin{array}{ll}
2\delta - d + g - 1 & {\rm if} \; L \; |\!\!\!\!\!  \cong N \\
g & {\rm if} \; L \cong N.
\end{array}
\right.
\end{equation}

\begin{lemma}\label{lem:1e2note} Let $\Ff$ be a (semi)stable, special, rank-two vector bundle on $C$ of degree $d \geqslant 2g-2$. Then $\Ff = \Ff_u$, for a special, effective line bundle $L$  on $C$, $N = \det(\Ff) \otimes L^{\vee}$ as in \eqref{eq:Fund} and 
$u \in {\rm Ext}^1(L,N)$.
\end{lemma}
\begin{proof}  By Serre duality, $i(\Ff) >0 $ gives a non-zero morphism $\Ff \stackrel{\sigma^{\vee}}{\to} \omega_C$. The line bundle $L:= {\rm Im}(\sigma^{\vee}) \subseteq \omega_C$ is special. Set $\delta= \deg(L)$. Since $\Ff$ is (semi)stable, then  \eqref{eq:assumptions} holds hence  
$\delta \geqslant \frac{d}{2} \geqslant g-1$, therefore $\chi(L) \geqslant 0$, so  $L$ is effective.
\end{proof}

\begin{remark}\label{rem:rigid} In the setting of Lemma \ref{lem:1e2note}, consider the scroll  $F= \Pp(\Ff)$ and let $\Gamma \subset F$ 
be the section corresponding to $L$, with $L \in {\rm Pic}^{\delta}(C)$ a special, effective quotient of minimal degree of $\Ff$. 
Suppose $\Ff$ indecomposable. From Proposition \ref{prop:lem4}, one has
\begin{equation}\label{eq:aiutoE}
\Gamma  \subset F \; \; {\rm is \; li\; with} \; \; a_{F}(\delta) \leqslant 1, 
\end{equation}
where $a_{F}(\delta)$ as in \eqref{eq:aga}. Then $\Ff$ is rsp via $L$ if $a_F(\delta)=0$, 
and even rp if $\Gamma$ is ai. 
\end{remark}
Fix  $L$ special, effective of degree $\delta$ and $N$ of degree $d-\delta$, where $d$ satisfies \eqref{eq:congd} and \eqref{eq:assumptions} (so $\deg(L) \geqslant \deg(N) \geqslant 0$). 
We fix once and for all the following notation:
\begin{equation}\label{eq:exthyp1}
\begin{array}{lcl}
j:= h^1(L) > 0, & &  \ell := h^0(L) = \delta-g+1 +j >0,\\
r := h^1(N) \geqslant 0,& &  n := h^0(N)= d- \delta-g+1 +r\geqslant 0.
\end{array}
\end{equation}
Any $u \in {\rm Ext^1}(L,N)$ gives rise to the following diagram
\begin{equation}\label{eq:exthyp}
\begin{array}{crcccccl}
(u): & 0 \to & N & \to & \Ff_u & \to & L & \to 0 \\
{\rm deg} & & d-\delta & & d & & \delta & \\
h^0 & &   n & & & & \ell & \\
h^1 & &  r & & & & j & 
\end{array}
\end{equation}Let  
$$\partial_u : H^0(L) \to H^1(N)$$be the {\em coboundary map} (simply denoted by $\partial$ if there is no danger of confusion) and let $\cork(\partial_u) := \dim({\rm Coker}(\partial_u))$. Then
$$i(\Ff_u) = j +  \cork(\partial_u).$$As for  (semi)stability of $\Ff_u$, information can be obtained by using \cite[Prop. 1.1]{LN}  (see  Proposition \ref{prop:LN} below) and the projection technique from \cite{CF} (see Theorem \ref{C.F.VdG} below). 

For the reader's convenience,  we recall \cite[Prop. 1.1]{LN} (cf. also \cite[\S's\,3,\,4]{Be}, \cite[\S\,3]{BeFe}). 
Take $u \in {\rm Ext}^1(L,N)$. Tensor by $N^{\vee}$ and consider $ \Ef_e:= \Ff_u \otimes N^{\vee}$, which is an extension  
$$(e) : \;\;\; 0 \to \Oc_C \to \Ef_e \to L \otimes N^{\vee} \to 0,$$where $e \in {\rm Ext}^1(L\otimes N^{\vee}, \Oc_C)$. Then $\deg(\Ef_e) = 2\delta - d$.  
From \eqref{eq:seginv}, one has $s(\Ff_u) = s(\Ef_e)$ and, by Serre duality,  $u$ and $e$ define the same point in  
\begin{equation}\label{eq:PP}
\Pp:= \Pp(H^0(K_C + L-N)^{\vee}).   
\end{equation}  

\begin{remark}\label{rem:sech}
If  $\deg(L-N)=2\delta - d \geqslant 2$,  then $\dim(\Pp) \geqslant g \geqslant 3 $  and the map $\varphi := \varphi_{|K_C+L-N|} : C \to \Pp$ is a morphism. Set $X := \varphi(C) \subset \Pp$. 
For any positive integer $h$ denote by ${\rm Sec}_h(X)$ the {\em $h^{th}$-secant variety of $X$}, defined as the closure of the union of all linear subspaces $\langle \varphi (D) \rangle \subset \Pp$, for all effective general  divisors of degree $h$. One has 
\begin{equation}\label{eq:sech}
\dim({\rm Sec}_h(X)) = {\rm min} \{\dim(\Pp), 2h-1\}. 
\end{equation} 
\end{remark}

\begin{proposition}  \label{prop:LN} (see \cite[Prop. 1.1]{LN})  Let $2\delta-d \geqslant 2$. For any integer $$\sigma \equiv  2\delta - d \pmod{2} \;\;and  \;\; 4 + d - 2 \delta \leqslant \sigma \leqslant 2\delta -d,$$one has 
$$s(\Ef_e) \geqslant \sigma \Leftrightarrow e \in \!\!\!\!| \ \ {\rm Sec}_{\frac{1}{2}(2\delta - d + \sigma -2)}(X).$$
\end{proposition}


To end this section, we remark that later on we will need the following technical result.

\begin{lemma}\label{lem:technical} Let $L$ and $N$ be as in \eqref{eq:exthyp} and such that $h^0(N-L) = 0$. 
Take $u, u' \in {\rm Ext}^1(L,N)$ such that:

\noindent
(i) the corresponding rank-two vector bundles $\Ff_u$ and $\Ff_{u'}$ are stable, and

\noindent
(ii) there exists an isomorphism $\varphi$ 
\[
\begin{array}{ccccccl}
0 \to & N & \stackrel{\iota_1}{\longrightarrow} & \Ff_{u'} & \to & L & \to 0 \\ 
 & & & \downarrow^{\varphi} & & &  \\
0 \to & N & \stackrel{\iota_2}{\longrightarrow} & \Ff_u & \to & L & \to 0
\end{array}
\]such that $\varphi \circ \iota_1 = \lambda \iota_2$, for some $\lambda \in \C^*$. 

Then $\Ff_u = \Ff_{u'}$, i.e.  $u, u'$ are proportional vectors in ${\rm Ext}^1(L,N)$. 
\end{lemma}

\begin{proof} The proof is similar to that in \cite[Lemma 1]{Ma3} so it can be left to the reader. 
 \end{proof}


\section{Stable bundles as extensions of line bundles}\label{sec:constr}

In this section we start with line bundles $L$ and $N$ on $C$ as in \S\,\ref{S:BNLN}, and consider rank--two vector bundles $\Ff$ on $C$ arising as extensions as in \eqref {eq:Fund}. We give conditions under which $\Ff$ is stable, with a given speciality, and $L$ is a quotient with suitable minimality properties.

\subsection{The case $N$ non-special}\label{S:Nns} Here we  focus on the case $N$ non-special.  Notation as in  \eqref{eq:exthyp1},  \eqref{eq:exthyp}, with  therefore $r=0$ (by the non-speciality assumption).  

\begin{theorem}\label{LN} Let $j \geqslant 1$ and  $g \geqslant 3$ be integers.  Let $C$ be of genus $g$ with general moduli. Let $\delta$ and $d$ be integers such that 
\begin{equation}\label{eq:ln2}
\delta \leqslant g-1 + \frac{g}{j} -j,  
\end{equation}
\begin{equation}\label{eq:ln3} 
\delta + g- 1 \leqslant d \leqslant 2\delta -2. 
\end{equation}Let $N \in {\rm Pic}^{d-\delta}(C)$ be general and $L \in W^{\delta-g+j}_{\delta}(C)$ be a smooth point (i.e. $L$ of speciality $j$ and so $h^0(L) = \ell$ as in \eqref{eq:exthyp1}). Then, for $u \in {\rm Ext}^1(L,N)$ general, the corresponding rank-two vector bundles $\Ff_u$ is indecomposable with:

\noindent
(i) $s(\Ff_u) = 2 \delta -d$. In particular $2 \leqslant s(\Ff_u) \leqslant \frac{g}j-j$, hence $\Ff_u$ is also stable; 

\noindent
(ii) $i(\Ff_u) = j$;  

\noindent
(iii) $L$ is a quotient  of  minimal degree of $\Ff_u$. 

\end{theorem}

\begin{remark}\label{rem:lnb}  (1) Inequality in \eqref{eq:ln2} is equivalent to $\rho(g, \ell -1, \delta) \geqslant 0$, where $\rho(g, \ell-1,\delta)$ is the Brill-Noether number as in \S\ref{sec:P} for line 
bundles  $L \in {\rm Pic}^{\delta}(C)$ of speciality $j$; this is a necessary and sufficient condition for $C$ of genus $g$ with general moduli to admit such a line bundle $L$ (cf.\,\cite{ACGH}). For what concerns \eqref{eq:ln3},  the upper-bound on $d$ reads $2 \delta - d \geqslant 2$, which is required to apply Proposition \ref{prop:LN}, whereas the lower-bound is equivalent to $\deg(N) = d-\delta \geqslant g-1$, hence the general  $N \in {\rm Pic}^{d-\delta}(C)$, for $C$ general, is non-special.

\noindent
(2)  Notice that \eqref{eq:ln3} implies $\delta \geqslant g+1$. Thus, together with \eqref{eq:ln2}, one has $g \geqslant j^2+ 2j$ i.e.  $1 \leqslant j \leqslant \sqrt{g+1} -1$. 
\end{remark}

\begin{proof}[Proof of Theorem \ref{LN}] By Remark  \ref{rem:lnb}-(1) $N$ is non-special, so (ii) holds. Moreover, always by Remark  \ref{rem:lnb}-(1), we can use Proposition \ref{prop:LN} with $\sigma :=  2\delta-d$. One has $\dim(\Pp) = 2\delta-d+g-2$. 
From \eqref{eq:sech}, we have$$\dim\left({\rm Sec}_{\frac{1}{2}(2(2\delta-d)-2)}(X)\right)  = {\rm min} \{ \dim(\Pp), 2(2\delta-d) -3 \} = 2(2\delta-d) -3,$$since 
 \eqref{eq:ln2} and \eqref{eq:ln3} yield $2\delta-d \leqslant\frac gj-j$ which implies $2(2\delta-d) -3  < 2\delta-d+g-2 $. In particular, \linebreak $ \dim\left({\rm Sec}_{\frac{1}{2}(2(2\delta-d)-2)}(X)\right) < \dim(\Pp)$. 
From Proposition \ref{prop:LN}, $u \in {\rm Ext}^1(L,N)$ general is such that $s(\Ff_u) \geqslant 2\delta-d$. If  $\Gamma$ is the section 
corresponding to $\Ff_u \to \!\! \to L$, one has $\Gamma^2 = 2 \delta - d$ as in \eqref{eq:Ciro410} so $s(\Ff_u) = 2 \delta - d$, hence (i) and (iii) are also proved. 
\end{proof}

\begin{corollary}\label{cor:LN}  Assumptions as in Theorem \ref{LN}.  Let $\Gamma$ be the section of $F_u = \Pp(\Ff_u)$ corresponding to $\Ff_u \to \!\! \to L$. Then $\Gamma$ is of minimal degree. In particular, $\Gamma$ is li and  $0 \leqslant \dim({\rm Div}_{F_u}^{1,\delta}) \leqslant 1$. 
\end{corollary}

\begin{proof} The fact that $\Gamma$ is of minimal degree follows from (i) of Theorem \ref{LN}. The rest is a consequence of minimality and 
Proposition \ref{prop:lem4}.  
\end{proof}

\begin{theorem}\label{C.F.VdG} Let $j \geqslant 1$ and  $g \geqslant 3$ be integers.  Let $C$ be of genus $g$ with general moduli. Let $\delta$ and $d$ be integers such that \eqref {eq:ln2} holds and moreover 
\begin{equation}\label{eq:cf3}
\delta + g+3 \leqslant d \leqslant 2\delta. 
\end{equation}Let $N \in {\rm Pic}^{d-\delta}(C)$ and  $L \in W^{\delta-g+j}_{\delta}(C)$ be general points. Then, 
for any $u \in {\rm Ext}^1(L,N)$,  the corresponding rank-two vector bundle $\Ff_u$ is very-ample, with  $i(\Ff_u) = j$. 
Moreover, for $u \in  {\rm Ext}^1(L,N)$ general

\noindent
(i) $L$ is the quotient of minimal degree of $\Ff_u$, thus  
$$0 \leqslant s(\Ff_u) = 2\delta -d \leqslant \frac{g-4j-j^2}{j},$$ so 
 $\Ff_u$ is stable when $2\delta-d >0$, strictly semistable when $d=2\delta$; 

\noindent
(ii) if $\Gamma$ is the section of $F_u = \Pp(\Ff_u)$ corresponding  to $\Ff_u \to \!\! \to L$, then ${\rm Div}^{1,\delta}_{F_u} = \left\{ \Gamma\right\}$ and $\Ff_u$ is rp via $L$.   
\end{theorem}

\begin{proof}  The proof is as in \cite[Theorem 2.1]{CF}, and it works also in the case $d= 2\delta$, not considered there. 
\end{proof}

\begin{remark}\label{rem:C.F.VdG}  (1)  The lower-bound  in \eqref{eq:cf3} reads $\deg(N) = d-\delta \geqslant g+3$, hence $N \in {\rm Pic}^{d-\delta}(C)$ general is non-special and $\delta \geqslant g+3$.

\noindent
(2) From \eqref{eq:ln2} and $\delta \geqslant g+3$, one has $g \geqslant j^2+ 4j$ i.e.  $1 \leqslant j \leqslant \sqrt{g+4} -2$.

\noindent 
(3) The bounds on $d$  in \eqref{eq:ln3} and \eqref{eq:cf3} are in general slightly worse than those in Theorem \ref{thm:TB} (cf. \cite[Remark 1.7]{CF}). For $j$ close to the upper-bound (see Remark \ref{rem:lnb}-(2), respectively (2) 
above), the difference is of the order of $\sqrt{g}$. However our approach gives the additional information of the description of 
  vector bundles  in irreducible components of $B_C^{k_j}(d)$ 
(see also \S\,\ref{S:BND}) simply as  line bundle extensions, with no use of either limit linear series or degeneration techniques.    
\end{remark}



\subsection{The case $N$ special}\label{S:Ns}

In this section $N \in {\rm Pic}^{d-\delta}(C)$ is assumed to be special.  Hence, in \eqref{eq:exthyp1}, we have $\ell, j, r >0$ whereas $n \geqslant 0$ (according to the fact that $N$ is either effective or not). For any integer $t >0$, consider 
\begin{equation}\label{eq:wt}
\mathcal W_t := \left\{ u \in  {\rm Ext}^1(L,N) \ | \ \cork(\partial_u) \geqslant t \right\} \subseteq {\rm Ext}^1(L,N),  
\end{equation}which has a natural structure of determinantal scheme; as such, $\mathcal W_t$ has  {\em expected codimension}  
\begin{equation}\label{eq:clrt} 
c(l,r,t) := t (\ell - r + t)  
\end{equation}As in \eqref{eq:lem3note}, put $m = \dim( {\rm Ext}^1(L,N))$. Thus, if $m >0$ and $\mathcal W_t \neq \emptyset$, then any irreducible component $\Lambda_t \subseteq \mathcal W_t$ is such that
\begin{equation}\label{eq:expdimwt} 
\dim(\Lambda_t) \geqslant {\rm min} \left\{m , m- c(\ell,r,t) \right\}, 
\end{equation}where the right-hand-side is the {\em expected dimension} of $\mathcal W_t$. These loci have been considered also in \cite[\S\,2,\,3]{BeFe}, \cite[\S\;6,\,8]{Muk} for low genus and for vector bundles with canonical determinant. 

\begin{remark}\label{rem:cirow} One has $\dim({\rm Ker}(\partial_u)) =  1 + \dim(\langle \Gamma\rangle)$, where $\Gamma$ is the section corresponding to the quotient $\Ff \to \!\!\to L$. 
Note that $\dim(\langle \Gamma\rangle) = -1$ 
if and only if $\dim({\rm Ker}(\partial_u)) = 0$, i.e. $H^0(\Ff) = H^0(N)$. This happens if and only if $\Gamma$ is 
a fixed component of $|\Oc_F(1)|$, i.e. if and only if the image of the map $\Phi_{|\Oc_F(1)|}$ has dimension 
smaller than $2$.   If $d \geqslant 2g-2$ and this  happens, then  one must have 
$n \geqslant i \geqslant j \geqslant 1$, where $i = i(\Ff)$. 
\end{remark}

\begin{remark}\label{rem:wt} The coboundary map $\partial_u$  can be interpreted in terms of multiplication maps  
among global sections of suitable line bundles on $C$. Indeed,  consider $r \geqslant t$ and $\ell \geqslant {\rm max} \{1,r-t\}$. Denote 
by
$$\cup : H^0(L) \otimes H^1(N - L) \longrightarrow H^1(N),$$the cup-product: for any $u \in H^1(N - L) \cong {\rm Ext}^1(L,N)$, one has $\partial_u (-) = - \cup u.$ By Serre duality, 
the consideration of  $\cup$ is equivalent to the one of  the multiplication map
\begin{equation}\label{eq:mu}
\mu := \mu_{L,K_C-N} : H^0(L) \otimes H^0(K_C-N)  \to H^0(K_C+L-N)
\end{equation} (when $N=L$, $\mu$ coincides with $\mu_0(L)$ as in \eqref{eq:Petrilb}).  
For any subspace $W \subseteq H^0(K_C-N)$, we set
\begin{equation}\label{eq:muW}
\mu_W:= \mu|_W : H^0(L) \otimes W \to H^0(K_C+L-N).
\end{equation}Imposing $\cork(\partial_u) \geqslant t$ is equivalent to 
$$V_t := {\rm Im}(\partial_u)^{\perp}  \subset H^0(K_C - N) $$ having dimension at least $t$. Therefore 
$$\mathcal W_t = \left\{u \in  H^0(K_C+L-N)^{\vee} \mid 
\exists\,V_t \subseteq H^0(K_C-N),\; {\rm s.t.} \; \dim(V_t) \geqslant t \;{\rm and}\; {\rm Im}(\mu_{V_t}) \subseteq \{u = 0 \}
\right\},$$where $\{u = 0 \} \subset H^0(K+L-N)$ is the hyperplane defined by $u \in H^0(K_C+L-N)^{\vee}$. 
\end{remark}

\begin{theorem}\label{thm:mainext1} Let $C$ be a smooth curve of genus $g \geqslant 3$. Let 
$$r\geqslant 1, \; \ell \geqslant {\rm max} \{1, r-1\}, \;  m \geqslant \ell +1$$be integers as in \eqref{eq:lem3note}, \eqref{eq:exthyp1}. Then (cf.\,\eqref{eq:wt},\eqref{eq:clrt}): 

\noindent
(i) $c(\ell, r,1) =\ell-r+1\geqslant 0$; 

\noindent
(ii)  $\mathcal W_1$ is irreducible of the expected dimension $\dim (\mathcal W_1) = m - c(\ell, r,1) \geqslant r$.  In particular  
$\mathcal W_1={\rm Ext}^1(L,N)$ if and only if $\ell=r-1$. 
\end{theorem}

\begin{proof}   Part (i) and $ m - c(\ell, r,1)\geqslant r$ are  obvious. Let us prove (ii). Since $\ell, r \geqslant 1$, both $L$ and $K_C-N$ are effective. One has an inclusion $$0 \to L \to K_C+L-N$$obtained by 
tensoring by $L$ the injection $\Oc_C \hookrightarrow K_C-N$ given by a given non--zero  section of $K_C-N$. 
Thus, for any $V_1 \in \Pp(H^0(K_C-N))$, $\mu_{V_1}$ as in \eqref{eq:muW} is injective. Since $m \geqslant \ell + 1$, one has $\dim({\rm Im} (\mu_{V_1})) = \ell \leqslant m-1 $, i.e. ${\rm Im}(\mu_{V_1})$ is contained in some hyperplane of 
$H^0(K_C+L-N)$.  Let  
$$\Sigma:= \left\{\sigma := H^0(L) \otimes V_1^{\sigma} \subseteq H^0(L) \otimes H^0(K_C-N) \; \mid \; V_1^{\sigma} \in \Pp(H^0(K_C-N)) \right\}.$$Thus 
$\Sigma \cong \Pp(H^0(K_C-N))$,  so it is irreducible of dimension $r-1$. Since 
$\Pp(H^0(K_C+L-N)^{\vee}) = \Pp$ as in \eqref{eq:PP}, we can define the incidence variety 
$$\Jj := \left\{(\sigma, \pi) \in \Sigma \times \Pp \; | \; \mu_{V_1^{\sigma}}(\sigma) \subseteq \pi \right\} \subset \Sigma \times \Pp.$$ Let 
$$ \Sigma \stackrel{pr_1}{\longleftarrow} \Jj \stackrel{pr_2}{\longrightarrow}  \Pp$$be the two  projections. As we saw, 
$pr_1$ is surjective. In particular $\Jj \neq \emptyset$ and, for any $\sigma \in \Sigma$, 
$$pr_1^{-1} (\sigma) = \left\{ \pi \in \Pp\,|\,\mu_{V_1^{\sigma}}(\sigma) \subseteq \pi\right\} \cong 
|\Ii_{\widehat{\sigma}\vert \Pp^{\vee}} (1)|,$$where $\widehat{\sigma} := \Pp(\mu_{V_1^{\sigma}}(\sigma))$. Since $\dim(\widehat{\sigma}) = 
\ell -1$, then $\dim(pr_1^{-1} (\sigma)) = m -1 - \ell \geqslant 0$.  

This shows that  $\Jj$ is irreducible and $\dim(\Jj) = m-1 - c(\ell,r,1) \leqslant m-1$. Then, 
$\widehat{\mathcal W}_1 := \Pp(\mathcal W_1) = \overline{pr_2(\Jj)}$. Recalling \eqref{eq:expdimwt}, $\mathcal{W}_1 \neq \emptyset$ is irreducible, of the expected dimension 
$m - c(\ell,r,1)$. 
\end{proof}

\begin{corollary}\label{cor:mainext1} Assumptions as in Theorem \ref{thm:mainext1}. If   $\ell \geqslant r$, then $\mathcal W_1 \varsubsetneq  {\rm Ext}^1(L,N)$ and 

\noindent
(i)  for $u \in {\rm Ext}^1(L,N)$ general, $\partial_u$ is surjective, in which case the corresponding bundle $\Ff_u$ is special with $i(\Ff_u) = h^1(L) = j$;  

\noindent
(ii) for $v \in \mathcal W_1$ general, $\cork(\partial_v)=1$, hence the corresponding bundle $\Ff_v$ is special with $i(\Ff_v) = j+1$.
\end{corollary}

\medskip


\subsubsection{{\bf Surjective coboundary}}\label{ss:uepi}  Take $ 0 \neq u \in {\rm Ext}^1(L,N)$ and assume 
$\partial_u$ is surjective (from Corollary \ref{cor:mainext1}, this happens e.g. when $\ell \geqslant r$, $m \geqslant \ell +1$ and $u$ general). 

\begin{theorem}\label{uepi} Let $j \geqslant 1$ and $g \geqslant 3$ be integers. 
Let $C$ be of genus $g$ with general moduli. Let $\delta$ and $d$ be integers such that  \eqref {eq:ln2} holds and moreover 
\begin{equation}\label{eq:uepi3}
2g-2 \leqslant d \leqslant 2 \delta - g. 
\end{equation}Let $L \in W^{\delta-g+j}_{\delta}(C)$ be a smooth point and $N \in {\rm Pic}^{d - \delta}(C)$ be any point. 
Then, for $u \in {\rm Ext}^1(L,N)$ general, the corresponding bundle $\Ff_u$ is indecomposable with 
 
\noindent
(i) speciality $i(\Ff_u) = j = h^1(L)$.

\noindent
(ii) $s(\Ff_u) = g-\epsilon$, $\epsilon \in \{0,1\}$ such that $\epsilon \equiv d+g \pmod{2}$. In particular, 
$\Ff_u$ is stable.   

\noindent
(iii) The minimal degree of a quotient of $\Ff_u$  is $\frac{d+g-\epsilon}{2}$ and  
$1- \epsilon \leqslant \dim\left({\rm Div}^{1,\frac{d+g-\epsilon}{2}}_{F_u} \right) \leqslant  1$; 

\noindent
(iv) $L$ is a minimal degree quotient of $\Ff_u$ if and only if $\epsilon =0$ and $ d = 2 \delta-g$. 

\end{theorem}

\begin{remark}\label{rem:uepib}  (1) From \eqref{eq:uepi3} we get $\delta \geqslant \frac{3}{2}g-1$ hence from \eqref{eq:ln2}, $j \leqslant \frac{\sqrt{g^2 + 16 g} - g}{4}$.

\noindent
(2) Since $L$ is special, then $\delta \leqslant 2g-2$. Therefore, the upper-bound in \eqref{eq:uepi3} implies 
 $d - \delta \leqslant \delta - g  \leqslant  g -2$, i.e., any $N \in {\rm Pic}^{d-\delta}(C)$ is special too.

\noindent
(3) The inequalities \eqref {eq:ln2}, \eqref{eq:uepi3} imply  $\ell \geqslant r $, 
$m \geqslant \ell +1$ as in the assumptions of Corollary \ref{cor:mainext1}.  Indeed:

\noindent
$\bullet$ $\ell \geqslant r$ reads 
\begin{equation}\label{eq:aiuto}
\delta \geqslant g-1 + r - j.
\end{equation} 
Since  $r =  \delta -d + g -1 + n$, then   \eqref{eq:aiuto} is equivalent to $d \geqslant 2g-2 -j +n $. Thus \eqref{eq:aiuto} holds by \eqref{eq:uepi3}, if  $n \leqslant 1$. 
 If  $n \geqslant 2$, $C$ with general moduli implies $r \leqslant \frac{g}{n} \leqslant \frac{g}{2}$ so 
$g-1 + r - j \leqslant \frac{3}{2} g - 1 - j$ and \eqref{eq:aiuto} holds because $\delta \geqslant \frac{3}{2} g -1$. 

\noindent
$\bullet$ We have $d - \delta \leqslant \delta -g < \delta$ by \eqref{eq:uepi3}. So from \eqref{eq:lem3note}
we have $m = 2 \delta - d + g - 1$. Thus $m \geqslant \ell +1$ reads  $d \leqslant \delta + 2g- 3 - j$. 
By \eqref{eq:uepi3}, to prove this it suffices to prove $2 \delta - g \leqslant \delta + 2g- 3 - j$. This in turn is a consequence of  \eqref {eq:ln2}.

\noindent
(4) Notice that, under hypotheses of Theorem \ref{uepi}, when $\epsilon =1$ $L$ is not of minimal degree: from (iii), one would 
have $d=2\delta - g + 1$ which is out of range in \eqref{eq:uepi3}. Indeed, if $d = 2\delta - g + 1$ and e.g. $\delta = 2g-2$, then 
$d=3g-3$, $\deg(N) = d - \delta = g-1$, thus if $N$ is general, it is non-special, which is a case already considered in Theorem \ref{LN}. 
From (1) above, to allow minimality for $L$ also for $\epsilon =1$, one should replace \eqref{eq:ln2}, \eqref{eq:uepi3} in the statement of Theorem \ref{uepi} with the more annoying conditions 
$\delta \leqslant {\rm min} \{ g-1 + \frac{g}{j}- j, 2g-3\}$ and $d \leqslant 2\delta - g + \epsilon$, respectively.   
\end{remark}

\begin{proof}[Proof of Theorem \ref{uepi}] By Remark \ref{rem:uepib}-(2),  $N$ is special. Moreover, by  Remark \ref{rem:uepib}-(3) and Corollary \ref{cor:mainext1}, 
for $u \in {\rm Ext}^1(L,N)$ general, $\partial_u$ is surjective. Hence (i) holds. 

From the upper-bound in \eqref{eq:uepi3} and $g \geqslant 3$, we can apply Proposition \ref{prop:LN} 
with the choice $\sigma :=  g-\epsilon$, i.e.,  the maximum for which $\sigma \equiv 2 \delta - d \pmod{2}$, 
$\sigma \leqslant2 \delta - d$ and one has a strict  inclusion
$${\rm Sec}_{\frac{1}{2}(2\delta-d + g-\epsilon -2)}(X) \subset \Pp,$$from \eqref{eq:lem3note} and \eqref{eq:sech}.

If $\epsilon =0$,  (ii) follows from Propositions \ref{prop:Nagata}, \ref{prop:LN}. Let $\Gamma \subset F_u$ be a section of minimal degree, which we denote 
by $m_0$. Then, $\Gamma^2= 2 m_0 - d = g$ (cf. \eqref{eq:Ciro410} and Remark \ref{rem:seginv}).  In particular, $m_0 = \frac{d+g}{2}$ and$$1 = \Gamma^2 - g + 1 \leqslant \chi(\N_{\Gamma/F_u}) \leqslant \dim\left({\rm Div}^{1,\frac{d+g}{2}}_{F_u} \right) \leqslant 1,$$where 
the upper-bound holds by the minimality condition (cf. proof of Proposition \ref{prop:lem4}). This proves (iii) in this case.

When $\epsilon=1$,  by Propositions  \ref{prop:Nagata}, \ref{prop:LN}, one has $g-1 \leqslant s(\Ff_u) \leqslant g$ and, by parity, the leftmost equality holds. As above, part (iii) holds also for $\epsilon =1$. 
 
Finally, $L$ is a minimal degree quotient if and only if $2 \delta = d+g- \epsilon$ which by \eqref{eq:uepi3} is only possible for $\epsilon =0$, proving (iv) (cf. Remark \ref{rem:uepib}-(4)). 
 \end{proof}


\subsubsection{{\bf Non-surjective coboundary}}\label{ss:unepi} From  Corollary 
\ref{cor:mainext1}, when $\ell \geqslant r$ and $m \geqslant \ell+1$, for  $v \in \mathcal W_1\subsetneq {\rm Ext}^1(L,N)$  general, one has $\cork(\partial_v)=1$.

\begin{definition}\label{def:goodc} Take $\ell \geqslant r \geqslant t \geqslant 1$ integers. Assume

\noindent
(1) there exists an irreducible component $\Lambda_t \subseteq \mathcal W_t$ with the expected dimension $\dim(\Lambda_t) = m - c(\ell, r,t)$;

\noindent
(2) for $v \in \Lambda_t$ general, $\cork(\partial_v)=t$. 

\noindent
Any such $\Lambda_t$  is called  a {\em good component} of $\mathcal W_t$. 
\end{definition}

\noindent
By Theorem \ref{thm:mainext1}, $\Lambda_1 = \mathcal W_1$ is good. In \S\,\ref{ssec:suffcon} we shall give sufficient conditions  for goodness  when $t \geqslant 2$.


With notation as in  \eqref{eq:PP},  for any $t \geqslant 1$ and any good component $\Lambda_t$, we set
\begin{equation}\label{eq:Lahat}
\widehat{\Lambda}_t := \Pp(\Lambda_t) \subset \Pp
\end{equation}(cf. notation as in the proof of Theorem \ref{thm:mainext1} for $\widehat{\mathcal W}_1$). 

\begin{theorem}\label{unepi} Let $g \geqslant 3$, $d \geqslant 2g-2$, $j \geqslant 1$, $\ell \geqslant r \geqslant t \geqslant 1$ be integers. Take $\epsilon \in \{0,1\}$ such that 
$$d + g - c(\ell,r,t) \equiv \epsilon \pmod{2}.$$Take $\delta = \ell + g - 1 - j$ and assume 

\begin{equation}\label{eq:unepi0}
g \geqslant c(\ell, r, t) + \epsilon,
\end{equation}
\begin{equation}\label{eq:unepi1}
g +  r - j - 1 \leqslant \delta \leqslant   g-1 + \frac{g}{j} - j,
\end{equation}
\begin{equation}\label{eq:unepi2}
2 \delta - d \geqslant {\rm max} \{ 2, g -  c(\ell, r, t)- \epsilon\}. 
\end{equation} Let $C$ be a curve of genus $g$ with general moduli. Let $L \in W^{\ell-1}_{\delta}(C)$ be a smooth point, 
$N \in {\rm Pic}^{d-\delta}(C)$ of speciality $r$. Then, for any good component $\Lambda_t$ and for $v \in \Lambda_t$ general, the corresponding bundle $\Ff_v$ is 

\noindent
(i) special with $i(\Ff_v) = j+t = h^1(L) +t$; 

\noindent
(ii) $s(\Ff_v) \geqslant g - c(\ell,r,t) - \epsilon \geqslant 0$; in particular, when $ g - c(\ell,r,t)>0$,  $\Ff_v$ is stable, hence indecomposable. 

\noindent
(iii) If  $2 \delta = d + g - c(\ell,r,t) - \epsilon$, then $L$ is a quotient of minimal degree of $\Ff_v$.  
\end{theorem}

\begin{remark}\label{rem:unepi}  (1) As before, the upper-bound on $\delta$ in \eqref{eq:unepi1} is  equivalent to $\rho(g , \ell -1 , \delta) \geqslant 0$ whereas the lower bound to  $\ell \geqslant r$. 

\noindent
(2) Condition $2\delta -d \geqslant 2$ in \eqref{eq:unepi2} is as in Proposition \ref{prop:LN}; the other condition 
in \eqref{eq:unepi2} will be clear by reading the proof of Theorem \ref{unepi}.

\noindent

\noindent
(3) Arguing as in Remark \ref{rem:uepib}-(3), one shows that $m \geqslant \ell +1$. 

\end{remark}

\begin{proof}[Proof of Theorem \ref{unepi}]  By \eqref{eq:unepi2} we may apply Proposition \ref{prop:LN} choosing $\sigma := g-c(\ell,r,t) - \epsilon$, which is non-negative 
by \eqref{eq:unepi0}. This is the maximum integer such that   $\sigma \leqslant 2 \delta - d$, $\sigma \equiv 2\delta - d \pmod{2}$ and such that 
$\dim(\widehat{\Lambda}_t) > \dim ({\rm Sec}_{\frac{1}{2}(2\delta-d + \sigma -2)}(X))$, as it follows from \eqref{eq:unepi2}.  Then, if $v \in \widehat{\Lambda}_t$ general, the assertions hold. 
\end{proof}


\begin{remark}\label{unepine} When $N$ of degree $d-\delta$ is non-effective of speciality $r$, by Riemann-Roch 
$ r = \delta - d + g-1$. Thus, by  \eqref{eq:ki}, one simply has $c (\ell, r, t) = t\;(d-2g+2+j+t)$, i.e. 
$c (\ell, r, t) = t\;k_{j+t}$, and conditions  in Theorem \ref{unepi} can be replaced by more explicit numerical conditions on $d$ and $\delta$   
$$\delta \leqslant   g-1 + \frac{g}{j} - j \; \; \; {\rm and} \; \;\; d \leqslant g-1 - t + {\rm min} \left\{ \delta, g - 1 - j +  \frac{g - \epsilon}{t}\right\},$$where $\epsilon \in \{0,1\}$  such that $d + g -  t k_{t+j} \equiv \epsilon \pmod{2}$, and
$$2 \delta - d \geqslant {\rm max} \{ 2, g -  t k_{t+j} - \epsilon\}.$$ 
\end{remark}

\begin{remark}\label{unepie} When  otherwise $N$, of degree $d-\delta$, is effective and of speciality $r$ one gets $d \geqslant \delta + g - r$.  Moreover, since its Brill-Noether number 
$\rho(g, n-1, d - \delta)$ has to be non negative (by the generality of $C$), one gets $d \leqslant \delta + g - 1 + \frac{g}{r}- r$. Thus, conditions in  Theorem \ref{unepi}  can be replaced by numerical conditions 
$$g -1 -j +r \leqslant \delta \leqslant   g-1 - j +  {\rm min} \left\{\frac{g}{j},  \frac{g - \epsilon}{t} + r - t-1 \right\},$$
$$g  + \delta - r \leqslant d \leqslant  \delta + g - 1 + \frac{g}{r} - r \;\;\; {\rm and} \;\;\; 2 \delta - d \geqslant {\rm max} \{ 2, g -  c(\ell,r,t) - \epsilon\},$$with $\epsilon$ and $c(\ell,r,t)$ as in Theorem \ref{unepi}.  
\end{remark}


\subsection{Existence of good components}\label{ssec:suffcon}  Recalling what defined in \eqref{eq:lem3note}, \eqref{eq:exthyp1}, \eqref{eq:wt}, \eqref{eq:clrt} and in Remark \ref{rem:wt}, one has 

\begin{theorem}\label{thm:mainextt} Let $C$ be a smooth curve of genus $g \geqslant 3$. Take integers $m$, $\ell$, $r$ and $t$ and assume $\ell \geqslant r \geqslant t \geqslant 2$. Take any integer $\eta$ such that 
\begin{equation}\label{eq:extt0}
0 \leqslant \eta \leqslant {\rm min}\{ t (r-t), \ell (t-1) \} \; \;\; {\rm and}\;\; \; m \geqslant \ell t+1-\eta.  
\end{equation} Suppose, in addition, that the subvariety  $\Sigma_{\eta} \subseteq \mathbb{G}(t, H^0(K_C-N)) := \mathbb{G}$, parametrizing 
$V_t \in  \mathbb{G}$ s.t.  
\begin{equation}\label{eq:aiutomamma}
\dim({\rm Ker}(\mu_{V_t})) \geqslant \eta,  
\end{equation}has pure codimension $\eta$ in $\mathbb{G}$ and that, for the general point $V_t$ in any irreducible component of $\Sigma_{\eta}$, 
equality holds in \eqref{eq:aiutomamma}. Then:

\noindent
(i) $c(\ell,r,t) >0$; 

\noindent
(ii) $\emptyset \neq \mathcal W_t \subset \mathcal W_1 \subset {\rm Ext}^1(L,N)$, where all the inclusions are strict;

\noindent
(iii) there exists a good component $\Lambda_t $ of $ \mathcal W_t$. 
 \end{theorem}

\begin{proof}   By \eqref{eq:extt0}  one has $m \geqslant \ell +1$; moreover $\ell \geqslant r$ by assumption. 
Thus, from Corollary \ref{cor:mainext1},  $\emptyset \neq \mathcal W_1 \subset {\rm Ext}^1(L,N)$ and the inclusion is strict. 
By definition $\mathcal W_t \subset  \mathcal W_1$, where the inclusion is strict by Corollary  \ref{cor:mainext1}. 
A similar argument  as in the proof  of Theorem \ref{thm:mainext1} applies. 
\end{proof}

\begin{corollary}\label{cor:mainextt} Let $C$ be of genus $g \geqslant 3$ with general moduli. Let $j \geqslant 1$, $\ell \geqslant r \geqslant t \geqslant 2$ and $m \geqslant \ell t+1$ be integers. Let $L \in W^{\delta-g+j}_{\delta}(C)$ be a smooth point. 

If $j \geqslant t$, $N \in {\rm Pic}^{d-\delta}(C)$ is general and $\ell \leqslant 2\delta - d$, then for $V_t \in  \mathbb{G} (t, H^0(K_C-N))$ general, $\mu_{V_t}$ is injective. In particular, there exists a good component $ \emptyset \neq  \Lambda_t \subseteq \mathcal W_t$. 
\end{corollary}

\begin{proof} Set $h := 2\delta -d$ and let $N_0:= L-D \in {\rm Pic}^{d-\delta}(C)$, with 
$D=\sum_{i=1}^h p_i \in C^{(h)}$ general. Since $0 < \ell \leqslant h$,  we have $h^0(N_0) =0$. Thus, $N \in {\rm Pic}^{d-\delta}(C)$ general is also non-effective, 
so $h^1(N) = h^1(N_0)$. 

Let $\mu$ be as in \eqref{eq:mu}. To prove injectivity of $\mu_{V_t}$ as in \eqref{eq:muW} 
for $N$ and $V_t$ general, it suffices to prove a similar condition for  
$$\mu^0: H^0(L) \otimes H^0\left(K_C - L + D \right) \to H^0\left(K_C + D \right).$$Consider 
$$W:= H^0(K_C-L) \subset H^0\left(K_C-L+ D\right).$$ One has $\dim(W) = j$. We have the  diagram
\[\begin{array}{rcl}
H^0(L) \otimes W \cong H^0(L) \otimes H^0(K_C-L) & \stackrel{ \mu^0_{W} }{\longrightarrow} &  H^0\left(K_C+D \right) \\
\searrow^{\mu_0(L)} & & \nearrow_{\iota} \\
& H^0(K_C) & 
\end{array}
\]where $\mu^0_{W} = \mu^0|_{H^0(L) \otimes W}$, $\mu_0(L)$ is as in \eqref{eq:Petrilb}   and $\iota$ is the obvious inclusion. 

By Gieseker--Petri theorem $\mu_0(L)$ is injective. By composition with $\iota$, $\mu^0_{W}$ is also injective. Since by assumptions $t \leqslant j$, then for any $\widetilde{V}_t \in \mathbb{G} (t,W)$, $\mu^0_{\widetilde{V}_t }$ is also injective. By semicontinuity, for $N \in {\rm Pic}^{d-\delta}(C)$ and 
$V_t \in \mathbb{G} (t, H^0(K_C-N))$ general, $\mu_{V_t}$ is injective. Then, one can conclude by using Theorem \ref{thm:mainextt}. 
\end{proof}


\section{Parameter spaces}\label{S:PSBN} 

Let $C$ be a projective curve of genus $g$ with general moduli. Given a sequence as in  \eqref{eq:exthyp},  for brevity we set  
$$\rho(L):= \rho(g, \ell-1, \delta) \; \;\; {\rm and} \; \;\; \rho(N) := \rho(g, n-1, d-\delta),$$
$$
W_L:= \left\{
\begin{array}{cc}
W_{\delta}^{\ell-1}(C) & {\rm if} \; \rho(L)>0\\
\{L\} & {\rm if} \; \rho(L)=0
\end{array}
\right. \;\;\; {\rm and}\;\;\; 
W_N:= \left\{
\begin{array}{cc}
W_{d-\delta}^{n-1}(C) & {\rm if} \; \rho(N)>0\\
\{N\} & {\rm if} \; \rho(N)=0
\end{array}
\right..$$Both $W_L$ and $W_N$ are irreducible, generically smooth, of dimensions $\rho(L)$ and $\rho(N)$
 (cf. \cite[p. 214]{ACGH}). Let 
$$\mathcal{N} \to C \times {\rm Pic}^{d-\delta}(C) \;\; {\rm and} \;\; \mathcal{L} \to C \times {\rm Pic}^{\delta}(C)$$be Poincar\'e line-bundles. With an abuse of notation, we will denote by $\mathcal L$ (resp., by $\mathcal N$) also the restriction of Poincar\'e line-bundle to the Brill-Noether locus. Set$$\mathcal Y := {\rm Pic}^{d-\delta}(C) \times W_L \;\; \;\; {\rm and} \;\; \;\; \mathcal Z:= W_N \times W_L \subset \mathcal Y .$$They are both irreducible,  of dimensions 
\begin{equation}\label{eq:yde1}
\dim(\mathcal Y) = g + \rho(L) \;\;\;  {\rm and} \; \;\; \dim(\mathcal Z) = \rho(N) + \rho(L).
\end{equation}Consider the  natural projections 
\[
\begin{array}{ccccc}
C \times {\rm Pic}^{d-\delta}(C) & \stackrel{pr_{1,2}}{\longleftarrow} &\!\!\!C \times\mathcal Y & \stackrel{pr_{2,3}}{\longrightarrow} & \mathcal Y \\
 & & \;\;\;\; \downarrow^{pr_{1,3}} & & \\
 & &\!\! C \times W_L & & .
\end{array}
\]As in \cite[p. 164-179]{ACGH}), we  define 
$$\Efra_\delta := R^1(pr_{2,3})_*  \left( pr_{1,2}^* (\mathcal N)   \otimes pr_{1,3}^* (\mathcal L^{\vee}) \right),$$depending on the choices of $d$ and $\delta$. By \eqref{eq:lem3note}, when $2 \delta - d \geqslant 1$, $\Efra_\delta$ is a vector bundle of rank $m = 2 \delta - d + g -1$ on $\mathcal Y$ whereas, 
when $d = 2 \delta$,  $\Efra_\delta$ is a vector bundle of rank $g-1$ on $\mathcal Y \setminus \Delta_{\mathcal Y}$, where 
$\Delta_{\mathcal Y} = \{ (M, M) \; | \; M \in W_L\} \cong W_L $. We set 
\begin{equation}\label{eq:ude}
\mathcal U :=  \left\{
\begin{array}{cc}
\mathcal Y & {\rm if} \; 2 \delta - d \geqslant 1 \\
\mathcal Y \setminus \Delta_{\mathcal Y} & {\rm if} \; d = 2 \delta
\end{array}
\right. \;\;\;\;\;\; \;\;\;\;\;\; {\rm and} \;\;\;\;\;\;\;\;\;\;\;\;  \Pp(\Efra_\delta) \stackrel{\gamma}{\longrightarrow} \mathcal U,
\end{equation}where $\gamma$ the  projective bundle morphism: the $\gamma$-fibre of $y = (N,L) \in \mathcal U$ is $\Pp(\Ext^1(L,N)) = \Pp$, as in \eqref{eq:PP}.

From \eqref{eq:lem3note} and \eqref{eq:yde1}, one has 
\begin{equation}\label{eq:yde2}
\dim(\Pp(\Efra_\delta)) = g + \rho(L) + m-1 \;\;\; \;\;\; {\rm and} \;\;\;\;\;\; \dim(\Pp(\Efra_\delta)|_{\zd}) = \rho(N)+ \rho(L) + m-1.
\end{equation}Since (semi)stability is an open condition (cf. e.g. \cite[Prop. 6-c, p. 17]{Ses}), for any choice of integers $g$, $d$ and $\delta$ satisfying numerical conditions as in the theorems and corollaries proved in \S's\;\ref{S:Nns} and \ref{S:Ns}, there is an open, dense subset 
$\Pp(\Efra_\delta)^0  \subseteq \Pp(\Efra_\delta)$ and a morphism 
\begin{equation}\label{eq:pde}
\pi_{d,\delta} : \Pp(\Efra_\delta)^0 \rightarrow U_C(d). 
\end{equation} We set 
\begin{equation}\label{eq:Nudde}
\vdj := {\rm Im} (\pi_{d,\delta}) \; \; \; {\rm and} \;\;\;  \nu_d^{\delta,j} = \dim(\vdj). 
\end{equation}



\subsection{Non-special $N$}\label{ssec:Nns} We will put ourselves in the hypotheses either of Theorem \ref{LN} or of Theorem \ref{C.F.VdG}. In either case, $d - \delta \geqslant g-1$ so $N$ can be taken general in $\Pic^{d-\delta}(C)$ and $\vdj \subseteq B^{k_j}_C(d)$, by what proved about (semi)stability.


\noindent
\subsubsection{{\bf Case $2 \delta - d \geqslant 1$}} In this case, by what proved in Theorems \ref{LN}, \ref{C.F.VdG}, one has $\vdj \subseteq B^{k_j}_C(d)  \cap U_C^s(d)$. Therefore any irreducible component of $B^{k_j}_C(d)$ intersected by  $\vdj$ has at least dimension $\rho_d^{k_j}$ 
(cf. Remark \ref{rem:BNloci} and Definition \ref{def:regsup}).

\begin{proposition}\label{C.F.1}  Assumptions as in Theorem \ref{C.F.VdG}, with $2 \delta - d \geqslant 1$. Then, for any integers $j,\,\delta,\,d$ therein, there exists an irreducible component $\mathcal{B} \subseteq B^{k_j}_C(d)$ such that: 

\noindent
(i) $\vdj \subseteq \mathcal{B}$;

\noindent
(ii) $\mathcal{B}$ is regular and generically smooth; 

\noindent
(iii) for $[\Ef] \in \mathcal{B}$ general, $\Ef$ is stable, with $s(\Ef) \geqslant 2\delta -d$ and $i(\Ef) =j$.  
\end{proposition}

\begin{proof} Parts (i) and (iii) follow from Theorem \ref{C.F.VdG} (note that the Segre invariant is lower-semicontinuous; cf. also \cite[\S\;3]{LN}).  
Assertion (ii) follows from the fact that, for $[\Ff] \in \vdj$ general, the Petri map $P_{\Ff}$ is injective  (cf. Remark \ref{rem:BNloci} and  \cite[Lemma2.1]{CF}).  
\end{proof}

\begin{lemma}\label{lem:claim1} In the hypotheses of Theorem \ref{C.F.VdG}, with $2 \delta - d \geqslant 1$, the morphism 
$\pi_{d,\delta}$ is generically injective. 
\end{lemma}
\begin{proof} Let $[\Ff] \in \vdj$ be  general; then $\Ff = \Ff_u$, for $u \in \Ext^1(L,N)$ and $y= (N,L) \in \ud$ general. Then 
$$\pi_{d,\delta}^{-1}([\Ff_u]) = \left\{ (N', L', u') \in \Pp(\Efra)^0 \, | \; \Ff_{u'} \cong \Ff_u \right\}.$$Assume by contradiction there exists 
$ (N', L', u') \neq (N,L,u)$ in $\pi_{d, \delta}^{-1} ([\Ff_u])$. Then  $N\otimes L \cong N' \otimes L'$. 

\noindent
(1) If $L \cong L' \in W_L$ then $N \cong N' \in \Pic^{d-\delta}(C)$. Thus, $u, u' \in \Pp$. Let $\varphi : \Ff_{u'} {\to} \Ff_u$ 
be the isomorphism between the two bundles. Since $\Ff_u$ is stable, then $u \neq u' \in \Pp$ (notation as in \eqref{eq:PP}) and we have the  diagram 
\[
\begin{array}{ccccccl}
0 \to & N & \stackrel{\iota_1}{\longrightarrow} & \Ff_{u'} & \to & L & \to 0 \\ 
 & & & \downarrow^{\varphi} & & &  \\
0 \to & N & \stackrel{\iota_2}{\longrightarrow} & \Ff_u & \to & L & \to 0.
\end{array}
\]The maps $\varphi \circ \iota_1$ and $\iota_2$ determine two non-zero sections $ s_1 \neq s_2 \in H^0(\Ff_u \otimes N^{\vee}) $. They are linearly dependent, otherwise the section 
$\Gamma \subset F_u$, corresponding to $\Ff_u \to \!\! \to L$,  would not be li (cf. \eqref{eq:isom2} and Theorem \ref{C.F.VdG}-(ii)).  So $s_1= \lambda s_2$. But then Lemma \ref{lem:technical} implies $u=u'$,  a contradiction.


\noindent
(2) If $L \cong\!\!\!\!\!| \,\,\; L' \in W_{L}$ (in particular, $\rho(L)>0$), sections $\Gamma \neq \Gamma'$, corresponding respectively to 
$\Ff_u \to \!\!\to L$ and $\Ff_u \to \!\! \to L'$, would be such that 
$\Gamma \sim_{alg} \Gamma'$ on $F_u$, contradicting  Theorem \ref{C.F.VdG}-(ii). 
 \end{proof}

\begin{example}\label{ex:C.F.1}  One can have loci $\vdj$ of the same dimension for different values of $\delta$. For instance, let  $j =2$, $g \geqslant 18$ and 
$d=2g+9$  in Theorem \ref{C.F.VdG}. Then  $g+5 \leqslant\delta \leqslant g+6$ are admissible values in \eqref{eq:ln2} and for both of them one has 
$2 \delta - d >0$. Now $\rho(g,7,g+5) = g-16 > g-18 = \rho(g,8,g+6) $; by \eqref{eq:Nudde} and Lemma \ref{lem:claim1} one has  
$\nu_{2g+9}^{g+5,2} = \nu_{2g+9}^{g+6,2} = 3g-17$. 
\end{example}

\begin{remark}\label{rem:claim1}  From \eqref{eq:bn}, \eqref{eq:Nudde} and Lemma \ref{lem:claim1}, for varieties $\vdj$ as in Proposition 
\ref{C.F.1} (i.e. defined by integers $d$, $\delta$ and $j$ as in Theorems \ref{LN}, \ref{C.F.VdG}, with $2\delta - d \geqslant 1$) one has   
\begin{equation}\label{eq:nudde2}
\nu_d^{\delta,j} - \rho_d^{k_j} = d (j-1) - \delta (j -2) - (g-1) (j+1). 
\end{equation}

\noindent
(1) For $j =1$,  $\nu_d^{\delta,1} - \rho_d^{k_1} = \delta - 2g+2 \leqslant 0$, since $L$ special, and equality holds if and only if $\delta$ reaches the upper-bound in \eqref{eq:ln2}. 

\noindent
(2) Otherwise, for $j \geqslant 2$, using the upper-bound in \eqref{eq:ln2} and  the fact $d < 2 \delta$, from \eqref{eq:nudde2} one gets 
$$\nu_d^{\delta,j} - \rho_d^{k_j} < \delta j - g j - g + j + 1 \leqslant 1 - j^2 <0.$$ Thus, $\vdj$ can never be dense in a regular component of $B_C^{k_j}(d)$ unless $j=1$ and $\delta = 2g-2$. 
\end{remark} 


\begin{corollary}\label{C.F.1b} Let $C$ be of genus $g \geqslant 5$ with general moduli. For any integer $d$ s.t.  $3g+1 \leqslant d \leqslant 4g-5$, 
the variety $\mathcal V_{d}^{2g-2,1}$ is dense in a regular, generically smooth component $\mathcal B \subseteq B_C^{k_1}(d)$. Moreover:

\noindent
(i) $[\Ff_u]\in  \mathcal V_{d}^{2g-2,1}$ general is stable and comes from $u \in \Ext^1(\omega_C,N)$ general, with $N \in \Pic^{d - 2g+2}(C)$ general. 
In particular, $i(\Ff_u)=1$.

\noindent
(ii) The minimal degree quotient of $\Ff_u$ is $\omega_C$, so $s(\Ff_u) = 4g-4-d >0$. 

\noindent
(iii) ${\rm Div}_{F_u}^{1,2g-2} = \{ \Gamma\}$, where $\Gamma$ is the section of $F_u = \Pp(\Ff_u)$ corresponding to $\Ff_u \to \!\!\! \to \omega_C$ 
(i.e. $\Ff_u$ is rp via $\omega_C$).  
\end{corollary}

\begin{proof} It  follows from Theorem \ref{C.F.VdG}, with $2\delta - d \geqslant 1$ and $j=1$, from  Proposition \ref{C.F.1} and from Remark \ref{rem:claim1}.  
\end{proof}

\begin{remark}\label{30.12l} 
Using Theorem \ref{LN} and Corollary \ref{cor:LN}, one can prove results similar to Proposition \ref{C.F.1} and  Corollary \ref{C.F.1b} with slightly different numerical bounds. As in Remark \ref{rem:claim1}, $\vdj$ can never be dense in a regular component  of $B_C^{k_j}(d)$, unless $\delta = 2g-2$ and $j = 1$. The  numerical bounds in this case are  $3g-3 \leqslant d \leqslant 4g-6$ with $g \geqslant 3$, hence the cases not already covered by Corollary \ref{C.F.1b} are $3g-3 \leqslant d \leqslant {\rm {min}} \{3g, 4g-6\}$.  
\end{remark}

\begin{corollary}\label{C.F.1c} Let $C$ be of genus $g \geqslant 6$ with general moduli. For  any integer $d$ s.t. $3g-3 \leqslant d \leqslant 3g$,  
the variety $\mathcal V_{d}^{2g-2,1}$ is dense in a regular, generically smooth component $\mathcal B \subseteq B_C^{k_1}(d)$. Moreover:

\noindent

\noindent
(i) $[\Ff_u] \in \mathcal V_{d}^{2g-2,1}$ general is stable and comes from $u \in \Ext^1(\omega_C,N)$ general, with $N \in \Pic^{d - 2g+2}(C)$ general (so non-special). In particular, $i(\Ff_u)=1$.

\noindent
(ii) The minimal degree quotient of $\Ff_u$ is $\omega_C$, thus $s(\Ff_u) = 4g-4-d \geqslant g-4$. 

\noindent
(iii)  ${\rm Div}_{F_u}^{1,2g-2} = \{ \Gamma\}$, where $\Gamma$ the section of $F_u = \Pp(\Ff_u)$ corresponding to $\Ff_u \to \!\!\! \to \omega_C$ (i.e. $\Ff_u$ is rp via $\omega_C$).  

\end{corollary}

\begin{proof} We need to prove that $\pi_{d, 2g-2}$ is generically injective. The proof of Lemma \ref{lem:claim1} shows that, for $[\Ff_u] \in \mathcal V_{d}^{2g-2,1}$ general, one has
$\dim(\pi_{d, 2g-2}^{-1} ([\Ff_u]) ) \leqslant \dim({\rm Div}_{F_u}^{1,2g-2})$. By construction of  $\mathcal V_{d}^{2g-2,1}$ and by \eqref{eq:Ciro410}, 
$\N_{\Gamma/F_u} \cong K_C-N$. Since $N$ is general of degree $d-2g+2$, one has $h^1(N)=0$. From Remark \ref{rem:linisol} we conclude. To prove the injectivity of $P_{\Ff_u}$ one can argue as in \cite[Lemma2.1]{CF} (we leave the easy details to the reader). 
\end{proof}


\subsubsection{{\bf Case $d = 2 \delta$}} Fom what proved in Theorem \ref{C.F.VdG}, for any integers $j \geqslant 1$, $g$ and $\delta$ as in \eqref{eq:ln2} and in 
Remark \ref{rem:C.F.VdG},  we have $$\mathcal V_{2\delta}^{\delta,j} \subseteq B_{C}^{k_j}(2\delta) \cap U_C^{ss}(2\delta).$$

\begin{lemma}\label{lem:claim2} The morphism $\Pp(\Efra_{\delta})^0 \stackrel{\pi_{2\delta, \delta}}{\longrightarrow} U_C(d)$ contracts the $\gamma$-fibres, with $\gamma$ as in \eqref{eq:ude}. Thus, $\nu_{2\delta}^{\delta,j} \leqslant g + \rho(L)$. 
\end{lemma}
\begin{proof} For any $y=(N,L) \in \ud$, $\gamma^{-1}(y) \cong \Pp$ as in \eqref{eq:PP}. For any $u \in \Pp$, one has ${\rm gr}(\Ff_u) = L\oplus N$, where ${\rm gr}(\Ff_u)$ is the graded object 
associated to $\Ff_u$ (cf. \cite[Thm. 4]{Ses}). Therefore, all elements in a $\gamma$-fibre determine $S$-equivalent bundles (cf. e.g. \cite{M, Tha}).  This implies that $\pi_{d, \delta}$ contracts any $\gamma$-fibre.  
\end{proof}

\begin{corollary}\label{C.F.2} Let $C$ be of genus $g \geqslant 5$ with general moduli. One has$$B_C^{k_1}(4g-4) = \overline{\mathcal V_{4g-4}^{2g-2,1}}.$$Thus:

\noindent
(i) $B_C^{k_1}(4g-4) $ is irreducible,  of dimension $g < \rho_{4g-4}^{k_1} = 2g-2$. In particular, it is birational to $\Pic^{2g-2}(C)$, with $B_C^{k_2}(4g-4) = \{[\omega_C \oplus \omega_C]\}$;

\noindent
(ii) $[\Ff_u] \in  B_C^{k_1}(4g-4) $ general comes from $u \in \Ext^1(\omega_C,N)$ general, with $N \in \Pic^{2g-2}(C)$ general. Hence, $i(\Ff_u)=1$.

\noindent
(iii) The minimal degree quotient of $\Ff_u$ is $\omega_C$, thus $s(\Ff_u) = 0$ and $\Ff_u$ is strictly semistable.

\noindent
(iv) ${\rm Div}_{F_u}^{1,2g-2} = \{ \Gamma\}$, where $\Gamma$ the section of $F_u = \Pp(\Ff_u)$ corresponding to $\Ff_u \to \!\!\! \to \omega_C$ (i.e. $\Ff_u$ is rp via $\omega_C$).

\end{corollary}
\begin{proof} From Theorem \ref{C.F.VdG}, the only case for $d = 4g-4$ is $j=1$ and $\delta = 2g-2$. Since $d = 2 \delta$, from \eqref{eq:ude} we have $\mathcal U \cong \Pic^{2g-2} (C) \setminus \{\omega_C\}$ and $\mathfrak{E}$ is a vector bundle of rank $g-1$ on $\mathcal U$. From Lemma 
\ref{lem:claim2}, the moduli map $\pi_{d, 2g-2}$ factors through a map from $\mathcal U$ to $B_C^{k_1}(4g-4) $, which is injective by Chern class reasons. 

Next we prove that $B_C^{k_1}(4g-4) $ is irreducible. Consider $[\Ff] $ general in a component of $B_C^{k_1}(4g-4) $; it can be presented via an exact sequence as in \eqref{eq:Fund}, with $L$ special and effective (cf. Lemma \ref{lem:1e2note}). Since $s(\Ff) =0$, then $\deg(L) = 2g-2$, i.e. $L \cong \omega_C$. Thus, we are in the image of $ \mathcal U$ to $B_C^{k_1}(4g-4)$. 

The remaining assertions  are easy to check and can be left to the reader. 
\end{proof}

Corollary \ref{C.F.2} has been proved already in   \cite[Theorems 7.2, 7.3 and Remark 7.4]{BGN}, via different techniques. Our proof is completely independent. 



\subsection{Special $N$}\label{ssec:Ns} Under the numerical assumptions of Theorem \ref{uepi}, any $N \in \Pic^{d-\delta}(C)$ is special (cf. Remark \ref{rem:uepib}-(2)) and,  for $u \in \Ext^1(L,N)$ general, $\partial_u$ is surjective (cf. Remark \ref{rem:uepib}-(3)). Hence $i(\Ff_u) = h^1(L) = j$. We have: 

\begin{proposition}\label{C.F.3} Assumptions as in Theorem \ref{uepi}. For any integers $j,\;\delta$ and $d$ therein, there exists an irreducible component 
$\mathcal B \subseteq B_C^{k_j}(d)$ such that: 

\noindent
(i) $\vdj \subseteq \mathcal B$; 

\noindent
(ii) $\mathcal B \cap U_C^s(d) \neq \emptyset$;

\noindent
(iii) For $[\Ef] \in \mathcal B$ general, $\Ef$ is stable, with $s(\Ef) \geqslant g-\epsilon$ and $\epsilon$ as in Theorem \ref{uepi}.  
The minimal degree quotients of $\Ef$ as well as the minimal degree sections of $\Pp(\Ef)$ are as in (iii) and (iv) 
of Theorem \ref{uepi}. In particular, $L$ is of minimal degree if and only if $ d = 2 \delta - g$.

\noindent
(iv) If moreover $d \geqslant \delta + g - 3$ (so $\delta \geqslant 2g-3$), then  $\mathcal B$ is also regular and generically smooth. 

\end{proposition}

\begin{proof}  Assertions (i), (ii) and (iii) follow from Theorem \ref{uepi}, the map \eqref{eq:pde} and the fact that the Segre invariant is lower-semicontinuous (cf. e.g. \cite[\S\;3]{LN}).

To prove (iv), we argue as in \cite[Lemma 2.1]{CF}. 
Take $\Ff_0 = L \oplus N$, with $N \in \Pic^{d-\delta}(C)$ general. Then, $N$ is non-effective and $1 \leqslant h^1(N) \leqslant 2$. The Petri map $P_{\Ff_0}$ decomposes as $\mu_0(L) \oplus \mu$, where $\mu_0(L)$ 
is the Petri map of $L$ as in \eqref{eq:Petrilb} and $\mu$ is as in \eqref{eq:mu}. Since $C$ has general moduli, $\mu_0(L)$ is injective 
(cf. \cite[(1.7), p. 215]{ACGH}).  The injectivity of $\mu$ is immediate when $h^1(N) = 1$ (cf. the proof of Theorem \ref{thm:mainext1}). When 
$h^1(N) =2$, the generality of $N$ implies that $|K_C-N|$ is a base-point-free pencil 
so the injectivity of $\mu$ follows from the base-point-free pencil trick, since $h^0(N- (K_C-L))=0$ (because $K_C-L$ is effective and $N$ non-effective).  
By semicontinuity on the elements of $\Ext^1(L,N)$ and the fact that $\vdj \subseteq \mathcal B$, the Petri map  $P_{\Ef}$ is injective. One concludes by Remark \ref{rem:BNloci}.  
\end{proof}

\begin{remark}\label{rem:C.F.3} Computing  $\dim(\Pp(\Efra_\delta)) - \rho_d^{k_j}$ one finds the right-hand-side of  \eqref{eq:nudde2}.  Since $d < 2 \delta$ (see \eqref{eq:uepi3}),  as in 
Remark \ref{rem:claim1}-(2) one sees that $\dim(\Pp(\Efra_\delta)) - \rho_d^{k_j} <0$,  unless $j =1$ and $\delta = 2g-2$,  in which case 
$\dim(\Pp(\Efra_\delta)) - \rho_d^{k_j} =0$. As in Lemma \ref{lem:claim1}, we see that $\pi_{d, 2g-2}$ is generically injective.  Thus, with notation as in \eqref{eq:Nudde}, one has $\nu_d^{\delta,j} \geqslant \rho_d^{k_j}$ only if $j=1$, $\delta = 2g-2$ and $N \in \Pic^{d-\delta}(C)$ is general, in which case $\nu_d^{2g-2,1} = \rho_d^{k_1}$. 
\end{remark}

\begin{corollary}\label{C.F.3b} Let $C$ be of genus $g \geqslant 3$ with general moduli. For  any integer $d$ such that $2g-2 \leqslant d \leqslant 3g-4$, one has $\nu_d^{2g-2,1}= \rho_d^{k_1} = 6g-6-d$. Moreover:   

\noindent
(i) $[\Ff_u] \in \mathcal V_{d}^{2g-2,1}$ general is stable and comes from $u \in \Ext^1(\omega_C,N)$ general, with 
$N \in \Pic^{d - 2g+2}(C)$ general (hence special, non-effective). In particular, $i(\Ff_u)=1$.

\noindent
(ii) If $ 3g-5\leqslant d \leqslant 3g-4$, then $\mathcal V_{d}^{2g-2,1}$ is  dense in a regular, generically smooth component of  
$B_C^{k_1}(d)$.

\noindent
(iii) $s(\Ff_u) = g-\epsilon$, with $\epsilon$ as in Theorem \ref{uepi}.  Quotients of minimal degree of $\Ff_u$ (equivalently sections of minimal degree on $F_u = \Pp(\Ff_u)$) are those described in Theorem \ref{uepi}-(111) and (iv). In particular, they are li sections. 

\noindent
(iv) The canonical section $\Gamma \subset F_u$ is the only special section; it is lsu and asu but not ai. Moreover,   
it is of minimal degree only when $d = 3g-4$.

\noindent
(v)  $\Ff_u$ is rsp but not rp via $\omega_C$. 
\end{corollary}

\begin{proof} (i), (ii) and (iii) follow from Theorem \ref{uepi}, Proposition \ref{C.F.3} and  Remarks \ref{rem:BNloci}, \ref{rem:C.F.3}. Sections of minimal degree are li (see the proof of Proposition \ref{prop:lem4}). 

As for (iv) and (v), from Serre duality and the fact that $\Ff_u$ is of rank-two with $\det(\Ff_u) = \omega_C \otimes N$, one has 
\begin{equation}\label{eq:casaciro2}
h^0(\Ff_u \otimes N^{\vee}) = h^1(\Ff_u^{\vee} \otimes \omega_C \otimes N) = h^1(\Ff_u).
\end{equation} Since $i(\Ff_u) = 1$, from \eqref{eq:isom2} $\Gamma$ is li. Since $N$ is special and non-effective, from 
\eqref{eq:Ciro410}, $ {\rm Div}^{1,2g-g}_{F_u}$ is smooth, of dimension $3g-3-d \geqslant 1$ at $\Gamma$. 
Thus, $\Gamma$ is not ai but, since $W_{2g-2}^{g-1} (C) = \{\omega_C\}$, it is asu (see the proof of Proposition \ref{prop:lem4} and 
Remark \ref{rem:C.F.3}). For the same reason, from Theorem \ref{uepi}-(iv), the only possibility for $\omega_C$ to be a minimal quotient is  $d = 3g-4$. Finally, 
the fact that $\Gamma \subset F_u$ is the only special section follows from  Remark \ref{rem:C.F.3}.   \end{proof}

Recall that, when $N \in \Pic^{d-\delta}(C)$ is special and $L \in W^{\delta-g+j}_{\delta}(C)$ is a smooth point, assumptions as in Theorem \ref{thm:mainext1} imply that $\partial_u$ is surjective for $u \in \Ext^1(L,N)$ general (cf. Corollary \ref{cor:mainext1}), and so $i(\Ff_u) = j$. 

Therefore, to have $i(\Ff_u) >j$, we are forced to use degeneracy loci described in \eqref{eq:wt}. To do this 
let $y = (N,L)$ be general in $\ud$, respectively in $\zd$, when $N$ is non-effective, respectively when it is effective (recall notation as in \eqref{eq:ude}). Set $\Pp(y) := \gamma^{-1}(y) \cong \Pp$. Take numerical assumptions as in Remark \ref{unepine}, respectively in Remark \ref{unepie}, according to  $N$ is respectively non-effective or effective. 

With notation as in \eqref{eq:Lahat}, for any good component $\widehat{\Lambda}_t(y)  \subseteq \widehat{W}_t(y) \subset \Pp(y)$ we have 
$$\emptyset \neq \widehat{\mathcal W}_t^{\rm {Tot}} \subset \Pp(\Efra_\delta),$$where a point in $\widehat{\mathcal W}_t^{\rm {Tot}}$ corresponds to the datum of a pair $(y, u)$, with $y = (N,L)$  and $u \in \widehat{W}_t(y)$. Any  irreducible component of $\widehat{\mathcal W}_t^{\rm {Tot}} $ has dimension at least $\dim (\Pp(\Efra_\delta)) - c (\ell,r,t)$ (where $c(\ell,r,t)$ as in \eqref{eq:clrt} and where $\dim (\Pp(\Efra_\delta))$ as in  \eqref{eq:yde2}). 
From the generality of $y$, for any good component  $\widehat{\Lambda}_t(y) $, 
we have an irreducible component $$\widehat{\Lambda}_t^{\rm {Tot}} \subseteq \widehat{\mathcal W}_t^{\rm {Tot}} \subset \Pp(\Efra_\delta) $$such that 
\begin{itemize}
\item[(i)] $\widehat{\Lambda}_t^{\rm {Tot}}$ dominates $\ud$ (resp., $\zd$); 
\item[(ii)] $\dim(\widehat{\Lambda}_t^{\rm {Tot}}) = \dim (\Pp(\Efra_\delta)) - c (\ell,r,t)$; 
\item[(iii)] for $(y, u) \in \widehat{\Lambda}_t^{\rm {Tot}}$ general, $\cork(\partial_u)=t$; 
\item[(iv)] if $\lambda:= \gamma|_{\widehat{\Lambda}_t^{\rm {Tot}}}$, for $y$ general one has $\lambda^{-1}(y) = \widehat{\Lambda}_t(y) $. 
\end{itemize}

\begin{definition}\label{def:goodtot} Any component $\widehat{\Lambda}_t^{\rm {Tot}}$ satisfying  (i)-(iv) above will be called a {\em (total) good component} of $\widehat{\mathcal W}_t^{\rm {Tot}}$. 
\end{definition}

We set  
\begin{equation}\label{eq:nuddet}
\vdjt := {\rm Im} \left(\pi_{d, \delta}|_{\widehat{\Lambda}_t^{\rm {Tot}}}\right) \subseteq B_C^{k_{j+t}}(d) \;\;\;\; {\rm and} \;\;\;\; \nu_d^{\delta,j,t} := \dim(\vdjt).
\end{equation} Two cases have to be discussed, according to the effectivity of $N$.

\noindent
\subsubsection{{\bf $N$ non-effective}}\label {ssec:noneff} With assumptions as in Remark  \ref{unepine}, $N$ can be taken general in $\Pic^{d-\delta}(C)$; the general bundle in $\vdjt $ is stable (by Theorem \ref{unepi} and by the open nature of stability). 
For brevity sake, set
\begin{equation}\label{eq:fojt}
\varphi_0 (\delta, j,t) :=\dim (\widehat{\Lambda}_t^{{\rm Tot}}) -\rho_d^{k_{j+t}}  =  d (j-1) - \delta (j-2) - (g-1) (j+1) + j t,
\end{equation}which therefore  takes into account the expected dimension of the general fibre of 
$\pi_{d,\delta}|_{\widehat{\Lambda}_t^{{\rm Tot}}}$ and the codimension of its image in a regular component of 
$B_C^{k_{j+t}}(d)$).

One has $\varphi_0 (\delta, j,t) \geqslant \nu_d^{\delta,j,t} - \rho_d^{k_{j+t}} $ with equality  if and only if 
$\pi_{d, \delta}|_{\widehat{\Lambda}_t^{{\rm Tot}}}$ is generically finite. Thus, from Remark \ref{rem:BNloci},  it is clear that 
$\vdjt$ cannot fill up a dense subset of a component of $B_C^{k_{j+t}}(d) $ if $\varphi_0(\delta, j,t) <0$; in other words, the negativity 
of $\varphi_0(\delta, j,t)$ gives numerical obstruction to describe the general point of a (regular) component of $B_C^{k_{j+t}}(d) $. 

\noindent
$\bullet$ For $j =1$, one has 
\begin{equation}\label{eq:fo1t}
\varphi_0 (\delta, 1,t)  = \delta - 2g + 2 + t. 
\end{equation}

\noindent
$\bullet$ When $j \geqslant 2$, from Remark \ref{unepine} and arguing as in Remark \ref{rem:claim1}, one gets 
\begin{equation}\label{eq:fojtb}
\varphi_0 (\delta, j,t) \leqslant j (t-j). 
\end{equation}Thus, $\vdjt$ never fills up a dense subset of a component of $B_C^{k_{j+t}}(d)$ as soon as $j >t \geqslant 1$.

\subsubsection {{\bf $N$ effective}} With assumptions as in Remark \ref{unepie}, $N$ is general in $W_{\rho(N)}$. From the second equality in 
\eqref{eq:yde2}, for any $n \geqslant 1$, one puts
\begin{equation}\label{eq:fnjt}
\varphi_n (\delta, j,t) :=\dim (\widehat{\Lambda}_t^{{\rm Tot}}) -\rho_d^{k_{j+t}}  =  \varphi_0(\delta, j,t) - n (r-t), 
\end{equation}where $\varphi_0(\delta,j,t)$ as in \eqref{eq:fojt} above. 

\begin{remark}\label{rem:conpar2} For a total  good component $\widehat{\Lambda}_t^{{\rm Tot}}$ and for  $(L,N,u) \in \widehat{\Lambda}_t^{{\rm Tot}}$ general, one has $n (r-t) = h^0(N) \; {\rm rk}(\partial_u)$. Hence, $n (r-t)$ is non-negative and it is zero if and only if $r=t$, i.e. 
$\partial_u$ is the zero map. Therefore, $\varphi_n(\delta, j,t) \leqslant \varphi_0(\delta, j,t)$ and equality holds if and only if $r=t$. The possibility for a $\vdjt$ to fill up a dense subset of a component of $B_C^{k_{j+t}}(d)$ can be discussed as in \S~ \ref {ssec:noneff}. 
\end{remark} 

By definition of $\nu_{d}^{\delta,j,t}$, it is clear that 
$\nu_{d}^{\delta,j,t} - \rho_d^{k_{j+t}} \leqslant \varphi_n(\delta, j,t)$, for any $n \geqslant 0$. Thus a necessary condition  for $\nu_d^{\delta,j,t} \geqslant \rho_d^{k_{j+t}}$, i.e. for $\mathcal V_d^{\delta,j,t}$ to have at least the dimension of a regular component of 
$B_C^{k_{j+t}}(d)$, is $\varphi_n(\delta, j,t) \geqslant 0$.

Next proposition easily follows.

\begin{proposition}\label{C.F.4} Assumptions as in Theorem \ref{unepi} (more precisely, either as in Remark 
\ref{unepine}, when $N$ is non-effective, or as in Remark \ref{unepie}, when $N$ is effective). Then for any integers $j,\;\delta$ and $d$ therein, there exists an irreducible component $\mathcal B \subseteq B_C^{k_{j+t}}(d)$ such that: 

\noindent
(i) $\vdjt \subseteq \mathcal B$; 

\noindent
(ii) For $[\Ef] \in \mathcal B$ general, $s(\Ef) \geqslant g- c(\ell,r,t)-\epsilon \geqslant 0$, 
where $c(\ell,r,t)$ as in \eqref{eq:clrt} and $\epsilon \in \{0,1\}$ such that $d+g-c(\ell,r,t) \equiv \epsilon \pmod{2}$;  

\noindent
(iii) $\mathcal B \cap U_C^s(d) \neq \emptyset$, if $g - c(\ell,r,t)-\epsilon >0$.  
\end{proposition}

\begin{remark}\label{rem:17lug} In order to estimate $\nu_d^{\delta,j,t}$, one has to estimate the dimension of the general fibre 
of the map $\pi_{d, \delta}$ restricted to a total good component $\widehat{\Lambda}_t^{\rm {Tot}}$. Thus, if for 
$[\Ff] \in \mathcal V_d^{\delta,j,t}$ general we put for simplicity $f_{\Ff} := \dim\left(\pi_{d, \delta}|_{\widehat{\Lambda}_t^{Tot}}^{-1} ([\Ff])\right)$, a rough estimate is  
\begin{equation}\label{eq:aiutoB}
f_{\Ff} \leqslant a_{F}(\delta), 
\end{equation}where $F = \Pp(\Ff)$ and $a_{F}(\delta)$ the dimension of the scheme of special unisecants of degree $\delta$ on $F$ as in \eqref{eq:aga}. 
\end{remark}

\begin{remark}\label{rem:bohb}  
Assume $j=1$ in Proposition \ref{C.F.4}.  


\noindent
(1) When $N \in \Pic^{d-\delta}(C)$ is general, assumptions as in Remark \ref{unepine}  give $\delta \leqslant 2g-2$ and $N$ non-effective, for any $t \geqslant 1$. The only case to consider is  therefore 
$\varphi_0 (\delta, 1,t)$. A necessary condition for 
$\mathcal{V}_d^{\delta, 1,t}$ to have dimension at least $\rho_d^{k_{j+t}}$ is 
$\varphi_0(\delta, 1,t) \geqslant 0$, i.e.  $\delta \geqslant 2g-2 - t$ (cf. \eqref{eq:fo1t}). Thus: 


\noindent
$\bullet$ when $\delta = 2g-2 - t$, then $L = \omega_C(-D_t)$, with $D_t \in C^{(t)}$, $t<g$,  imposing independent conditions to $|\omega_C|$. Since  $\varphi_0(\delta, 1,t)=0$, for $D_t \in C^{(t)}$  
general, the estimate \eqref{eq:aiutoB} and a parameter count suggest that for $[\Ff] \in \mathcal V_d^{2g-2-t,1,t} $ general one has $a_{F}(2g-2-t) = 0$, i.e. $\Ff$  is rsp via $\omega_C(-D_t)$, and $\mathcal{B}  = \overline{\mathcal{V}_d^{\delta, 1,t}}$ is regular.


\noindent
$\bullet$ to the opposite, when $\delta = 2g-2$, then $L = \omega_C$.  Let $[\Ff] \in \mathcal{V}_d^{2g-2, 1,t}$ be general, and let 
$\Gamma \subset F = \Pp(\Ff)$ be the canonical section corresponding to $\Ff \to\!\! \to \omega_C$. By definition of 
$\mathcal{V}_d^{2g-2, 1,t}$, $\Ff = \Ff_v$ for $v \in \Lambda_t \subset \Ext^1(N, \omega_C)$ general in a good component. 
By \eqref{eq:isom2} and  \eqref{eq:casaciro2}, one has 
$\dim(|\Oc_{F}(\Gamma)|) = t$. Thus,  $[\Ff] \in \mathcal V_d^{2g-2,1,t} $ general is  not rsp via $\omega_C$, since the general fibre of $\pi_{d, 2g-2}|_{\widehat{\Lambda}_t^{{\rm Tot}}}$ has dimension at least $t$. 


It is therefore natural to expect that the component $\mathcal B$ in Proposition \ref{C.F.4} is such that
$$\mathcal B = \overline{\mathcal V_d^{2g-2,1,t}} = \overline{\mathcal V_d^{2g-3,1,t}} = 
\cdots = \overline{\mathcal V_d^{2g-2-t,1,t}},$$where $[\Ff] \in \mathcal B$ general is rsp only 
when $[\Ff]$ is considered as element in $\mathcal V_d^{2g-2-t,1,t}$.


\noindent
(2) One may expect something similar when $j=1$ and $N$ effective general  in $W_{\rho(N)}$. In this case, $\varphi_n(\delta, 1,t) \geqslant 0$ gives $ \delta \geqslant 2g-2 + rn - t (n+1) $ whereas, from the first line of bounds on $\delta$ in Remark \ref{unepie}, we get $\delta \leqslant {\rm min} \{ 2g-2, g-2 + r-t + \frac{g-\epsilon}{t} \}$. A necessary condition for $\nu_d^{\delta,j,t} \geqslant \rho_d^{k_{j+t}}$ is therefore  
\begin{equation}\label{eq:8.1}
rn - t (n+1) <0,  
\end{equation}otherwise either $L$ would be non special, contradicting Lemma \ref{lem:1e2note}, or 
$L \cong \omega_C$, so $\Ff_v$ would be not rsp as in (1) above. 
In the next section, we will discuss these questions. 
\end{remark}


\section{Low speciality, canonical determinant}\label{S:BND}
In this section we apply results in \S's\;\ref{S:Nns}, \ref{S:Ns} and \ref{S:PSBN} to describe Brill-Noether loci of vector bundles with canonical determinant and Brill-Noether loci of vector bundles of fixed degree $d$ and low speciality $i \leqslant 3$ on a curve $C$ with general moduli. In particular, the more general analysis discussed in the previous sections allows us to determine 
rigidly specially presentation of the general point of  irreducible components arising from constructions in \S\;\ref{S:PSBN}, 

From now on, for any integers $g \geqslant 3$, $i \geqslant 1$ and $2g-2 \leqslant d \leqslant 4g-4$, we will set 
\begin{equation}\label{eq:Btilde}
\widetilde{B_C^{k_i}}(d) := 
\left\{
\begin{array}{cl}
B_C^{k_i}(d) & \mbox{if either $d$ odd or $d=4g-4$} \\
B_C^{k_i}(d)\cap U_C^s(d) & \mbox{otherwise} 
\end{array}
\right.
\end{equation}




\subsection{Vector bundles with canonical determinant}\label{SS:low} Given an integer $d$ and any $\xi \in \Pic^d(C)$, there exists 
the {\em moduli space of (semi)stable, rank-two vector bundles with fixed determinant $\xi$}. Following 
\cite{Muk2, Muk}, we denote it by $M_C(2, \xi)$ (sometimes a different notation is used, see e.g. 
\cite{Ses,Be,BeFe,TB000,Oss,Voi,Beau2,LNP}). 

The scheme $M_C(2, \xi)$ is defined as the fibre over $\xi \in \Pic^d(C)$ of the {\em determinantal map} 
\begin{equation}\label{eq:det}
U_C(d) \stackrel{\rm det}{\longrightarrow} \Pic^d (C).
\end{equation}For any $\xi \in \Pic^d(C)$, $M_C(2, \xi)$ is smooth, irreducible, of dimension $3g-3$ (cf. \cite{NR1,Ses}). 

Brill-Noether loci can be considered in $M_C(2, \xi)$. Recent results for arbitrary $\xi$ are given in \cite{Oss,Oss2,LaNw}. A case which has been particularly studied (for its connections with Fano varieties) is    $M_C(2, \omega_C)$. Seminal papers on the subject are 
\cite{BeFe, Muk2}; other important results are contained in \cite{TB000,Voi,IVG,LNP}. If $[\Ff] \in M_C(2, \omega_C)$, Serre duality gives 
\begin{equation}\label{eq:net1}
i(\Ff) = k_i(\Ff) := k.
\end{equation} For $[\Ff] \in M_C(2, \omega_C) \subset U_C(2g-2)$, the Petri map $P_{\Ff}$ in \eqref{eq:petrimap} 
splits as $P_{\Ff} = \lambda_{\Ff} \oplus \mu_{\Ff}$, where 
$$\lambda_{\Ff} : \bigwedge^2 H^0(\Ff) \to H^0(\omega_C) \quad \mbox {and}\quad 
\mu_{\Ff} : {\rm Sym}^2 (H^0(\Ff))  \to H^0({\rm Sym}^2 (\Ff));$$the latter is called the {\em symmetric Petri map}.

For $[\Ff] \in  M_C(2, \omega_C)$ general, one has $k = 0$ (cf. \cite[\S\,4]{Muk}, after formula (4.3)). For any $k \geqslant 1$, one sets$$M_C^k(2, \omega_C) := \{ [\Ff] \in M_C(2, \omega_C) \; | \; h^0(\Ff) = h^1(\Ff) \geqslant k\}$$which is called the $k^{th}$-{\em Brill-Noether locus} in $M_C(2, \omega_C)$. In analogy with \eqref{eq:Btilde}, we set 
$$\widetilde{M_C^k}(2, \omega_C) := M_C^k(2, \omega_C) \cap U_C^s(2g-2).$$By  \cite[Prop.\,1.4]{Muk}, \cite[\S\,2]{BeFe} and \eqref{eq:net1}, one has
$${\rm expcodim}_{M_C(2, \omega_C)} (\widetilde{M_C^k}(2, \omega_C))  = \frac{k(k+1)}{2} \leqslant k^2 = i(\Ff) k_i(\Ff).$$Similarly to  $\widetilde{B_C^{k_i}}(d)$, if $[\Ff]\in  
\widetilde{M_C^k}(2, \omega_C)$, then $\widetilde{M_C^k}(2, \omega_C)$ is smooth and regular (i.e. of the expected dimension) at 
$[\Ff]$ if and only if $\mu_{\Ff}$ is injective (see \cite{BeFe,Muk2,Muk}).  

Several basic questions on $\widetilde{M_C^k}(2, \omega_C)$, like non-emptiness, irreducibility, etc.,  are still open.  A  description of these bundles in terms of extensions (as we do here) is available only for some $k$ on $C$ general of genus $g \leqslant 12$ (cf.  \cite[\S\,4]{Muk}, \cite{BeFe}). Further existence results are contained in  \cite{TB1,LNP}. On the other hand, if one assumes $ [\Ff]\in M_C^k(2, \omega_C)$,  
injectivity of $\mu_{\Ff}$ on $C$ general of genus  $g \geqslant 1$ has been proved in \cite{TB000} (cf.  \cite{Beau2} for  $k < 6$  with a  different approach).

\subsection{Case $i=1$}\label{ssec:1} In this case $\rho_d^{k_1} = 6g-6 -d$. Using notation and results as in \S\,\ref{S:PSBN}, we get:

\begin{theorem}\label{i=1} Let $C$ be of genus $g \geqslant 5$, with general moduli. For  any integer $d$ s.t. $2g-2 \leqslant d \leqslant 4g-4$, 
$$\widetilde{B_C^{k_1}}(d)= \overline{\mathcal V_d^{2g-2,1}},$$as in Corollaries \ref{C.F.1b}, \ref{C.F.1c}, \ref{C.F.2} and \ref{C.F.3b}. 
In particular, 

\noindent
(i)  $\widetilde{B_C^{k_1}}(d)$ is non-empty, irreducible. For $2g-2 \leqslant d \leqslant 4g-5$ 
it is regular, whereas $\dim(\widetilde{B_C^{k_1}}(4g-4)) = g < \rho_{4g-4}^{k_1} = 2g-2$. 

\noindent
(ii) For $3g-5 \leqslant d \leqslant 4g-4$, $\widetilde{B_C^{k_1}}(d)$ is  generically smooth. 

\noindent
(iii) $[\Ff] \in \widetilde{B_C^{k_1}}(d)$ general is stable for $2g-2 \leqslant d \leqslant 4g-5$, and strictly semistable for $d=4g-4$, fitting in a (unique) sequence$$ 0 \to N \to \Ff \to \omega_C \to 0,$$where $N \in \Pic^{d-2g+2}(C)$ is  general, the coboundary map  is surjective and $i(\Ff)=1$.

\noindent
(iv) For $3g-4 \leqslant d \leqslant 4g-4$ and $[\Ff] \in \widetilde{B_C^{k_1}}(d)$ general, one has  $s(\Ff) = 4g-4-d$, the quotient of minimal degree being $\omega_C$.
The section $\Gamma \subset F= \Pp(\Ff)$ corresponding to $\Ff \to \!\!\! \to \omega_C$ is the only special section of $F$. Moreover: 
\begin{itemize} 
\item for $d \geqslant 3g-3$,  $\Gamma$ is ai, 
\item for $d = 3g-4$, $\Gamma$ is lsu and asu but not ai. 
\end{itemize}

\noindent
(v) For $2g-2 \leqslant d \leqslant 3g-5$ and $[\Ff] \in \widetilde{B_C^{k_1}}(d)$ general, one has  $s(\Ff) = g-\epsilon$, with $\epsilon \in \{0,1\}$ such that $d+g \equiv \epsilon \pmod{2}$. The section $\Gamma \subset F$ is the only special section; it is asu but not ai.  Moreover, $\Gamma$ is not of minimal degree; indeed:

\begin{itemize}
\item when $d+g$ is even, minimal degree sections of $F$ are li sections of degree $\frac{d+g}{2}$ s.t. $\dim( {\rm Div}_{F}^{1,\frac{d+g}{2}}) = 1$;

\item when $d+g$ is odd, minimal degree sections are li of degree $\frac{d+g-1}{2}$ and  $\dim( {\rm Div}_{F}^{1,\frac{d+g-1}{2}}) \leqslant  1$.

\end{itemize}

\noindent
(vi) In particular, for $2g-2 \leqslant d \leqslant 4g-4$, $[\Ff] \in \widetilde{B_C^{k_1}}(d)$ general  is rp via $\omega_C$.  
\end{theorem}

\begin{proof} All the assertions, except the irreducibility,  follow 
from Corollaries \ref{C.F.1b}, \ref{C.F.1c}, \ref{C.F.2} and \ref{C.F.3b}. 
For $d = 4g-4$ irreducibility has been proved in Corollary \ref{C.F.2}. Thus, we focus on cases 
$ 2g-2 \leqslant d \leqslant 4g-5$.  

Let us consider an irreducible component $\mathcal B \subseteq \widetilde{B_C^{k_1}}(d)$. From Lemma \ref{lem:1e2note}, $[\Ff] \in \mathcal B$ general is as in  \eqref {eq:Fund}, with $h^1(L) = j \geqslant 1$ and $L$ of minimal degree among special,  effective quotient line bundles. 
Moreover $\dim(\mathcal B) \geqslant \rho_d^{k_1}$ (cf. Remark \ref{rem:BNloci}). Two cases have to be considered. 

\smallskip

\noindent
(1) If $i(\Ff) =1$, then $j=1$ (notation as in 
\eqref{eq:exthyp1}, \eqref{eq:exthyp}) and $\partial: H^0(L) \to H^1(N)$ is surjective. In particular $\ell \geqslant r$. If $r=0$ then we are in cases of Corollaries \ref{C.F.1b}, \ref{C.F.1c}, and $\mathcal B =  \overline{\mathcal V_d^{2g-2,1}}$. If $r>0$,  
as in Remark \ref{rem:C.F.3}  one has $$0 \leqslant \dim(\Pp(\Efra_\delta)) - \dim(\mathcal{B}) \leqslant  \delta-2g+2$$(cf. \eqref{eq:nudde2}). Hence 
$\delta = 2g-2$ and  $\mathcal B =  \overline{\mathcal V_d^{2g-2,1}}$ as in Corollary \ref{C.F.3b}. 

\smallskip

\noindent
(2) Assume $i(\Ff) = i > 1$. As in Remarks \ref{rem:claim1}, \ref{30.12l}, \ref{rem:C.F.3} one has $L = \omega_C$. Thus $i > 1$ forces  $r \geqslant {\rm cork}(\partial) = i-1 > 0 $. Recalling \eqref{eq:lem3note} and \eqref{eq:PP}, one has $\dim(\Pp) = 5g-6-d$. Therefore, $\mathcal B$ must be regular and 
$\Ff$ corresponds to the general point of $\Ext^1(\omega_C, N)$, with $N \in \Pic^{d-2g+2}(C)$ general, so non-effective. In particular, one has 
$2g-2 \leqslant d \leqslant 3g-4$ and $ r = 3g-3-d$. On the other hand, since $$\ell = g, \;  1 \leqslant r \leqslant g-1, \; 2g -1 \leqslant m = 5g-5-d \leqslant 3g-3$$we are in the hypotheses of Corollary \ref{cor:mainext1}, hence ${\rm cork}(\partial)= 0$, a contradiction.    
\end{proof}


\begin{remark}\label{rem:i=1} (1) Theorem \ref{i=1} gives alternative proofs of results in \cite{Sun,Lau,BGN} for the rank-two case. It provides in addition a description of the general point of $\widetilde{B_C^{k_1}}(d) \cong \widetilde{B_C^1}(4g-4-d)$, for any $2g-2 \leqslant d \leqslant 4g-4$. The same description is  given in \cite{Ballico1}, with a different approach, i.e. using {\em general negative elementary transformations} as in \cite{Lau}. In terms of scrolls of speciality $1$, partial classification are given also in \cite[Theorem 3.9]{GP3}.

\noindent
(2) As a consequence of Theorem \ref{i=1}, one observes that the Segre invariant $s$ does not stay constant on a component of the Brill-Noether locus. For example, the general element of $\widetilde{B_C^{k_1}}(4g-7)$ has  $s=3$ and  $i = 1$; on the other hand, in Theorem \ref{C.F.VdG}, we constructed vector bundles in $ \mathcal V_{4g-7}^{2g-3,1} \subset \widetilde{B_C^{k_1}}(4g-7)$ with $s = i = 1$. The minimal special quotient of the latter vector bundles is the canonical bundle minus a point, whereas for the general vector bundle in $\widetilde{B_C^{k_1}}(4g-7)$ is the canonical bundle.  

\noindent
(3) From the proof of Theorem \ref{i=1}, for $d \leqslant 4g-3$, the map $\pi_{d,2g-2}$ is birational onto $\widetilde{B_C^{k_1}}(d) = B_C^{k_1}(d)$, i.e. $B_C^{k_1}(d)$ is uniruled.  
\end{remark}

\begin{theorem}\label{prop:M12K} Let $C$ be of genus $g \geqslant 5$, with general moduli. Then
$\widetilde{M_C^1}(2, \omega_C) \neq \emptyset$. Moreover, there exists an irreducible component which is  

\begin{itemize}
\item[(i)] generically smooth 
\item[(ii)] regular (i.e. of dimension $3g-4$), and 
\item[(iii)] its general point $[\Ff_u]$ comes from $u \in \Pp(\Ext^1(\omega_C, \Oc_C))$ general. In particular, 
$s(\Ff_u) = g - \epsilon$, where $\epsilon \in \{0,1\}$ such that $g \equiv \epsilon \pmod{2}$.    
\end{itemize}
\end{theorem} 

\begin{proof} Take $u \in \Pp(\Ext^1(\omega_C, \Oc_C))$ general. With notation as in \eqref{eq:exthyp1}, \eqref{eq:exthyp} one has 
$$\ell = r = g, \; {\rm and} \;  m = 3g-3 \geqslant \ell+1.$$Thus, from Corollary \ref{cor:mainext1} and from \eqref{eq:net1}, 
$h^0(\Ff_u) = h^1(\Ff_u) =1$. 
From \eqref{eq:lem3note}, $\dim( \Pp(\Ext^1(\omega_C, \Oc_C)) = 3g-4$. Thus,  $\Ff_u $  stable with 
$s(\Ff_u) = g-\epsilon$ follows from Proposition \ref{prop:LN}, with $\sigma = g-\epsilon$.  This shows that $\widetilde{M_C^1}(2, \omega_C) \neq \emptyset$. 

Since $\bigwedge^2 H^0(\Ff_u) = (0)$, $\mu_{\Ff_u}$ is injective if and only if $P_{\Ff_u}$ is. On the other hand, 
one has $H^0(\Ff_u) \otimes H^0(\omega_C \otimes \Ff_u^{\vee}) \cong \C$. 
Therefore, one needs to show that $P_{\Ff_u}$ is not the zero-map. This follows by limit of $P_{\Ff_u}$ when $u$ tends to $0$, so that $\Ff_0 = \Oc_C \oplus \omega_C$: then the limit of $P_{\Ff_u}$  is the map $H^ 0(\Oc_C)\otimes H^ 0(\Oc_C)\to H^ 0(\Oc_C)$.

To get (i)-(iii) at once,  one observes that $\pi_{2g-2, 2g-2}|_{\Pp(\Ext^1(\omega_C, \Oc_C))} $ is generically injective, since the exact sequence 
$$ 0 \to \Oc_C \to \Ff_u \to \omega_C \to 0$$is unique: indeed, the surjection $\Ff_u \to\!\!\to \omega_C$ is unique and $h^0(\Ff_u) = 1$ (cf. \eqref{eq:isom2} and computations as in \eqref{eq:casaciro2}), 
moreover, by Lemma \ref{lem:technical}, two general vector bundles in $\Pp(\Ext^1(\omega_C, \Oc_C))$ cannot be isomorphic.   
\end{proof}

\begin{remark}\label{rem:net2} (1) For a similar description, cf. \cite{BeFe}. 
Generic smoothness for components of $\widetilde{M_C^1}(2, \omega_C)$ follows also from results in \cite{TB000,Beau2}.

\noindent
(2) From Theorem \ref{i=1}, $[\Ff] \in \widetilde{B_C^{k_1}}(2g-2)$ general fits in a sequence
$0 \to \eta \to \Ff \to \omega_C \to 0$, with $\eta \in \Pic^0(C)$ general. Hence  the map $\widetilde{d} := {\rm det}|_{\widetilde{B_C^{k_1}}(2g-2)}$ is dominant. Since $\rm{Pic}^{2g-2}(C)$ and $\widetilde{B_C^{k_1}}(2g-2)$ are irreducible and generically smooth, then $\widetilde{d}^{-1}(\eta) = \widetilde{M_C^1}(2, \omega_C \otimes \eta)$ is equidimensional and each component is  generically smooth. Theorem \ref{prop:M12K} yields that in this situation each component of $ \widetilde{M_C^1}(2, \omega_C \otimes \eta)$ has dimension $3g-4$ (equal to the expected dimension). This agrees with \cite[Theorem 1.1]{Oss}.  
\end{remark}



\subsection{Case $i =2$}\label{ssec:2} In this case, $\rho_d^{k_2} = 8g- 11 - 2d$.  

\begin{theorem}\label{i=2} 

Let $C$ be of genus $g \geqslant 3$, with general moduli. For any integer $d$ s.t. $2g-2 \leqslant d \leqslant 3g-6$, one has $\widetilde{B_C^{k_2}}(d) \neq \emptyset$.

\noindent
(i)  $\overline{\mathcal V_d^{2g-3,1,1}}$ is the unique component of $\widetilde{B_C^{k_2}}(d) $, whose  general point corresponds to a vector bundle $\Ff$ 
with $i(\Ff) = 2$. Moreover,  $\overline{\mathcal V_d^{2g-3,1,1}} = \overline{\mathcal V_d^{2g-2,1,1}}$ and it is a regular component of $\widetilde{B_C^{k_2}}(d) $. 

\noindent  
(ii) For  $[\Ff]\in \overline{\mathcal V_d^{2g-3,1,1}}$ general, one has
$s(\Ff) \geqslant 3g-4-d-\epsilon >0$ and $\Ff$ fits in an exact sequence$$0 \to N \to \Ff \to \omega_C (-p) \to 0,$$with 

\begin{itemize}
\item[$\bullet$] $p \in C$ general, 
\item[$\bullet$]  $N \in \Pic^{d-2g+3}(C)$ general, and 
\item[$\bullet$] $\Ff=\Ff_v$ and $v$ is general in the good locus  $ \mathcal W_1  \subset \Ext^1(\omega_C(-p), N)$. 
\end{itemize}

\noindent
(iii) A section  $\Gamma \subset F$, corresponding to a quotient $\Ff \to\!\!\to \omega_C(-p)$, is not of minimal degree. 
However, it is of minimal degree among special sections and it is asi but not ai (i.e. $\Ff$ is rsp but not rp via  $\omega_C(-p)$).

\noindent
(iv) For $g \geqslant 13$ and $2g+6 \leqslant d \leqslant 3g-7$, $\overline{\mathcal V_d^{2g-3,1,1}}$ is generically smooth.

\end{theorem}

\begin{proof}  Once part (i) has been proved, parts (ii)--(iii) follow from Theorem \ref{unepi} and Proposition \ref{C.F.4}, with $\delta=2g-3$, $j=t=1$, whereas part (iv) follows from Proposition \ref{C.F.1}-(ii), with $j=2$.

The proof of part (i) consists of four  steps.


\noindent
{\bf Step 1}. In this step, we show that if $\mathcal B$ is an irreducible component of  $\widetilde{B_C^{k_2}}(d)$ such that, for $[\Ff] \in \mathcal B$ general, $i (\Ff) =2$, then $\mathcal B$ comes from a total good component $\widehat{\mathcal W}^{\rm Tot}_1 \subseteq \Pp(\mathfrak{E}_{\delta})$, for some $\delta$ (cf. Thm. \ref{thm:mainext1} and Def. \ref{def:goodc}). 

Indeed, let \eqref{eq:Fund} be a special presentation of $\Ff$ with $L$ a quotient of minimal degree. Then, from Remarks \ref{rem:claim1}, \ref{30.12l}, \ref{rem:C.F.3}, one has $h^1(L)=j= 1$ as $\mathcal B$ is a component. Hence, 
with notation as in \eqref{eq:exthyp1}, \eqref{eq:exthyp} and \eqref{eq:wt}, one must have $t=1$ and $\ell \geqslant r-1$.  Moreover, $d \leqslant 3g - 6$, $\delta \geqslant g-1$ and  $j=1$ imply that $m = 2\delta - d + g -1 \geqslant \delta - g + 3 = \ell + 1 $ (recall notation as in \eqref{eq:lem3note}).  Therefore, we can apply Theorem \ref{thm:mainext1}, finishing the proof of this step.


\noindent
{\bf Step 2}. In this step we determine which of the constructed loci either $\mathcal V_{d}^{\delta,j}$, as in \eqref{eq:Nudde}, or $\mathcal V_{d}^{\delta,j,t}$, as in \eqref{eq:nuddet}, has general point $[\Ff]$ such that $i(\Ff) =2$ and dimension at least $\rho_d^{k_2} = 8g-11-2d$ (hence, it can be conjecturally dense in a component of $\widetilde{B_C^{k_2}}(d)$).  We will prove that this only happens for $2g-3 \leqslant \delta \leqslant 2g-2$ and $j = t = 1$.  Moreover, we will show that the presentation  of $\Ff$ is specially rigid only if 
$\delta = 2g-3$. 

Let $\mathcal V$ be any such locus, and let $\Ff$ be its general point which is presented as in \eqref{eq:Fund} with special quotient $L$. 
As in Step 1, one finds $j=1$ hence $N$ has to be special and $t=1$, so $\mathcal V$ has to be necessarily of the form $\mathcal V_{d}^{\delta,1,1}$ (i.e. loci of the form $\mathcal V_{d}^{\delta,j}$ are excluded).  We have two cases to consider: (a) $N$ non-effective, (b) $N$ effective.


\noindent
{\bf Case (a)}.  As in Remark \ref{rem:bohb}-(1), recall that a necessary condition for 
$\dim(\mathcal V_{d}^{\delta,1,1}) = \nu^{\delta,1,1}_d \geqslant \rho_d^{k_2}$ is 
$\varphi_0(\delta,1,1) \geqslant 0$, i.e. $2g-3 \leqslant \delta \leqslant 2g-2$. 

In case $\delta = 2g-2$, 
$[\Ff] \in \mathcal{V}_d^{2g-2,1,1}$ general is not rsp via $\omega_C$ (it follows from the fact that $i=2$ and computation as in \eqref{eq:casaciro2}). 

In case $\delta = 2g-3$, the hypotheses $ 2g-2 \leqslant d \leqslant 3g-6$ ensure stability for $\Ff$ (cf. Theorems \ref{thm:mainext1}, \ref{unepi} and Proposition \ref{C.F.4}). By definition, $\mathcal V_{d}^{2g-3,1,1} = {\rm Im} (\pi_{d,2g-3}|_{\widehat{\mathcal W}^{\rm Tot}_1})$, where $\widehat{\mathcal W}^{\rm Tot}_1 \subset \Pp(\mathfrak{E}_{2g-3})$ is the good locus for $t=1$. To accomplish the proof, we need to show that the fibre of $\pi_{d,2g-3}|_{\widehat{\mathcal W}^{\rm Tot}_1}$ over $[\Ff] \in \mathcal V_{d}^{2g-3,1,1}$ general is finite. As in 
\eqref{eq:aiutoB} in Remark \ref{rem:17lug}, it suffices 
to prove the following:

\begin{claim}\label{4.1} $a_{F}(2g-3) = 0$.
\end{claim}
\begin{proof}[Proof of the Claim] Assume by contradiction this is not zero. Since $\Ff$ is stable, hence unsplit, from $\varphi_0(2g-3,1,1) =0$ and Remark \ref{rem:rigid}, $a_F(2g-3)$ must be $1$ (cf. Proposition \ref{prop:lem4}). 

Let $\mathfrak{F}$ be the corresponding one-dimensional family of sections of $F= \Pp(\Ff)$, which has positive self-intersection, since 
$\Ff$ is stable. From Proposition \ref{prop:lem4} and Step 1, the system $\mathfrak{F}$ cannot be contained in a linear system, otherwise we would have sections of degree lower than $2g-3$.

Thus, from the proof of Proposition \ref{prop:lem4}, there is an open, dense subset $C^0 \subset C$ such that, for any $q \in C^0$, one has $\Ff= \Ff_v$ with $ v = v_q \in \Ext^1(\omega_C(-q), N_q))$, where $\{N_q\}_{q \in C^0}$ is a $1$-dimensional family of non-isomorphic line bundles of degree $d - 2g+3$, whose general member is general in $\Pic^{d-2g+3}(C)$ .  Let $\Gamma_q \subset F_v$ be the section corresponding to 
$\Ff_v \to\!\!\to \omega_C(-q)$;  so the one-dimensional family is $\mathfrak{F} = \{\Gamma_q \}_{q\in C^0}$.

We set $\widetilde{\Gamma}_q := \Gamma_q + f_q$, for  $q \in C^0$. From \eqref{eq:Fund2},  $\widetilde{\Gamma}_q$ corresponds to $\Ff_v \to\!\!\!\!\to \omega_C(-q) \oplus \Oc_q$, whose kernel we denote by $N'_q$. 
Then  $\widetilde{\mathfrak{F}} = \{ \widetilde{\Gamma}_q \}_{q\in C^0}$ is a one-dimensional family of unisecants  of $F_v$ of degree $2g-2$ and speciality $1$ (cf. \eqref{eq:iLa}). For $h,q \in C^0$, we have
$$c_1(N'_h) = \det(\Ff_v) \otimes \omega_C^{\vee} = c_1(N'_q).$$Therefore, 
from \eqref{eq:isom2}, $\widetilde{\mathfrak{F}}$ is contained in a linear system $|\Oc_{F_v} (\Gamma)|$. By Bertini's theorem, the general member of $|\Oc_{F_v} (\Gamma)|$ is a section of degree $2g-2$. In particular,  
$\dim(|\Oc_{F_v} (\Gamma)|) \geqslant 2$.  

If $L_{\Gamma}$ is the corresponding 
quotient line bundle, since $\Gamma \sim \widetilde{\Gamma}_q$, then $c_1(L_{\Gamma}) = \omega_C$, i.e. 
$\Gamma$ is a canonical section. This is a contradiction: indeed, if $M_{\omega_C}$ is the kernel 
of the surjection $\Ff_v \to\!\!\!\to \omega_C$, we have (cf. \eqref{eq:casaciro2})  
$$2 \leqslant \dim(|\Oc_{F_v} (\Gamma)|) = h^0(\Ff_v \otimes M^{\vee}_{\omega_C}) - 1 = i(\Ff_v) -1 =1.$$
\end{proof}


\noindent
{\bf Case (b)}.  As in Remark \ref{rem:bohb}-(2), a necessary condition for $ \nu^{\delta,1,1}_d \geqslant \rho_d^{k_2}$ is  \eqref{eq:8.1}, i.e. $nr - n - 1 <0$. Since $n,r\geqslant 1$, the only possibility is $r = t= 1$. 
Taking into account \eqref{eq:fo1t}, \eqref{eq:fnjt} and Proposition \ref{C.F.4}, one has $\varphi_n(\delta, 1,1) = \varphi_0(\delta, 1,1) = \delta-2g+3 \geqslant 0$. In any case we would have $d \geqslant 3g-5$, which is out of our range. Thus, case (b) cannot occur.


\noindent
{\bf Step 3}. In this step we prove that $\overline{\mathcal V_{d}^{2g-3,1,1}}$ is actually a component of $\widetilde{B_C^{k_2}}(d)$. 

Let $\mathcal B \subseteq \widetilde{B_C^{k_2}}(d)$ be a component containing $\overline{\mathcal V_{d}^{2g-3,1,1}}$ and let 
$[\Ff] \in \mathcal B$ general.  By semicontinuity, $\Ff$ has speciality $i=2$. It has also a special presentation as in \eqref{eq:Fund}, with $2g-3 \leqslant \deg(L) = \delta \leqslant 2g-2$. Since $C$ has general moduli, then $h^1(L) = j =1$ so the corank of the coboundary map is $t=1$. If $\delta = 2g-3$, from Step 2 we are done. 

Assume therefore $\delta = 2g-2$, so $L = \omega_C$. Notice that:$ (i) \; r = h^1(N) \leqslant g; \;\;  (ii) \; m = \dim({\rm Ext}^1(N, \omega_C)) \geqslant g+1$. Indeed, (i) is trivial if $N$ is effective. On the other hand, if $h^0(N)=0$, then $h^1(N) = 3g-3-d < g$ since $d \geqslant 2g-2$. As for (ii), 
$m = 5g-5-d$ (cf. \eqref{eq:lem3note}), hence (ii) follows since $d \leqslant 3g-6$.  So we are in position to apply Theorem \ref{thm:mainext1} and Corollary \ref{cor:mainext1}, which yield that $\widehat{\mathcal{W}}^{\rm Tot}_1 \subseteq \Pp(\mathfrak{E}_{2g-2})$ is irreducible and good. 
Hence $\dim (\widehat{\mathcal{W}}_1^{\rm Tot}) \leqslant 8g-10 - 2d$ (equality holds when $N$ is general, i.e. non effective). 

On the other hand, $\mathcal B$ is the image of $\widehat{\mathcal{W}}^{\rm Tot}_1$ via $\pi_{d,2g-2}$ (cf. Step 1) and the general fibre of this map has dimension at least $1$ because $h^1(\Ff) = 2$ (cf. \eqref{eq:isom2} and computation as in \eqref{eq:casaciro2}). Thus $$8g-11-2d \geqslant  \dim(\mathcal B) \geqslant \dim (\overline{\mathcal V_{d}^{2g-3,1,1}}) = 8g-11-2d = \rho_d^{k_2}.$$This proves that $\mathcal B = \overline{\mathcal V_{d}^{2g-3,1,1}}$ is a regular component.

The previous argument also shows that $ \overline{\mathcal V_{d}^{2g-3,1,1}} =  \overline{\mathcal V_{d}^{2g-2,1,1}}$ (cf. Remark \ref{rem:bohb}) and that the dimension of the general fibre of $\pi_{d, 2g-2}|_{\widehat{\mathcal{W}}_1^{\rm Tot}}$ onto $\mathcal V_{d}^{2g-2,1,1}$ has exactly dimension $1$ (actually, it is a $\Pp^1$, cf. Lemma \ref{lem:ovviolin}).


\noindent
{\bf Step 4}. Assume we have a component  $\mathcal B \subseteq \widetilde{B_C^{k_2}}(d)$, whose general point corresponds 
to a vector bundle $\Ff$ with $i(\Ff)=2$.  From Step 1, $[\Ff] \in \mathcal B$ 
general can be specially presented as in \eqref{eq:Fund}, with $h^1(L) = j =1$, so $N$ is special. The same discussion as in Steps 2 and 3 shows that $\mathcal B = \overline{\mathcal V_{d}^{2g-3,1,1}}$. 
\end{proof}

\begin{remark}\label{rem:casaciro3} (i) For $[\Ff] \in \overline{\mathcal V_{d}^{2g-3,1,1}}$ general, one has 
$\ell = g-1$, $r = 3g-4-d$ and $m = 5g-7-d$ (cf. \eqref{eq:lem3note}, \eqref{eq:exthyp1}). So 
$\ell \geqslant r + 1$ (because $d \geqslant 2g-2$), moreover $m \geqslant \ell +g$ (because $d \leqslant 3g-6$). Note that 
the inequality $m \geqslant \ell +1$ is necessary to ensure $\emptyset \neq \mathcal{W}_1^{T\rm ot} \subset \Pp(\mathfrak{E}_{2g-2})$ (see the proof of Theorem \ref{thm:mainext1}).  

\noindent
(ii) Step 4 of Theorem \ref{i=2} shows that, if $\mathcal B$ is a component of $ \widetilde{B_C^{k_2}}(d)$, 
different from $\overline{\mathcal V_{d}^{2g-3,1,1}}$, then $\mathcal B$ is a component of 
$ \widetilde{B_C^{k_i}}(d)$, for some $i \geqslant 3$, and as such it is not regular. Otherwise, we would have 
$$8g-11 - 2d \leqslant \dim(\mathcal B) = 4g-3 - i(d-2g+2+i) \leqslant  10 -3d -16,$$i.e $d \leqslant 2g-7$ which is out of our range for $d$.   

\end{remark}

\begin{remark}\label{rem:i=2} (1) Take $\widetilde{\Gamma}_p$ as in the proof of Claim \ref{4.1}. Then, $\N_{\widetilde{\Gamma}_p /F_v}$ is non special on the (reducible) unisecant $\widetilde{\Gamma}_p$. Indeed, 
$\omega_{\widetilde{\Gamma}_p} \otimes \N_{\widetilde{\Gamma}_p /F_v}^{\vee}|_{\Gamma} 
\cong \Oc_{\Gamma} (K_F)$ whereas $ \omega_{\widetilde{\Gamma}_p} \otimes \N_{\widetilde{\Gamma}_p /F_v}^{\vee}|_{f_p} 
\cong \Oc_{\Pp^1}(-2)$.  Thus $\widetilde{\Gamma}_p \in {\rm Div}^{1,2g-2}_{F_v}$ is a smooth point. Moreover, $h^0(\N_{\widetilde{\Gamma}_p /F_v}) = 3g-2-d \geqslant 2$ for $d \leqslant 3g-4$.  From the generality of $v$ in the good locus $\mathcal W_1$, \eqref{eq:isom2} and from computation as in \eqref{eq:casaciro2}, one has 
that $\widetilde{\Gamma}_p \subset F_v$ is a (reducible) unisecant, moving in a complete linear pencil of special unisecants whose general member is a canonical section, and $\widetilde{\Gamma}_p$ is algebraically equivalent on $F_v$ to non-special sections of degree $2g-2$. 

As soon as $d \leqslant 3g-6$, there are in ${\rm Div}^{1,2g-2}_{F_v}$ unisecants containing two general fibres 
(cf. Proposition \ref{prop:lem4}) hence the ruled surface $F_v$ has (non special) sections of degree smaller than $2g-3$. 


\noindent
(2) Take $N \in \Pic^k(C)$ general with $0 < k \leqslant g-2$. Since $N$ is special, non-effective, 
from Corollary \ref{cor:mainext1} and Remark \ref{unepine}, $v \in \Lambda_1 \subset \Ext^1(\omega_C,N)$ general determines $\Ff:=\Ff_v$ stable, with $i(\Ff)=2$. If $\Gamma$ denotes the canonical section corresponding to $\Ff \to \omega_C$, from \eqref{eq:isom2} one has $\dim(|\Oc_{F} (\Gamma)|) = 1 $ and all unisecants in this linear pencil are special (cf. Lemma \ref{lem:ovviolin}). Since $F$ is indecomposable, 
$|\Oc_{F} (\Gamma)|$ has base-points (cf. the proof of  Proposition \ref{prop:lem4}, from which we keep the notation). 
Thus, $\Ff$ is rsp via  $\omega_C(-p)$, for $p= \rho(q)$ and $q \in F$ a base point of the pencil (recall Remark \ref{rem:bohb}).  
\end{remark}

\begin{remark}\label{i=2Teix} In \cite{TB0,TB00} the locus $\widetilde{B_C^{2}}(b)$ is studied, for $g \geqslant 2$ and $3 \leqslant b \leqslant 2g-1$. It is proved there  with different arguments that, when $C$ has general moduli, then $\widetilde{B_C^{2}} (b)$ is not empty, irreducible, regular (with $\rho^2_b = 2b-3$), generically smooth and  $[\Ef] \in \widetilde{B_C^{2}}(b)$ general is stable,  with $h^0(\Ef) =2$, fitting in  a sequence 
\begin{equation}\label{eq:aiutoF}
0 \to \Oc_C \to \Ef \to L \to 0.
\end{equation}Considering the natural isomorphism $\widetilde{B_C^2}(b) \cong \widetilde{B_C^{k_2}}(4g-4-b)$,  when $d := 4g-4-b$ is as in Theorem \ref{i=2} we recover Teixidor's results (without irreducibility) via a different approach.  Thus, Teixidor's results and our analysis imply that, for any $2g-2 \leqslant d \leqslant 4g-7$,  $\widetilde{B_C^{k_2}}(d) = \overline{\mathcal V_{d}^{2g-3,1,1}}$. Theorem \ref{i=2} provides in addition the rigidly special presentation of the general element of $\widetilde{B_C^{k_2}}(d)$. 
\end{remark}


\begin{theorem}\label{prop:M22K} Let $C$ be of genus $g \geqslant 3$, with general moduli. Then, 
$\widetilde{M_C^2}(2, \omega_C) \neq \emptyset$ and irreducible. Moreover, it is regular (i.e. of dimension $3g-6$), and its general point $[\Ff_v]$ fits into  an exact  sequence 
$$0 \to \Oc_C(p) \to \Ff_v \to \omega_C (-p) \to 0,$$where 

\begin{itemize}
\item[$\bullet$] $p \in C$ is general, and 
\item[$\bullet$] $v \in \Lambda_1 =\mathcal W_1  \subset \Ext^1(\omega_C(-p), \Oc_C(p))$ is general.    
\end{itemize}
\end{theorem} 

\begin{proof} Irreducibility follows from \cite[Thm. 1.3]{Oss2}. With notation as in \eqref{eq:exthyp1}, \eqref{eq:exthyp}, one has $$\ell = r = g-1, \; m = h^1(2p- K_C) = 3g - 5 \geqslant 
g = \ell + 1;$$from Corollary \ref{cor:mainext1}, $\mathcal{W}_1$ is good and $v \in \mathcal W_1$ general is such that 
$\cork(\partial_v) =1$. In particular, $\dim(\mathcal W_1) = 3g-6$. 

Stability of $\Ff_v$, with $1 < s(\Ff_v) = \sigma < g$, follows from Proposition \ref{prop:LN}. Finally one uses the same approach as in Claim \ref{4.1} to deduce that $\pi_{2g-2, 2g-3}|_{\mathcal W^{\rm Tot}_1}$ is generically finite (cf. \eqref{eq:pde}), since $\Ff_v$ is rsp via $\omega_C(-p)$.   
\end{proof}

Generic smoothness of the components of $\widetilde{M_C^2}(2, \omega_C)$ follows from results in \cite{TB000,Beau2}. 
Theorem \ref{prop:M22K} can be interpreted in the setting of \cite{BeFe} as saying that, for a curve $C$ of general moduli of genus $g \geqslant 3$, $\Pp({\rm Ext}^1(\omega_C, \Oc_C))$ is not contained in the divisor $D_1$ considered in that paper.



\subsection{Case $i=3$}\label{ssec:3} One has $\rho_d^{k_3} = 10g - 18 - 3d$. We have the following:  

\begin{theorem}\label{i=3} Let $C$ be of genus $g \geqslant 8$, with general moduli. 
For any integer $d$ s.t. $2g-2 \leqslant d \leqslant \frac{5}{2}g-6$, one has $\widetilde{B_C^{k_3}}(d) \neq \emptyset$. 
Moreover:

\noindent
(i) $ \overline{\mathcal V_d^{2g-4,1,2}}$ is the unique component of $\widetilde{B_C^{k_3}}(d) $ of types either  \eqref{eq:Nudde} or  \eqref{eq:nuddet}, whose  general point corresponds to a vector bundle $\Ff$ with $i(\Ff) = 3$. Furthermore, 
it is regular and $\overline{\mathcal V_d^{2g-4,1,2}}=  \overline{\mathcal V_d^{2g-3,1,2}} = \overline{\mathcal V_d^{2g-2,1,2}}$.

\noindent
(ii) For $[\Ff]\in  \overline{\mathcal V_d^{2g-4,1,2}}$ general, one has $s(\Ff) \geqslant 5g-10-2d-\epsilon \geqslant 2-\epsilon$ and $\Ff$ fits into an exact sequence  
$$0 \to N \to \Ff \to \omega_C (-D_2) \to 0,$$where
\begin{itemize}
\item[$\bullet$] $D_2 \in C^{(2)}$ is general, 
\item[$\bullet$] $N \in \Pic^{d-2g+4}(C)$  is general (special, non-effective), 
\item[$\bullet$] $\Ff=\Ff_v$ with $v$ general in a good component $\Lambda_2 \subset \Ext^1(N, \omega_C(-D_2))$ (cf. Definition \ref{def:goodc}). 
\end{itemize}

\noindent
(iii) Any section $\Gamma \subset F= \Pp(\Ff)$, corresponding to a quotient $\Ff \to\!\!\to \omega_C(-D_2)$, is not of minimal degree. However, 
it is minimal among special sections of $F$; moreover,  $\Gamma$ is asi but not ai (i.e., $\Ff$ is rsp via $\omega_C(-D_2)$). 
\end{theorem}

\begin{proof} As in Theorem \ref{i=2}, once (i) has been proved, parts (ii)-(iii) follow from results proved in previous sections. 
Precisely, by definition of $\mathcal V_d^{2g-4,1,2}$ one has $L = \omega_C(-D_2)$, with $D_2 \in C^{(2)}$, $t=2$ and $N \in {\rm Pic}^{d-2g+4} (C)$ of speciality $r \geqslant 2$. From regularity of the component, Proposition  \ref{C.F.4} and \eqref{eq:fojt}, \eqref{eq:fo1t}, \eqref{eq:fnjt} give 
$$0 = \nu^{2g-4,1,2}_d - \rho_d^{k_3} \leqslant \dim(\widehat{\Lambda}^{\rm Tot}_2) - \rho_d^{k_3} = \varphi_n(2g-4,1,2) = \varphi_0(2g-4,1,2) - n (r-2) = - n (r-2).$$Thus, $n(r-2) = 0$. This implies that 
the general fibre of $\pi_{d,2g-4}|_{\widehat{\Lambda}^{\rm Tot}_2}$ is finite, i.e. $[\Ff_v] \in \mathcal V_d^{2g-4,1,2}$ general is rsp via  $\omega_C(-D_2)$ (correspondingly $\Gamma \subset F_v = \Pp(\Ff_v)$ is asi as in (iii)). 

Since $n(r-2) = 0$, then either $n=0$ or $r=2$. The latter case cannot occur  
otherwise we would have $n = d - 3g+7 \leqslant - \frac{g}{2} + 1 <0$, by the assumptions on $d$. 

Thus $n=0$ and $r = 3g-5-d$. Moreover, from \eqref{eq:lem3note}, \eqref{eq:clrt}, one has 
$m = 5g-10 -d$ and $c (\ell, r,2) = 2d+10 -4g$, so a good component $\Lambda_2 \subset \Pp({\rm Ext}^1(\omega_C(-D_2), N))$  has dimension $9g-20-3d$. If we add up to this quantity $g$, for the parameters of $N$, we get 
$10 g - 20 - 3d$. Thus, regularity forces $D_2 $ to be general in $C^{(2)}$.

Now, $\N_{\Gamma/F_v} \cong K_C - D_2 -N$ (cf. \eqref{eq:Ciro410}) so $h^i(\N_{\Gamma/F_v}) = h^{1-i}(N+D_2)$ for $0 \leqslant i \leqslant 1$. By the assumptions on $d$, $\deg(N+D_2) = d-2g+6 \leqslant \frac{g}{2}$, thus generality of $N$ implies that $N+D_2$ is also general, 
so $h^0(N+D_2) = 0$ and $h^1(N+D_2) = 3g-7-d \geqslant \frac{g}{2} -1$. This implies that $\Gamma$ is not ai and not of minimal degree among quotient line bundles of $\Ff$. 

Numerical conditions of Theorem \ref{unepi} (see also Remark \ref{unepine}) are satisfied for $j=1$, $t = 2$ and $\delta = 2g-4$, under the assumptions $d \leqslant \frac{5}{2} g -6$.  

Finally, the fact that $\Gamma$ is of minimal degree among  special, quotient line bundles of $\Ff$ follows from the proof of part (i) below, which consists of the following steps.


\noindent
{\bf Step 1}. In this step we determine which of the loci  of the form $\mathcal V_{d}^{\delta,j}$, as in \eqref{eq:Nudde}, or $\mathcal V_{d}^{\delta,j,t}$, as in \eqref{eq:nuddet}:

\begin{itemize}
\item[(a)] has the general point $[\Ff]$ with $i(\Ff) = 3$,

\item[(b)] is the image, via $\pi_{d,\delta}$, of a parameter space in $\Pp(\mathfrak{E}_{\delta})$ of dimension at least $\rho_d^{k_3} = 10g - 18 - 3d$. 
\end{itemize}

Let $\mathcal V$ be such locus and use notation as in \eqref{eq:exthyp1}, \eqref{eq:exthyp}. From Remarks \ref{rem:claim1}, \ref{30.12l}, \ref{rem:C.F.3},  conditions (a) and (b) are both satisfied only if the presentation of $[\Ff] \in \mathcal V$ general as in \eqref{eq:Fund} with $L$ special and effective, is such that $N \in {\rm Pic}^{d-\delta}(C)$ is special and the coboundary map $\partial : H^0(L) \to H^1(N)$ is not surjective. Possibilities are:  

\noindent
$(i)$  $j=1$ and $t= \cork(\partial)=2$;

\noindent
$(ii)$ $j=2$ and $t=1$.

In any event, one has $\ell \geqslant r$ (in particular, we will be in position to 
apply Theorems \ref{thm:mainext1}, \ref{thm:mainextt}; cf. e.g the proof of 
Claim \ref{cl:good} below). Indeed: 


\noindent
$\bullet$ in case $(i)$, the only possibilities for $\ell <r$ are $r-2 \leqslant \ell \leqslant r-1$. Then, 
$\dim(\Pp({\rm Ext}^1(L,N)) = 2 \delta - d + g - 2$, $\rho(L) = g - (\delta - g + 2)$, 
$\rho(N) = g-rn$, so the  number of parameters is 
$\delta + 4(g-1) - d - rn  < \rho_d^{k_3}$ since $d \leqslant \frac{5}{2}g -6$.


\noindent
$\bullet$ in case $(ii)$, the only possibility  for $\ell <r$  is $\ell = r-1$.  The same argument as above applies, the only difference is that $\rho(L) = g - 2(\delta - g + 3)$.

Since $\ell \geqslant r$, we see that case $(ii)$ cannot occur by  \eqref{eq:fojtb} and \eqref{eq:fnjt}. Thus, we focus on  $ \mathcal V_{d}^{\delta,1,2}$, investigating for which $\delta$  it satisfies (b).  We will prove that this only happens for $2g-4 \leqslant \delta \leqslant 2g-2$.

We have two cases: (1) $N$ effective, (2) $N$ non-effective. We will show that only case (2) occurs.


\noindent
{\bf Case (1)}. When $N$ is effective, from Remark \ref{rem:bohb}, a necessary condition for (b) to hold is \eqref{eq:8.1}, which reads $(r-2) n - 2 <0$. This gives $2 \leqslant r \leqslant 3$, since $r \geqslant t=2$.  We can apply Theorem \ref{unepi} (more precisely, Remark \ref{unepie}):  the first line of bounds on $\delta$ in Remark \ref{unepie} gives  $\delta \leqslant \frac{3g-\epsilon}{2} + r - 4 $. 
In particular, one must have $\delta \leqslant \frac{3g-\epsilon}{2} -1$. 

On the other hand, another necessary condition for (b) to hold 
is  $\varphi_n(\delta, 1,2) \geqslant 0$  (cf. Proposition \ref{C.F.4}). From Remark \ref{rem:conpar2}, 
$\varphi_n(\delta, 1,2) \leqslant \varphi_0(\delta, 1,2) = \delta-2g+4$ (cf. \eqref{eq:fo1t}) and  $\varphi_0(\delta, 1,2) \geqslant 0$ gives $\delta \geqslant 2g-4$ which contradicts $\delta \leqslant \frac{3g-\epsilon}{2} -1$, since $g \geqslant 8$.


\noindent
{\bf Case (2)}. When $N$ is non-effective, we apply Theorem \ref{unepi} (more precisely, 
Remark \ref{unepine}), with $j=1$ and $t=2$. By the same argument as in case (1), we see that $2g-4 \leqslant \delta \leqslant 2g-2$. 


\noindent
{\bf Step 2}. In this step we prove that the loci $\mathcal V_d^{\delta,1,2}$, with $2g-4 \leqslant \delta \leqslant 2g-2$, 
are not empty. Precisely, we will exhibit components $\mathcal V_d^{\delta,1,2}$ which are the image, via  $\pi_{d,\delta}$, of a total good component $\widehat{\Lambda}^{\rm Tot}_2 \subset \Pp(\Efra_\delta)$, of dimension $\rho_d^{k_3} + \delta - 2g + 4$ (cf. Definition \ref{def:goodtot} and \eqref{eq:nuddet}).

We will treat only the case $\delta= 2g-4$, i.e. $L= \omega_C(-D_2)$, with $D_2 \in C^{(2)}$, since the cases $L= \omega_C, \, \omega_C(-p)$ can be dealt  with similar arguments and can be left to the reader.

\begin{claim}\label{cl:Grass} Let  $N \in \Pic^{d-2g+4}(C)$ be general. For   
$V_2 \in \mathbb{G} (2, H^0(K_C-N))$ general, the map $\mu_{V_2}$ as in \eqref{eq:muW} is injective. 
\end{claim}
\begin{proof}[Proof of Claim \ref{cl:Grass}] The general $V_2 \in \mathbb{G} (2, H^0(K_C-N))$ determines a base point free linear pencil on $C$. Indeed, $h^0(K_C-N) = 3g-5-d \geqslant 5$. Take $\sigma_1, \sigma_2 \in H^0(K_C-N)$ general sections. If $p \in C$ is such that $\sigma_i(p)=0$, for $i=1,2$, by the generality of the sections we would have 
$p \in {\rm Bs}(|K_C-N|)$ so $h^0(N+p)=1$. This is a contradiction because $N$ is general and $\deg(N) < g-1$. The injectivity of $\mu_{V_2}$ follows from the base-point-free pencil trick: indeed, ${\rm Ker}(\mu_{V_2}) \cong H^0(N(-D_2))$ which is zero since $N$ is non-effective.
\end{proof}

\begin{claim}\label{cl:good} Let $N \in \Pic^{d-2g+4}(C)$ and $D_2 \in C^{(2)}$ be general. Then, there exists a unique good component $\Lambda_2 \subset \Ext^1(\omega_C(-D_2), N)$ whose general point $v$ is such that 
${\rm Coker}(\partial_v)^{\vee}$ is general in $\mathbb{G} (2, H^0(K_C-N))$ (cf. Remark \ref{rem:wt}). 
\end{claim}
\begin{proof}[Proof of Claim \ref{cl:good}] With notation as in \eqref{eq:lem3note}, \eqref{eq:exthyp1}, we have $\ell  = g-2$, $m=  5g-9-d$ and 
$r=  3g-5-d$. Then assumptions on $d$ and $g$ imply
\begin{equation}\label{eq:aiutow2}
m \geqslant 2 \ell +1 \; \;{\rm and}  \;\; \ell  \geqslant r \geqslant t=2
\end{equation} (cf. Step 1 for $\ell \geqslant r$). From \eqref{eq:aiutow2} and Claim \ref{cl:Grass}, we are in position to apply Theorem \ref{thm:mainextt}, with $\eta=0$ and $\Sigma_{\eta} =  \mathbb{G} (2, H^0(K_C-N))$. 

This yields the existence of a good component $\Lambda_2 \subseteq \mathcal W_2 \subset {\rm Ext^1}(\omega_C(-D_2), N)$.  Actually, $\Lambda_2$ is the only good component whose general point $v$ gives ${\rm Coker}(\partial_v)^{\vee} = V_2$ general in $\mathbb{G} (2, H^0(K_C-N))$. 

Indeed any component of $\mathcal W_2$, whose general point $v$ is such $\dim({\rm Coker}(\partial_v)) =2$, is obtained in the following way (cf. the proofs of Theorems \ref{thm:mainext1}, \ref{thm:mainextt}):  


\noindent
$\bullet$ take any $\Sigma \subseteq \mathbb{G} (2, H^0(K_C-N))$ irreducible, of codimension $\eta\geqslant 0$;

\noindent
$\bullet$  for $V_2$ general in $\Sigma$, consider  $H^0(\omega_C(-D_2)) \otimes V_2$ and the  map $\mu_{V_2}$ as in \eqref{eq:muW};

\noindent
$\bullet$ let  $\tau := \dim({\rm Ker}(\mu_{V_2})) \geqslant 0$ and $\Pp := \Pp({\rm Ext}^1(\omega_C(-D_2), N)) = \Pp(H^0(2K_C -D_2 - N)^{\vee})$ (cf. \eqref{eq:PP}); 

\noindent
$\bullet$ consider the incidence variety 
$$\Jj_{\Sigma} := \left\{(\sigma, \pi) \in \Sigma \times \Pp \; | \; {\rm Im}(\mu_{V_2}) \subset \pi \right\}.$$Since 
$m \geqslant 2 \ell + 1 - \tau$, one has $\Jj_{\Sigma} \neq \emptyset$ (cf. the proofs of Theorems \ref{thm:mainext1}, \ref{thm:mainextt}); 

\noindent
$\bullet$ consider the projections $\Sigma \stackrel{pr_1}{\longleftarrow} \Jj_{\Sigma} \stackrel{pr_2}{\longrightarrow}  \Pp$;

\noindent
$\bullet$ the fibre of $pr_1$ over any point $V_2$ in the image is  $\left\{ \pi \in \Pp\,|\, {\rm Im}(\mu_{V_2}) \subset \pi\right\}$, i.e. it 
is isomorphic to the linear system of hyperplanes of $\Pp$ passing through the linear subspace $\Pp({\rm Im}(\mu_{V_2}))$. For $V_2 \in \Sigma$ general, 
this fibre is irreducible of dimension $m -1 - 2 \ell  + \tau = 3g-6-d+\tau$. In particular, there exists a unique component $\Jj \subseteq \Jj_{\Sigma}$ dominating $\Sigma$ via $pr_1$;

\noindent
$\bullet$  since $r=3g-5-d$, one has 
$$\dim(\Jj) = 9g-20-3d + \tau - \eta = {\rm expdim}(\widehat{\mathcal W}_2) + \tau - \eta,$$where $\widehat{\mathcal W}_2 = \Pp(\mathcal W_2) \subset \Pp$ (notation as in the proof of Theorem \ref{thm:mainext1}). By construction, $pr_2(\Jj) \subseteq \widehat{\mathcal W}_2$. Moreover, if $\epsilon$ denotes the dimension of the general fibre of $pr_2|_{\Jj}$, then $pr_2(\Jj)$ can fill up a component $\mathcal X$ of $\widehat{\mathcal W}_2$  only if $\tau - \eta - \epsilon \geqslant 0$: the component $\mathcal X$ is  good when equality holds.


When $\Sigma = \mathbb{G} (2, H^0(K_C-N))$, then $\eta = 0$ and, by Claim \ref{cl:Grass},  also $\tau = 0$. Since $\Jj \subseteq \Jj_{\mathbb{G} (2, H^0(K_C-N))}$ is the unique component dominating 
$\mathbb{G} (2, H^0(K_C-N))$, then  $pr_2(\Jj)$  fills up a component $\widehat{\Lambda}_2$ of $\widehat{\mathcal W}_2$, i.e. $\epsilon = 0$. Thus $\widehat{\Lambda}_2$ is good. 
\end{proof}

By the generality of $N \in {\rm Pic}^{d-2g+4} (C)$ and of $D_2 \in C^{(2)}$, Claim \ref{cl:good} ensures the existence of a total good component $\widehat{\Lambda}_2^{\rm Tot} \subset \Pp (\mathfrak{E}_{2g-4})$.

For $2g-4 \leqslant \delta \leqslant 2g-2$, we will denote by $V_d^{\delta}$ the total good component we constructed in this step. To ease notation, we will  denote by $\mathcal V_d^{\delta}$ its image in $\widetilde{B_C^{k_3}}(d)$ via  $\pi_{d,\delta}$, which is a $\mathcal V_{d}^{\delta,1,2}$ as in Step 1.


\noindent
{\bf Step 3}. In this step, we prove that $\mathcal V_d^{2g-2}$ has dimension $\rho_d^{k_3}$.

From Step 2, one has $\dim(V_d^{2g-2}) = \rho_d^{k_3} + 2 = 10g - 16 - 3d$. 
We want to show that the general fibre of $\pi_{d,2g-2}|_{V_d^{2g-2}}$ has dimension two. To do this, we use similar arguments as in the proof of Lemma \ref{lem:claim1}. 

Let $[\Ff] \in \mathcal V_d^{2g-2}$ be  general; by Step 2,  
$\Ff = \Ff_u$, for $u \in \widehat{\Lambda}_2  \subset \Pp(H^0(2 K_C- N)^{\vee})$ general and $N \in {\rm Pic}^{d-2g+2}(C)$ general, where $\widehat{\Lambda}_2 = pr_2(\Jj)$ and  
$\Jj \subset \mathbb{G} (2, H^0(K_C-N)) \times \Pp(H^0(2 K_C- N)^{\vee})$ the unique component dominating $ \mathbb{G} (2, H^0(K_C-N))$ (cf. the proof of Claim \ref{cl:good}). Then $$(\pi_{d,2g-2}|_{V_d^{2g-2}})^{-1}([\Ff_u]) = \left\{ (N', \omega_C, u') \in V_d^{2g-2} \, | \; \Ff_{u'} \cong \Ff_u \right\}.$$In particular, one has $N \cong N'$ so $u, u' \in  \widehat{\Lambda}_2 \subset \Pp(H^0(2 K_C- N)^{\vee})$.

Let $\varphi : \Ff_{u'} \stackrel{\cong}{\to} \Ff_u$ be the isomorphism between the two bundles and consider the diagram 
\[
\begin{array}{ccccccl}
0 \to & N & \stackrel{\iota_1}{\longrightarrow} & \Ff_{u'} & \to & \omega_C & \to 0 \\ 
 & & & \downarrow^{\varphi} & & &  \\
0 \to & N & \stackrel{\iota_2}{\longrightarrow} & \Ff_u & \to & \omega_C & \to 0.
\end{array}
\]


  If $u= u'$,  then $\varphi = \lambda \in \C^*$ (since $\Ff_u$ is simple) and the maps 
$\lambda \iota_1$ and $\iota_2$ determine two non-zero sections $s_1 \neq s_2 \in H^0(\Ff_u \otimes N^{\vee})$. 
Similar computation as in \eqref{eq:casaciro2} shows that $h^0(\Ff_u \otimes N^{\vee}) = i(\Ff_u) = 3$, since 
$u \in \widehat{\Lambda}_2$ general. Therefore, if $\Gamma \subset F_u$ denotes the section 
corresponding to $\Ff_u \to \!\! \to \omega_C$,  $(\pi_{d,2g-2}|_{V_d^{2g-2}})^{-1}([\Ff_u])$ contains 
a $\Pp^2$ isomorphic to $|\Oc_{F_u}(\Gamma)|$ (cf. \eqref{eq:isom2} and Lemma \ref{lem:ovviolin}).


 The case $u \neq u'$ cannot occur. Indeed, for any inclusion $\iota_1$ as above, there exist an inclusion $\iota_2$ and a $\lambda = \lambda (\iota_1, \iota_2) \in \C^*$ such that $\varphi \circ \iota_1 = \lambda \iota_2$, otherwise we would have $\dim(|\Oc_{F_u}(\Gamma)|) > 2$, a contradiction. One concludes by Lemma \ref{lem:technical}. 


In conclusion, the general fibre of $\pi_{d,2g-2}|_{V_d^{2g-2}}$ has dimension two (actually, this fibre is a $\Pp^2$).


\noindent
{\bf Step 4}. In this step we prove that $\overline{\mathcal V_d^{2g-2}}= \overline{\mathcal V_d^{2g-3}} = \overline{\mathcal V_d^{2g-4}} := \overline{\mathcal V}$. In particular, the presentation of $[\Ff]\in \overline{\mathcal V}$ general will be specially rigid only for $\delta = 2g-4$.

From Step 2 one has $\dim(V_d^{\delta}) = \rho_d^{k_3} + \delta - 2g+4$, for $2g-4 \leqslant \delta \leqslant 2g-2$. 
Moreover, the general element of $V_d^{\delta}$ can be identified with a pair $(F,\Gamma)$, where $F = \Pp(\Ff)$, $\Gamma \subset F$ a section corresponding to $\Ff \to\!\!\to \omega_C(-D)$, where 
$D \in C^{(2g-2-\delta)}$ and, for $\delta = 2g-2$, one has $D = 0$ and $\dim(|\Oc_F(\Gamma) |)=2$.

We will now prove that  there exist dominant, rational maps: 


\noindent
$(a)$ $r_1: V^{2g-2}_d \times C \dasharrow V^{2g-3}_d$, such that $r_1((F,\Gamma), p) = (F,\Gamma_p)$, where $\Gamma_p \subset F$ is a section corresponding to $\Ff \to\!\!\to \omega_C(-p)$,


\noindent
$(b)$ $r_2: \widetilde{V}^{2g-2}_d  \dasharrow V^{2g-4}_d$, where $ \widetilde{V}^{2g-2}_d$ is a finite cover $\varphi: \widetilde{V}^{2g-2}_d \to V^{2g-2}_d$ endowed with a rational map 
$\psi:\widetilde{V}^{2g-2}_d \dasharrow C^{(2)}$:  if $\xi \in \widetilde{V}^{2g-2}_d$ is general and $\varphi(\xi) := (F,\Gamma)$, then $r_2(\xi) = (F, \Gamma')$, with 
$\Gamma'$ a section corresponding to $\Ff \to\!\!\to \omega_C(-\psi(\xi))$.


\noindent
The existence of these maps clearly proves that $\overline{\mathcal V_d^{2g-2}}= \overline{\mathcal V_d^{2g-3}} = \overline{\mathcal V_d^{2g-4}}$.


\noindent
{$(a)$} Take  $(F,\Gamma)$  general in  $V^{2g-2}_d$ and $p \in C$ general. Then, the restriction map 
$$ \C^3 \cong H^0(\Oc_F(\Gamma)) \to H^0(\Oc_{f_p} (\Gamma)) \cong \C^2$$is surjective, because the general member of 
$|\Oc_F(\Gamma) |$ is irreducible. Hence there is a unique $\Gamma_p \in | \Oc_F(\Gamma - f_p)|$. 

We claim that $\Gamma_p$ is irreducible, i.e. it is a section. If not, $\Gamma_p$ would be a section plus a number $n \geqslant 1$ of fibres. As we saw, $n \leqslant 1$ (cf. Step 1) so $n=1$. 
This determines an automorphism of $C$ and, since $C$ has general moduli, this automorphism must be the identity. This is impossible because 
the map $\Phi_{\Gamma} : F \dasharrow \Pp^2$, given by $|\Oc_F(\Gamma)|$, is dominant hence it is ramified only in codimension one.

In conclusion, $\Gamma_p$ corresponds to $\Ff \to\!\! \to \omega_C(-p)$ and $(F, \Gamma_p)$ belongs to $V_d^{2g-3}$, and this defines $r_1$. 
The proof that $ (F, \Gamma_p)$ belongs to $V_d^{2g-3}$ is postponed for a moment (cf. Claim \ref{cl:maronna}).


\noindent
{$(b)$}  Given $(F, \Gamma)$ general in  $V^{2g-2}_d$, we can consider the map $\Phi_{\Gamma}$ as in Case (a). 
Since $\Phi_{\Gamma}$ maps the rulings of $F$ to lines, it determines a morphism $\Psi : C \to  C' \subset (\Pp^2)^{\vee}$. From Step 1, no (scheme-theoretical) 
fibre of $\Psi$ can have length bigger than two. Therefore, since $C$ has general moduli, $\Psi : C \to C'$ is birational and moreover, since $g \geqslant 8$, 
$C'$ has a certain number $n$ of double points, corresponding to  curves of type $\Gamma_D + f_D$, with $D \in C^{(2)}$ fibre of $\Psi$ over a double point of $C'$. 

Then the general point $\xi$ of $\widetilde{V}_d^{2g-2}$ corresponds to a triple $(F, \Gamma, D)$ (with $D \in C^{(2)}$ as above), the pair $(F, \Gamma_D)$ belongs to $V_d^{2g-4}$ (cf. 
Claim \ref{cl:maronna})  and $r_2(F, \Gamma, D) = (F,\Gamma_D)$, $\psi(F, \Gamma, D)= D$. 

\begin{claim}\label{cl:maronna} With the above notation,  $ (F, \Gamma_p)$ belongs to $V_d^{2g-3}$ and $(F, \Gamma_D)$ belongs to $V_d^{2g-4}$. 

\end{claim}
\begin{proof}[Proof of Claim \ref{cl:maronna}] We prove the claim for $(F, \Gamma_D)$, since the proof is similar in the other case. Take $(F,\Gamma)$  general in  $V^{2g-2}_d$; this determines 
a sequence 
\begin{equation}\label{eq:maronna}
0 \to N \to \Ff \to \omega_C \to 0,
\end{equation}where $N$ is general of degree $d-2g+2$ and the corresponding extension is general in the unique (good) component $\widehat{\Lambda}_2 \subset 
\Pp({\rm Ext}^1(\omega_C,N))$ dominating $\mathbb{G}(2, H^0(K_C-N))$ (cf. the proof of Step 2); thus,  if $\partial$ is the coboundary map, then 
${\rm Coker} (\partial) $ is a general two-dimensional quotient of $H^1(N)$. 

On the other hand, $(F, \Gamma_D)$ determines a sequence $$0 \to N(D) \to \Ff \to \omega_C(-D) \to 0.$$ 
Since $\deg(N(D)) = d-2g+4 \leqslant \frac{g}{2} - 2 < g-1$ and $N(D)$ is general of its degree, then $h^0(N(D))=0$. In view of 
$$0 \to N \to N(D) \to \Oc_D \to 0,$$one has 
the exact sequence 
$$0 \to H^0(\Oc_D) \cong \C^2 \to H^1(N) \stackrel{\alpha}{\longrightarrow} H^1(N(D)) \to 0.$$

The existence of the unisecant $\Gamma_D + f_D$ on $F$ gives rise to the sequence 
\begin{equation}\label{eq:maronna2}
0 \to N \to \Ff \to \omega_C(-D) \oplus \Oc_D \to 0
\end{equation}(cf. \eqref{eq:Fund2}). This sequence corresponds to an element $\xi \in {\rm Ext}^1( \omega_C(-D) \oplus \Oc_D, N)$, which by Serre duality, is isomorphic to $H^0(\Oc_D)^{\vee} \oplus {\rm Ext}^1 (\omega_C(-D), N)$ (cf. \cite[Prop. III.6.7, Thm. III.7.6]{Ha}).  So $\xi = 
(\sigma, \eta)$, with $\sigma \in  H^0(\Oc_D)^{\vee}$and $\eta \in {\rm Ext}^1 (\omega_C(-D), N) \cong H^1(N(D) \otimes \omega_C^{\vee})$. 

We have the following diagram
\[
\begin{array}{ccc}
H^0(\Oc_D) \oplus H^0(\omega_C (-D)) & \stackrel{\partial_0}{\longrightarrow} & H^1(N)  \\
\uparrow & & \;\; \downarrow^{\alpha} \\
H^0(\omega_C(-D)) & \stackrel{\partial'}{\longrightarrow} & H^1(N(D)) \\
\uparrow & & \downarrow \\
0 & & 0
\end{array}
\]where $\partial_0$, $\partial'$ are the coboundary maps. The action of $\xi$ on $\{0\} \oplus H^0(\omega_C (-D))$ coincides with the action of $\eta$ on $H^0(\omega_C (-D))$ via cup-product. This yields an isomorphism 
 ${\rm Coker}(\partial') \stackrel{\cong}{\longrightarrow} {\rm Coker}(\partial_0)$.

Notice that \eqref{eq:maronna2} can be seen as a limit of \eqref{eq:maronna}. Since ${\rm Coker} (\partial)$ is a general two-dimensional quotient of $H^1(N)$, then also ${\rm Coker} (\partial_0)$ is general. The above argument 
implies that ${\rm Coker} (\partial')$ is also general, proving the assertion (cf. the proof of Claim \ref{cl:good}). \end{proof}

Finally, to prove that $r_1$, $r_2$ are dominant, it suffices to prove the following: 

\begin{claim}\label{cl:ri} The general fibre of $r_i$ has dimension two, for $1 \leqslant i \leqslant 2$.

\end{claim} \begin{proof}[Proof of Claim \ref{cl:ri}] It suffices to prove that there are fibres of dimension two. For $r_1$,  take $(F, \Gamma, p) $ general in $V_{d}^{2g-2} \times C$. 
The fibre of $r_1$ containing this triple consists of all triples $(F, \Gamma', p)$, with $\Gamma' \in |\Oc_{F} (\Gamma)|$ so it has dimension 
two since $i(\Ff) =3$ (cf. computation as in \eqref{eq:casaciro2}). The same argument works for $r_2$. \end{proof}


\noindent
{\bf Step 5}. In this step we prove that $\overline{\mathcal V}$ is an irreducible component of $\widetilde{B_C^{k_3}}(d)$.

\begin{claim}\label{cl:tutte} Let $(F, \Gamma_D) \in V_d^{2g-2-i}$ be general, with $1 \leqslant i \leqslant 2$ and $D \in C^{(i)}$. Then $|\Gamma_D + f_D|$ has dimension two, its general member $\Gamma$ is smooth and it corresponds to a sequence $0 \to N (-D) \to \Ff \to \omega_C \to 0$. Consequently, the pair $(F, \Gamma_D)$ is in the image of $r_i$. 

\end{claim}

\begin{proof}[Proof of Claim \ref{cl:tutte}]  Given the first part of the statement, the conclusion is clear. 
To prove the first part, note that the existence of $\Gamma \subset F$ gives an exact sequence 
\begin{equation}\label{eq:lala}
0 \to N \to \Ff \to \omega_C(-D) \to 0,
\end{equation}hence $h^0(\Oc_F(\Gamma_D + f_D)) = h^0(\Ff \otimes N^{\vee} (D)) = h^0(\Ff \otimes \omega_C \otimes \det(\Ff)^{\vee}) = h^1(\Ff) = 3$ (cf. \eqref{eq:isom2}). This implies the assertion. \end{proof}

Let now $\mathcal B \subseteq \widetilde{B_C^{k_3}}(d)$ be a component containing $\overline{\mathcal V}$. From Step 4, $[\Ff] \in \mathcal B$ general  has speciality $i=3$ and a special presentation as in \eqref{eq:Fund} with $L$ of minimal degree $\delta$.  
Thus $2g-4 \leqslant \delta \leqslant 2g-2$, since the Segre invariant is lower semi-continuous  (cf. Remark \ref{rem:seginv} and also 
\cite[\S\;3]{LN}).

By Claim \ref{cl:tutte}, $2g-3\leqslant \delta \leqslant 2g-2$ does not occur under the minimality assumption on $L$. Indeed, in both cases we have a two-dimensional linear system $|\Gamma|$, whose general member is a section, 
corresponding to a surjection $\Ff \to\!\! \to \omega_C$ and we proved that there would be curves in this linear system containing two rulings.

If $\delta = 2g-4$, we have an exact sequence as in \eqref{eq:lala}. By specializing to a general point of $\overline{\mathcal V} = \overline{\mathcal V^{2g-4,1,2}}$, because of Claim \ref{cl:tutte}, 
we see that in \eqref{eq:lala} one has $h^0(N) = 0$. Hence, $h^1(N)$ is constant. Since for the general element 
of  $\overline{\mathcal V}$, ${\rm Ker}(\mu_{V_2}) = (0)$ the same happens for the general element of $\mathcal B$ i.e., 
with notation as in the proof of Claim \ref{cl:good}, $\tau$ is constant equal to zero. Therefore, also $\eta = \epsilon = 0$ for the general point of $\mathcal B$ (see l.c.),  which implies the assertion. \end{proof}

With our approach,  we cannot conclude that $\overline{\mathcal V_{d}^{2g-4,1,2}} $ in Theorem \ref{i=3} is the unique regular component, whose general  point $[\Ff]$ is such that $i(\Ff)=3$, because 
we do not know if $\Lambda_2 \subset \mathcal W_2$ is the only good component when $N \in \Pic^{d-2g+4}(C)$ and $D_2 \in C^{(2)}$ general. However, results in  \cite{Tan,TB00} imply that $\widetilde{B_C^{k_3}}(d)$ is irreducible for $d \leqslant \frac{10}{3}g -7$, though they say nothing on rigid special presentation of the general element. Putting all together, we have:

\begin{corollary}\label{cor:i=3} Under the assumptions of Theorem \ref{i=3}, one has  $\widetilde{B_C^{k_3}}(d) = \overline{\mathcal V_{d}^{2g-4,1,2}}$. 
\end{corollary}

\begin{remark}\label{rem:i=3} (1) Theorem \ref{C.F.VdG}, for $j=3$, shows the existence of elements of $\widetilde{B_C^{k_3}}(d)$ with injective Petri map in the range 
 $g \geqslant 21$, $g+3 \leqslant \delta \leqslant \frac{4}{3} g - 4$,  $2g+6 \leqslant d \leqslant \frac{8}{3}g-9$. This gives a proof, alternative to the one in \cite{TB00}, of generic smoothness 
of  $\widetilde{B_C^{k_3}}(d)$ in the above range.

\noindent
(2) If $\frac{5}{2}g - 5 \leqslant d \leqslant \frac{10}{3}g-7 $, $\widetilde{B_C^{k_3}}(d)$ is, as we saw, irreducible but in general it is no longer true that it is determined by a (total) good component. To see this, 
we consider a specific example.

Take $\widetilde{B_C^{k_3}}(3g-4)$, which  is non-empty, irreducible, generically smooth, of dimension $g-6$ and 
$[\Ff] \in  \widetilde{B_C^{k_3}}(3g-4)$ general is such that $i(\Ff)=3$ by \cite{Tan,TB00}. 
By Lemma \ref{lem:1e2note}, $\Ff$ can be rigidly presented as in \eqref{eq:Fund}, where $L  \in W_{\delta}^{\delta-g+j}(C)$ and $1 \leqslant j \leqslant 3$. 

The cases $j=2,\;3$ cannot occur: the stability  condition \eqref{eq:assumptions}  imposes $\delta  > \frac{3}{2}g - 2$, but if $j=3$, $\rho(L) \geqslant 0$ forces $\delta \leqslant \frac{4}{3}g-3$ whereas if $j=2$, $\rho(L) \geqslant 0$ implies $\delta  \leqslant \frac{3}{2}g - 3$; in both cases we get a contradiction. 

The only possible case is therefore $j=1$, so the corank of the coboundary map is $t =2$, which implies that $N$ is of speciality $r \geqslant 2$. Since $\chi(N) = 2g-3 - \delta$, the case $N$ non-effective would give $\delta>2g-3$, i.e. $L \cong \omega_C$. 
But in this case, $a_{F} (2g-2) \geqslant 2$ (usual computations as in \eqref{eq:casaciro2}) against the rigidity assumptions.

Therefore $N$ must be effective, with $n = h^0(N) = 2g-3-\delta + r$. We want to show that the hypotheses of Corollary \ref{cor:mainext1} hold. Assume by contradiction $\ell <r$; then 
\begin{equation}\label{eq:zaid}
\delta < g-2+r.
\end{equation} From stability $3g-4 < 2\delta < 2g-4+2r$, i.e. $g-2r <0$. Since $C$ has general moduli, one has $\rho(N)\geqslant 0$, hence $h^0(N)=1$.  So $d-\delta = g-r$ and  \eqref{eq:zaid} yields 
$d = \delta + d-\delta < 2g-2$ a contradiction. Thus, $\ell \geqslant r$.

Now, from \eqref{eq:lem3note}, $m = 2 \delta - 2 g + 3$ since $N$ is not isomorphic to $L$. Thus, $m \geqslant \ell +1$: this is equivalent to $\delta \geqslant g$, which holds by stability. 

In conclusion, by Corollary \ref{cor:mainext1}, $\mathcal W_1^{\rm Tot}$ is irreducible, of the expected dimension.  Assume that $\widehat{\mathcal W}^{\rm Tot}_1$ contains a total good component $\widehat{\Lambda}_2^{\rm Tot}$, 
whose image via $\pi_{3g-4, \delta}$ is $\widetilde{B_C^{k_3}}(3g-4)$. Thus, $r \geqslant 2$. On the other hand $r=2$ cannot occur since $$c(\ell,2,2) = 2(\delta - g +2) > 2\delta - 2g +2 = m-1= \dim(\Pp({\rm Ext}^1(L,N))),$$a contradiction. Therefore, one has $r \geqslant 3$. 

From the second equality in \eqref{eq:yde2}$$\dim(\Pp(\Efra_\delta)|_{\zd}) = (r+1) \delta - (2r-1) g - r (r-3)$$and the codimension of $\widehat{\Lambda}_2^{\rm Tot}$ 
is $$c(\ell,r,2) =  2 (\delta - g + 4 - r).$$Set $a := a_{F_v}(\delta)$. From Remark \ref{rem:rigid} we can assume  $a \leqslant 1$. Therefore,$$\dim({\rm Im}(\pi_{d, \delta}|_{\widehat{\Lambda}_2^{\rm Tot}})) = g-6$$gives $(r-1) \delta = 2r g - 2g + r^2 - 5 r + 2 + a$, i.e. 
$$\delta = 2g+r-4 + \frac{a-2}{r-1}.$$This yields a contradiction. Indeed, since $0 \leqslant a \leqslant 1$, $r \geqslant 3$ and $\delta$ is an integer, the only possibility is $r=3$, $a=0$, $\delta = 2g-2$ which we already saw to contradict 
the rigidity assumption. 
\end{remark}


\begin{theorem}\label{prop:M32K} Let $C$ be of genus $g \geqslant 4$, with general moduli. Then, 
$\widetilde{M_C^3}(2, \omega_C) \neq \emptyset$. Moreover, there exists an irreducible component which is regular (i.e. of dimension $3g-9$), whose general point $[\Ff]$ fits in a sequence
\begin{equation}\label{eq:presott}
0 \to \Oc_C(p+q) \to \Ff \to \omega_C (-p-q) \to 0,
\end{equation}where 

\begin{itemize}
\item[$\bullet$] $p +q \in C^{(2)}$ general, and 
\item[$\bullet$] $\Ff=\Ff_v$ with $v \in \Lambda \subset \mathcal W_2  \subset \Ext^1(\omega_C(-p), \Oc_C(p))$ general in $\Lambda$, 
which is a component of $\mathcal W_2$ of dimension $3g-10$ (hence, not good).    
\end{itemize}
\end{theorem} 

\begin{proof} With notation as in \eqref{eq:exthyp1}, \eqref{eq:exthyp}, for $\Ff=\Ff_v$ as in \eqref {eq:presott}, we have $$\ell = r = g-2, \; t = 2, \;  m = h^1(2p + 2q - K_C) = 3g - 7.$$Consider the map (notation as in   \eqref{eq:mu} and \eqref {eq:muW}) 
$$\mu: H^0(\omega_C(-p-q)) \otimes H^0(\omega_C(-p-q)) \to H^0(\omega_C^{\otimes 2} (-2p-2q)).$$
For $V_2 \in \mathbb{G}(2, H^0(\omega_C(-p-q)))$ general, $\mu_{V_2}$ has kernel of dimension $1$ (cf. computations as in  Claim \ref{cl:good}).  
Arguing as in the proofs of Theorem \ref{thm:mainextt} and Claim \ref{cl:good}, there is a component $\Lambda \subset \mathcal W_2  \subset \Ext^1(\omega_C(-p-q), \Oc_C(p+q))$ (dominating $\mathbb{G} (2, H^0(\omega_C(-p-q))$, hence not good) of dimension $3g-10$. 

Stability of $\Ff$ follows from Proposition \ref{prop:LN}. This shows that $\widetilde{M_C^3}(2, \omega_C) \neq \emptyset$. Regularity and generic smoothness follow from the injectivity of the symmetric Petri map as in \cite{TB000,Beau2}. 

The fact that $[\Ff]$ general has a presentation as in \eqref{eq:presott} follows from an obvious parameter computation. 
\end{proof}

By \cite[\S\,4.3]{IVG}, $\widetilde{M_C^3}(2, \omega_C)$ contains a unique regular component; Theorem \ref{prop:M32K} provides in addition 
a rigidly special presentation of its general element. 



\subsection{A conjecture for $i \geqslant 4$}\label{sec:high} For any integers $i \geqslant 4$ and $d$ as in Theorems \ref{LN}, \ref{C.F.VdG}, \ref{uepi}, \ref{unepi}, one has $\widetilde{B_C^{k_i}}(d) \neq \emptyset$. In particular, when $d$ is as in Theorem \ref{C.F.VdG}, with $j=i$, one deduces that $\widetilde{B_C^{k_i}}(d)$ contains a regular, generically smooth component. This gives existence results in the same flavour as Theorem \ref{thm:TB}. 

One may  wish to give a special, rigid presentation of the general point of all components of $\widetilde{B_C^{k_i}}(d)$.  The following less ambitious conjecture is inspired by the results in this paper.

\begin{conjecture}\label{igeq4} {\em  Let $i \geqslant 4$ and  $g > i^2 - i + 1$ be integers. Let $C$ 
be of genus $g$, with general moduli. Let $d$ be an integer such that 
$$2g-2 \leqslant d < 2g-2 - i + \frac{g-\epsilon}{i-1},$$with $\epsilon \in \{0,1\}$ such that  $d + g - (i-1) k_i \equiv \epsilon \pmod{2}$. Then, there exists an irreducible, regular component $\mathcal B \subseteq \widetilde{B_C^{k_i}}(d)$, s.t. 
$$\mathcal B = \overline{\mathcal V_{d}^{2g-1-i,1,i-1}}.$$In particular, 
$[\Ff] \in \mathcal B$ general is stable, with $i(\Ff)=i$, $s(\Ff) \geqslant g - (i-1) k_i - \epsilon >0$ and it is rigidly specially presented as 
$$0 \to N \to \Ff \to \omega_C (-D_{i-1}) \to 0,$$where

\begin{itemize}
\item[$\bullet$] $D \in C^{(i-1)}$ is general, 
\item[$\bullet$] $N \in \Pic^{d-2g+1+i}(C)$ is  general (i.e. special, non-effective), 
\item[$\bullet$] $\Ff=\Ff_v$ with $v \in \Lambda_{i-1} \subset \Ext^1(N, \omega_C(-D))$ general in a good component; 
\end{itemize}

}
\end{conjecture}


\noindent
The bounds on $g$ and $d$ in Conjecture \ref{igeq4} ensure the following: 

\begin{itemize}
\item[(1)] $\rho_d^{k_i} \geqslant 0$ which is equivalent to $ d \leqslant 2g-2 - i + \frac{4g-3}{i}$ (cf. \eqref{eq:bn}). 
\item[(2)] $N \in \Pic^{d-2g+1+i}(C)$ general is special, non-effective; indeed $r= 3g-2-i -d >0$. 
\item[(3)] $r \geqslant i-1 = \cork(\partial_v)$, which is equivalent to  $d \leqslant   3g-1-2i$. 
\item[(4)] There are no obstructions for a good component $\Lambda_{i-1} \subset \Ext^1(\omega_C(-D),N)$ to exist; indeed, one has $$\dim(\Pp) = m-1 = 5g-4 - 2i - d$$ from \eqref{eq:lem3note}, and from \eqref{eq:clrt}, Remark \ref{unepine}, we have 
$$c(\ell,r,t) =  (i-1) k_i = (i-1) (d-2g+2+i),$$since $t =i-1$ and $N$ non-effective. Therefore 
$$\dim(\Lambda_{i-1}) = m-1 - c (\ell,r,t) = 3g-2 - i - i (d-2g+2+i)$$is non-negative as soon as 
$d \leqslant 2g-3-i + \frac{3g-2}{i}$. 
\item[(5)] From Remark \ref{unepine}, $v \in \Lambda_{i-1}$ general is such that 
$s(\Ff_v) \geqslant g - (i-1) k_i - \epsilon$, which is positive because of the upper-bound on $d$. Thus 
$\Ff_v$ is stable.
\item[(6)] The interval $ 2g-2 \leqslant d < 2g-2 - i + \frac{g-\epsilon}{i-1}$ is not empty. 
\end{itemize}


From Remark \ref{rem:rigid} and \eqref{eq:fojt}, to prove 
Conjecture \ref{igeq4} it would suffice to prove the following two facts:


\begin{itemize}
\item[(a)] {\em For $i \geqslant 4$, $D \in C^{(i-1)}$ general and $L = \omega_C(- D)$, there exists a good component 
$$\Lambda_{i-1} \subseteq \mathcal W_{i-1} \subset {\rm Ext}^1(\omega_C(- D), N).$$ }

\item[(b)] {\em For $v \in \Lambda_{i-1}$ general, $\Ff_v$ is rsp via $\omega_C(-D)$. }

\end{itemize}


Concerning (a), notice that for no $V_{i-1} \in \mathbb{G}(i-1, H^0(K_C-N))$ the map $\mu_{V_{i-1}}$ can be injective. Indeed, $d\geqslant 2g-2$ and $i \geqslant 4$ imply
$$\dim(H^0(K_C-D) \otimes V_{i-1}) = (i-1) g - (i-1)^2>5g-3-d-2i=h^0(2K_C - N - D).$$
Hence, according to Theorem \ref{thm:mainextt}, one should find an irreducible subvariety 
$\Sigma_{\eta} \subset \mathbb{G}(i-1, H^0(K_C - D))$  of codimension $\eta:= d + (i-6) g + 7 - (i-2)^2$ such that $\dim({\rm Ker}(\mu_{V_{i-1}})) = \eta$ for $V_{i-1} \in \Sigma_{\eta}$ general.

Concerning (b), the minimality assumption implies $a_{F_v} (2g-1-i) \leqslant 1$, for $v \in \Lambda_{i-1}$ general and $F_v = \Pp(\Ff_v)$. 
To prove rigidity, one has to show that $a_{F_v} (2g-1-i) = 0$. This is equivalent to prove a regularity statement for a Severi variety of nodal curves on $F_v$. Indeed, for any section $\Gamma_{D}$ corresponding to a quotient  $\Ff_v \to\!\! \to \omega_C(-D)$ as above, 
the linear system $|\Gamma_D + f_D|$ has dimension $i-1$, it is independent on $D$ and its general member $\Gamma$ is a section corresponding to a quotient $\Ff_v \to \!\! \to \omega_C$. The curve $\Gamma_D + f_D$ belongs to the Severi variety of $(i-1)$-nodal curves in $|\Gamma|$. So rigidity is equivalent to show that this Severi variety has the expected dimension zero. Proving this is equivalent to prove that $D$, considered as a divisor on $\Gamma_D$, imposes independent condition to $|\Gamma|$. Unfortunately, the known results on regularity of Severi varieties (see  \cite{N,Tanb,Tanb1}) do not apply in this situation.



\begin{thebibliography}{}


\bibitem{ACGH} E.~Arbarello, M.~Cornalba, P.A.~Griffiths, J.~Harris,  {\em Geometry of algebraic curves, Vol. I}.
Grundlehren der Mathematischen Wissenschaften, {\bf 267}.
Springer-Verlag, New York, 1985.

\bibitem{Ballico} E.~Ballico,
Brill-Noether theory for vector bundles on projective curves, {\em Math. Proc. Camb. Phil. Soc.} {\bf 124} (1998), 483--499.

\bibitem{Ballico1} E.~Ballico,
Special stable vector bundles on smooth curves, {\em International Mathematical Forum} {\bf 3} (2008), 1121--1130.


\bibitem{Beau} A.~Beauville, Vector bundles on curves and generalized theta functions: recent results and open problems, Current topics in complex algebraic geometry (Berkeley, CA, 1992/93), {\em Math. Sci. Res. Inst. Publ.}, {\bf 28} (1995), Cambridge Univ. Press, Cambridge, 17-Ð33.


\bibitem{Beau2} A.~Beauville, Remark on a conjecture of Mukai, {\em arxiv:math/0610516v1[mathAG]17 Oct 2006},  1--6. 


\bibitem{Be} A.~Bertram, Moduli of rank-2 vector bundles, theta divisors and the geometry of curves in projective space, in {\em J. Differentia Geometry}, {\bf 35} (1992), 429--469.

\bibitem{BeFe} A.~Bertram, B.~Feinberg, On stable rank-two bundles with canonical determinant and many sections, in {\em Algebraic Geometry (Catania 1993/Barcelona 1994)}, Lecture Notes in Pure and Pallied  Mathematics, {\bf 200} (1998), Dekker, 259--269.




\bibitem{BGN} L.~Brambila-Paz, I.~Grzegorczyk, P.E.~Newstead, Geography of Brill-Noether loci for small slopes,  
{\em J. Algebraic Geom.}, {\bf 6} (1997), no. 4, 645-669. 


\bibitem{BMNO} L.~Brambila-Paz, V.~Mercat, P.E.~Newstead, F.~Ongay, Nonemptiness of Brill-Noether loci, 
{\em Internat. J. Math.}, {\bf 11} (2000), no. 6, 737-760. 


\bibitem{CCFMnonsp} A.~Calabri, C.~Ciliberto, F.~Flamini, R.~Miranda,
Non-special scrolls with general moduli, {\em Rend. Circ. Mat. Palermo} {\bf 57} (2008), no. 1, 1--32.


\bibitem{CCFMsp} A.~Calabri, C.~Ciliberto, F.~Flamini, R.~Miranda,
Special scrolls whose base curve has general moduli, {\em Contemporary Mathematics} {\bf 496} (2009), 133--155.

\bibitem{CCFMBN} A.~Calabri, C.~Ciliberto, F.~Flamini, R.~Miranda, Brill-Noether theory and non-special scrolls with general moduli, {\em Geom. Dedicata}, {\bf 139} (2009), 121-138.


\bibitem{CMT} A.~Castorena, A.L.~Martin, M.~Teixidor I Bigas,
Petri map for vector bundles near good bundles, {\em arxiv:1203.0983v1 [math.AG] 5 Mar 2012}, 1--10.



\bibitem{CF} C.~Ciliberto, F.~Flamini, Brill-Noether loci of stable Rank-two vector bundles on a general curve, 
to be published in the {\em Proceedings of the conference "Geometry and Arithmetic: a conference on the occasion of  
Gerard van der Geer's 60th birthday}, Schiermonnikoog, September 20-24  2010, (2012), 1--19.





\bibitem{FO} G.~Farkas, A.~Ortega, The maximal rank conjecture and rank two Brill-Noether theory,
{\em Pure and Applied Mathematics Quarterly}, {\bf 7} (2011), 1265--1296.


\bibitem{Frie} R.~Friedman, {\em Algebraic surfaces and holomorphic vector bundles}, Universitext.
Springer-Verlag, New York, 1998.




\bibitem{GP2} L.~Fuentes-Garcia, M.~Pedreira, The projective theory of ruled surfaces,
{\em Note Mat.} {\bf 24} (2005), no. 1, 25--63.

\bibitem{GP3} L.~Fuentes-Garcia, M.~Pedreira, The general special scroll of genus
$g$ in $\Pp^N$. Special scrolls in $\Pp^3$, {\em math.AG/0609548
20Sep2006} (2006), pp. 13.


\bibitem{Ghio} F.~Ghione, Quelques r\'esultats de Corrado Segre sur les surfaces r\'egl\'ees, 
{\em Math. Ann.}, {\bf 255} (1981), no. 3, 77--95.


\bibitem{GT} I.~Grzegorczyk, M.~Teixidor I Bigas, Brill-Noether theory for stable vector bundles.
{\em Moduli spaces and vector bundles}, 29-50, {\em London Math. Soc. Lecture Note Ser.}, {\bf 359}, Cambridge Univ. Press, Cambridge, 2009. 


\bibitem{Ha} R.~Hartshorne, {\em Algebraic Geometry} (GTM No. {\bf 52}), Springer-Verlag, New York - Heidelberg, 1977.


\bibitem{HR} J.M.~Hwang, S.~Ramanan, Hecke curves and Hitchin discriminant,
{\em Annales Sci. \'Ecole Normale Sup\'erieure}, {\bf 37} (2004), 801--817.




\bibitem{IVG} E.~Izadi, B.~van Geemen, The tangent space to the moduli space of vector bundles on a curve and the singular locus of the theta divisor of the Jacobian, {\em J. Algebraic Geom.}, {\bf 10} (2001), no. 1, 133--177.


\bibitem{LN} H.~Lange, M.S.~Narashiman, Maximal subbundles of rank-two vector bundles on curves,
{\em Mat. Ann.}, {\bf 266} (1983), 55--72.


\bibitem{LaNw} H.~Lange, P.E.~Newstead, V.~Strehl, Non-emptiness of Brill-Noether loci in $M(2,L)$, {\em arxiv:1312.1844v2 [math.AG] 28 Jan 2014}, 1--22.


\bibitem{LNP} H.~Lange, P.E.~Newstead, S.S.~Park, Non-emptiness of Brill-Noether loci in $M(2,K)$, {\em arxiv:1311.5007v1 [math.AG] 20 Nov 2013}, 1--21.

\bibitem{Lau} G.~Laumon, Fibres vectoriels speciaux,
{\em Bull. Soc. Math. France}, {\bf 119} (1990), 97--119.


\bibitem{Ma3} M.~Maruyama, On automorphism groups of ruled surfaces.
{\em J. Math. Kyoto Univ.},  {\bf 11-1} (1971), 89--112. 



\bibitem{M1} V.~Mercat, Le probl$\grave{e}$me de Brill-Noether pur des fibr$\acute{e}$s stables de petite pente,  
{\em J. reine. angew. Math}, {\bf 56} (1999), 1--41.


\bibitem{M} V.~Mercat, Le probl$\grave{e}$me de Brill-Noether: pr\'esentation. $http://www.liv.ac.uk/ \sim newstead/bnt.html$ (2001). 

\bibitem{M4} V.~Mercat, Le probl$\grave{e}$me de Brill-Noether: 
{\em Bull. London Math. Soc.},  {\bf 33} (2001),  no. 5, 535-542.
 
\bibitem{Muk2} S.~Mukai, Vector bundles and Brill-Noether theory, in {\em Current topics in complex algebraic geometry}, 145--158; Math.Sci. Res. Inst. Publ.  {\bf 28} (1995), Cambridge Univ. Press. 

\bibitem{Muk} S.~Mukai, Non-Abelian Brill-Noether theory and Fano 3-folds,
{\em Sugaku Expositions}, {\bf 14} (2001), no. 2, 125-153. [translation of {\em Sugaku}, {\bf 49} (1997), no. 1, 1-24]


\bibitem{Na} M.~Nagata, On the self-intersection number of a section on a ruled surface.
{\em Nagoya Math. J.}, {\bf 37} (1970), 191--196 . 

\bibitem{NR1}  M.S.~Narashiman, S.~Ramanan, Deformation of the moduli of vector bundles, {\em Ann. Math.}, 
{\bf 101} (1975), 391--417 . 


\bibitem{New} P.E.~Newstead, {\em Introduction to moduli problems and orbit spaces},
Tata Institute of Fundamental Research Lectures on Mathematics and
Physics, {\bf 51}, Narosa Publishing House, New Delhi, 1978.


\bibitem{New2} P.E.~Newstead, Vector bundles on algebraic curves. Lectures during 
the 25th Autumn School in Algebraic Geometry {\em Vector bundles over curves and higher dimensional varieties, and their moduli}, Lukecin, Poland; 8--15 September 2002 (available at $http://www.mimuw.edu.pl/\sim jarekw/postscript/Lukecin-Newstead.ps$). 

\bibitem{N} Nobile A., Families of curves on surfaces, {\em Math. Zeitschrift}, {\bf 187} (1984), 453-470.


\bibitem{Oss} B.~Osserman, Brill-Noether loci with fixed determinant in rank 2, {\em arxiv:1005.0448v2 [math.AG] 24 Aug 2011}, 1--20.


\bibitem{Oss2} B.~Osserman, Special determinants in higher rank Brill-Noether theory, {\em Int. J. Math.}, {\bf 24}  (2013), no. 11, 1--20.

\bibitem{Ram} S.~Ramanan, The moduli spaces of vector bundles over an algebraic curve,
{\em Math. Ann.}, {\bf 200} (1973), 68--84.

\bibitem{Seg} C.~Segre, Recherches g\'en\'erales sur les courbes et les surfaces r\'egl\'ees alg\'ebriques,
{\em OPERE - a cura dell'Unione Matematica Italiana e col
contributo del Consiglio Nazionale delle Ricerche}, {\bf vol. 1},
\S\,XI - pp. 125-151, Edizioni Cremonese, Roma 1957 (cf. {\em
Math. Ann.} {\bf 34} (1889), 1--25).

\bibitem{Ser} E.~Sernesi, {\it Deformations of Algebraic Schemes}, Grundlehren der mathematischen Wissenschaften {\bf 334}, Springer-Verlag, Berlin, 2006.

\bibitem{Ses} C.S.~Seshadri, {\em Fibr\'es vectoriels sur les courbes alg\'ebriques}, Ast\'erisque, {\bf 96}. S.M.F., Paris, 1982.


\bibitem{Sun} N.~Sundaram, Special divisors and vector bundles, {\em T\^ohoku Math. J.}, {\bf 39} (1987), 175--213.


\bibitem{Tha} M.~Thaddeus, An introduction to the topology of the moduli space of stable bundles on a Riemann surface,
{\em Geometry and physics (Aarhus, 1995)}, 71-99, Lecture Notes in Pure and Appl. Math., 
{\bf 184}, Dekker, New York, 1997


\bibitem{Tan} X.J.~Tan, Some results on the existence of rank 2 special stable vector bundles, {\em Man. Math.}, {\bf 75} (1992), 365--373.

\bibitem{Tanb} A .~Tannenbaum A., Families of algebraic curves with nodes, 
{\em Comp. Math.}, {\bf 41}, (1980), 107-126.

\bibitem{Tanb1} A .~Tannenbaum, Families of curves with nodes on $K3$ surfaces, 
{\em Math. Ann.}, {\bf 260} (1982), 239-253.

\bibitem{TB0} M.~Teixidor I Bigas, Brill-Noether theory for vector bundles of rank $2$,
{\em T\^ohoku Math. J.}, {\bf 43} (1991), 123--126.

\bibitem{TB00} M.~Teixidor I Bigas, On the Gieseker-Petri map for rank 2 vector bundles,
{\em Man. Math.}, {\bf 75} (1992), 375--382.


\bibitem{TB1} M.~Teixidor I Bigas, Rank two vector bundles with canonical determinant,
{\em Math. Nachr.}, {\bf 265} (2004), 100--106. 

\bibitem{TB} M.~Teixidor I Bigas, Existence of  vector bundles of rank-two with sections,
{\em Adv. Geom.}, {\bf 5} (2005), 37--47.


\bibitem{TB000} M.~Teixidor I Bigas, Petri map for rank two bundles with canonical determinant,
{\em Compos. Math.}, {\bf 144} (2008), no. 3, 705--720. 



\bibitem{Voi} C.~Voisin, Sur l'application de Wahl des courbes satisfaisant le condition de Brill-Noether-Petri, {\em Acta Math.}, {\bf 168} (1992), 249--272.

\end{thebibliography}
\end{document}